\documentclass[a4paper]{amsart}
\usepackage[utf8]{inputenc}
\usepackage{amscd,amsthm,amsfonts,amssymb,esint}
\usepackage{amsmath}
\usepackage[colorlinks=true]{hyperref}
\usepackage{fullpage}
\usepackage[all]{xy}
\usepackage{mathdots}

    \newcommand{\BC}{{\mathbb {C}}}

    \newcommand{\BQ}{{\mathbb {Q}}} \newcommand{\BR}{{\mathbb {R}}}

     \newcommand{\BZ}{{\mathbb {Z}}}

    \newcommand{\CC}{{\mathcal {C}}}

    \newcommand{\CM}{{\mathcal {M}}} \newcommand{\CN}{{\mathcal {N}}}
    \newcommand{\CO}{{\mathcal {O}}} 
     \newcommand{\CR}{{\mathcal {R}}}
    \newcommand{\CS}{{\mathcal {S}}} 
     
    \newcommand{\CW}{{\mathcal {W}}}

     \newcommand{\fp}{{\mathfrak{p}}}

     \newcommand{\fz}{{\mathfrak{z}}}

    \newcommand{\Ad}{{\mathrm{Ad}}}

    \newcommand{\der}{{\mathrm{der}}}
    \newcommand{\diag}{{\mathrm{diag}}}

     \newcommand{\GL}{{\mathrm{GL}}}
    
    \newcommand{\Hom}{{\mathrm{Hom}}}
    
    \newcommand{\id}{{\mathrm{id}}}\renewcommand{\Im}{{\mathrm{Im}}}

    \newcommand{\Ker}{{\mathrm{Ker}}}

    \newcommand{\reg}{{\mathrm{reg}}}

    \newcommand{\SL}{{\mathrm{SL}}}
     
    \newcommand{\Sp}{{\mathrm{Sp}}}

    \newcommand{\Stab}{{\mathrm{Stab}}}

    \newcommand{\Vol}{{\mathrm{Vol}}}

\newcommand{\matrixx}[4]{\begin{pmatrix}
#1 & #2 \\ #3 & #4
\end{pmatrix} }        

    \newcommand{\wt}{\widetilde}
    \newcommand{\wh}{\widehat}

    \newcommand{\ov}{\overline}

    \newcommand{\sk}{\medskip}
    \newcommand{\lra}{\longrightarrow}
    \newcommand{\ra}{\rightarrow} 
    \newcommand{\bs}{\backslash}
    
    \newcommand{\s}{\sk\noindent}

    \theoremstyle{plain}
    \newtheorem{thm}{Theorem}[section] \newtheorem{cor}[thm]{Corollary}
    \newtheorem{lem}[thm]{Lemma}  \newtheorem{prop}[thm]{Proposition}
     \newtheorem{defn}[thm]{Definition}

\theoremstyle{remark} \newtheorem{remark}[thm]{Remark}
\theoremstyle{remark} 
\theoremstyle{remark} \newtheorem{example}[thm]{Example}

    \newcommand{\etale}{\'{e}tale~}

    \numberwithin{equation}{section}

     \renewcommand{\Ad}{ {\mathrm{Ad}}}
     \newcommand{\rel}{ {\mathrm{rel}}}

     \newcommand{\Kl}{{\mathrm{Kl}}}

\begin{document}

\title{Bessel Distributions and Kloosterman Sums}

\author{Li Cai} 
\address[L.C.]{Academy for Multidisciplinary Studies, Beijing National Center for Applied Mathematics, Capital Normal University, Beijing, 100048, People's Republic of China
} 
\email{caili@cnu.edu.cn}

\author{Jingsong Chai} 
\address[J.C.]{School of Mathematics, Physics and Finance, Anhui Polytechnic University, Wuhu, Anhui, 241000, People's Republic of China
} 
\email{jingsongchai@hotmail.com}

\author{Yadi Liu} 
\address[Y.L.]{Academy for Multidisciplinary Studies, Beijing National Center for Applied Mathematics, Capital Normal University, Beijing, 100048, People's Republic of China
} 
\email{2220502124@cnu.edu.cn}

\begin{abstract}
Let $G$ be a split reductive group over a $p$-adic field. 
In this paper, we prove that Bessel distributions are regular for all generic representations on $G$ provided that certain Kloosterman sums for any Levi subgroups of $G$ have nontrivial bounds. For this, we give germ expansions of Kloosterman integrals for $G$. As an application, the regularity is proved
for groups of type $A$.
\end{abstract}



\maketitle
\tableofcontents

\section{Introduction}

\subsubsection{Bessel distributions} Let $G$ be a reductive group over a $p$-adic field $F$. 
Denote by $\CS(G)$ the space of Schwartz functions on $G$. Let
$\pi$ be an irreducible smooth admissible representation on $G$ with $\pi^*$ denoting its linear dual. Let $\pi^\vee$ be the contragredient 
representation of $\pi$, that is, the smooth part of $\pi^*$. The group $G$ then acts on all these spaces and these actions
induce actions of $\CS(G)$. Note that for any $\ell \in (\pi^\vee)^*$ and $f \in \CS(G)$, $\pi^\vee(f)\ell$ is smooth so that 
$\pi^\vee(f)\ell \in (\pi^\vee)^\vee = \pi$. For any $\ell_1 \in \pi^*$ and $\ell_2 \in (\pi^\vee)^*$, the relative character 
associated to $\ell_1$ and $\ell_2$ is defined as
\[ B_{\ell_1,\ell_2}(f) = \ell_1(\pi^\vee(f)\ell_2), \quad f \in \CS(G).\]
	
The relative character is a basic object in the relative Langlands program \cite{SV17}. Such relative characters arise naturally in the relative trace formula and local harmonic analysis of spherical varieties. Usually, the linear functionals $\ell_1$ and $\ell_2$
are given to be invariant for a specific subgroup $H$ of $G$.


It is a fundamental problem to determine whether the relative character $B_{\ell_1,\ell_2}$ is regular in the sense that there exists a (unique) 
smooth function $j$ defined on a dense subset of $G$, which is locally integrable on $G$, such that 
\[B_{\ell_1,\ell_2}(f) = \int_G j(g)f(g)dg, \quad f \in \CS(G).\]

Consider the
group case, that is $H$ embeds into $G = H^2$ diagonally. For suitable $H$-invariant linear functionals $\ell_1,\ell_2$ on an irreducible admissible representation $\pi = \sigma \boxtimes \sigma^\vee$ on $G$, the relative character $B_{\ell_1,\ell_2}$ is essentially the character of $\sigma$.
The regularity for (usual) characters of irreducible admissible representations is the celebrated work of Harish-Chandra \cite{HC70,HC99}. 
The regularity gives an injection from  
representations on $H$ to conjugate invariant smooth functions on the regular locus of $H$ which are locally integrable on $H$.

When $\ell_1$ and $\ell_2$ are $H$-invariant linear functionals, where $H$ is the subgroup of fixed points of an involution of $G$, this kind of relative characters are studied in the work of Rader-Rallis \cite{RR}, which shows that the restriction of $B_{\ell_1,\ell_2}$ to the set of regular elements is given by a smooth function \cite[Corollary 5.2]{RR}. But often this smooth function is not locally integrable, and $B_{\ell_1,\ell_2}$ is not regular \cite[Section 7]{RR}. It is an interesting and subtle question to find regular relative characters \cite{Guo, Ha}.

In this paper, we shall focus on the Whittaker case. Let $G$ be a split reductive group over $F$. Let $B = AN$ be 
a Borel subgroup of $G$ with $A$ a maximal torus of $G$ and $N$ the unipotent radical of $B$. Denote by $X^*(A)$ the characters on $A$
and $\Delta \subset X^*(A)$ the set of simple roots. Let $\psi$ be a generic character on $N$, in other words, $\psi|_{N_\alpha} \not=1$
for all $\alpha \in \Delta$ with $N_\alpha$ the corresponding root subgroup of $N$. Consider the space $\Hom_N(\pi,\psi)$ of Whittaker
functionals. The dimension of this space is less than one. Assume $\pi$ is generic (with respect to $\psi$), that is, there is a nonzero Whittaker functional
$\ell_1 \in \Hom_N(\pi,\psi)$. The contragredient representation $\pi^\vee$ is generic (with respect to $\psi^{-1}$) and we take a nonzero Whittaker functional $\ell_2 \in \Hom_N(\pi^\vee,\psi^{-1})$. The relative character $B_{\ell_1,\ell_2}$ is also called the Bessel distribution for $\pi$.  

By the multiplicity one property, the regularity of Bessel distributions is independent of the choice of Whittaker functionals.
By the works of Baruch, the regularity is known for $G = \GL_2$ and $\GL_3$ \cite{Ba97, Ba04}. Moreover, for the general case (that is, $G$ is a quasi-split group over a local field $F$), Baruch \cite{Ba01} showed that the restriction of Bessel distributions to the open Bruhat cell is given by a smooth function (Theorem \ref{restriction}). Once the regularity holds, it
gives an injection from  
$\psi$-generic representations on $G$ to bi-$(N,\psi)$-invariant smooth functions on the open Bruhat cell of $G$ which are locally integrable on $G$ (Proposition \ref{prop-FLO}).

We believe that the regularity of Bessel distributions holds for the general case.

\subsubsection{Kloosterman sums}

Let $G$ be  a split reductive group with a Borel subgroup $B=AN$. Let $\psi$ be a generic character on $N$, $\dot{w_0}$ is a representative for the longest Weyl element $w_0 \in W$ 
and $K$ an open compact subgroup of $G$ with $\psi|_{N \cap K} = 1$. Consider the Kloosterman sum with respect to a triple $(\psi,\dot{w_0},K)$ (See Definition \ref{defn-sys-Kl})
\[\mathrm{Kl}(a) =\Kl(a;\psi,\dot{w_0},K) = \sum_{x \in X_K(\dot{w_0}a)} \psi(u(x))\psi(u'(x)), \quad a \in A.\]
Here, the Kloosterman set
\[X_K(\dot{w_0}a) = (N\cap K)\bs \left(N\dot{w_0}a N \cap K \right)/(N\cap K)\]
is endowed with the following two maps
\[u: X_K(\dot{w_0}a) \ra (N\cap K)\bs N, \quad u': X_K(\dot{w_0}a) \ra N/(N\cap K).\]
If $[n_1\dot{w_0}an_2] \in X_K(\dot{w_0}a)$ with $n_1, n_2 \in N$ and $a \in A$, then 
\[u([n_1\dot{w_0}an_2]) = [n_1], \quad u'([n_1\dot{w_0}an_2]) = [n_2].\]

\begin{example}
Consider the case $G = \GL_2$. Take $\psi\left[ \matrixx{1}{x}{0}{1} \right] = \psi_F(x)$ for some nontrivial additive
character $\psi_F$ on $F$ with $\psi_F|_{\CO} = 1$, $\dot{w_0} = \matrixx{0}{-1}{1}{0}$ and $K = \GL_2(\CO)$ the maximal
open compact subgroups of $G$.
Consider those $a = \matrixx{\alpha}{0}{0}{\alpha^{-1}} \in A$ with $v(\alpha) = r > 0$. Then 
\[\mathrm{Kl}(a) = \sum_{x \in \CO^\times/1+\fp^r} \psi_F\left( \frac{x+x^{-1}}{\alpha} \right).\]
In this case, we have the Weil bound (see Corollary \ref{classical kloosterman sum}), that is, there is a constant $C$ (depending only on $F$) such that
\[|\mathrm{Kl}(a) \cdot \delta(a)^{1/4}| = |\mathrm{Kl}(a)|\cdot q^{-r/2} \leq C.\]
Here, $\delta$ is the modulus character of $B$.
\end{example}

The Kloosterman sum $\mathrm{Kl}$ is said to have a nontrivial bound (See Definition \ref{defn-sys-Kl}) if there exists $\varepsilon>0$ and $C>0$ such that
\[\left|\mathrm{Kl}(a)\cdot\delta^{1/2-\varepsilon}(a)\right| \leq C, \quad a \in A^\der\]
where $A^\der$ is the maximal torus for the derived group $G^\der$ of $G$.

In the above, the exponent $1/2-\varepsilon$ of $\delta$ for any $\varepsilon > 0$ is viewed as a nontrivial bound while the exponent $1/2$  is viewed as the trivial bound. This follows from the counting formula for the cardinality of the Kloosterman set $X(\dot{w_0}a)$ by 
D\k{a}browski-Reeder \cite{DR98}. See Example \ref{special filtration} for more details.

\subsubsection{Main Results}

Now, we can state the main results. Choose the generic character $\psi = \psi_0$ given before Definition \ref{relevant with respect to psi_0}. 
For each subset $I \subset \Delta$, denote by $M_I$ the standard Levi subgroup associated to $I$, $N_{M_I}$ the unipotent radical for the Borel subgroup $B \cap M_I$ of $M_I$ and we fix a representative $\dot{w_I}$ of the longest Weyl element $w_I \in W_I$. 

\begin{thm}[Theorem \ref{main theorem}]\label{main-intro}
     Assume that for any $I \subset \Delta$, there exists an
     admissible open compact subgroup $K_I$ of $M_I$ such that the Kloosterman sum on  $M_I$  with respect to the triple  $(\psi_0|_{N_{M_I}}^{-1}, \dot{w_I}, K_I)$
    has a nontrivial bound. Then, the regularity for the Bessel distribution on $\pi$ holds for any irreducible generic representation $\pi$ on $G$.
\end{thm}

Here, for a split group $G$, an open compact subgroup
$K$ of $G$ is called admissible if for any
$z$ in the center of $G^\der$ and $n  \in N$ with $zn \in K$, we have
$z = 1$ and $n \in K$ (See Definition \ref{defn-adm}).

\begin{thm}[Theorem \ref{thm-type-A} and Theorem \ref{thm-type-C_2}]\label{thm-type-A-intro}
    If $G$ is a split group of type $A$ or $C_2$, then the regularity for Bessel distributions holds for any 
    irreducible generic representation on $G$. 
\end{thm}


As we mentioned before, the regularity is known for $G = \GL_2$ and
$\GL_3$ by Baruch. The case $G = \GL_n(\BQ_p)$ is considered by
Miao \cite{Mia24}. Here, we obtain the regularity for $\GL_n$ with
general $p$-adic fields. Our method is different from \cite{Mia24} (See
Section \ref{sec-intro-kloosterman}).


\subsubsection{Nontrivial bounds for Kloosterman sums}\label{sec-intro-kloosterman}


There are at least two approaches to obtain the nontrivial
bound for Kloosterman sums for higher rank groups. Both
are considered in this paper.

The first approach is based on the stratification of D\k{a}browski-Reeder \cite{DR98} for Kloosterman sets.
Such a stratification decomposes the Kloosterman sum on $G$ 
into a sum of nested Kloosterman sums on $\GL_2$. 
One can then apply the Weil bound on $\GL_2$ to obtain the non-trivial bound. This approach was firstly considered by Blomer-Man \cite{BM24} for $G = \GL_n$ and maximal $K$. It is generalized in \cite{Lin26} for general Weyl elements.

For our application, note that the stratification in \cite{DR98} requires the open compact subgroup $K$ of $G$ is maximal which is not admissible in general (unless the center of $G^\der$ is trivial). It is an interesting question to find admissible groups suitable for this stratification. 

Inspired by the result in \cite[Section 4.4]{Lin26} for the level $\Gamma_0(q)$, for $G = \GL_n$,
we consider the following subgroup
\[K=\left\{g=\begin{pmatrix}
    g' & v\\
    z & t  
\end{pmatrix}\in\GL_n(\CO)\Big|g' \in M_{n-1,n-1}(\CO), v \in M_{n-1,1}(\CO), z \in M_{1,n-1}(\fp^\ell),  t\in1+\fp^\ell\right\}.\]
Here, $\ell$ is a sufficiently large positive integer such that $(1+\fp^\ell) \cap \mu_n = 1$ where
$\mu_n = Z \cap G^\der$ is the group of $n$th roots of unity. For such $\ell$,  $K$ is admissible.

Such an admissible group $K$ has the following nice 
property for the stratification. The stratification is 
based on a product of maps
\[b_\alpha: F \lra G\]
where $\alpha$ are all the simple roots occuring in a
fixed reduced representation for the longest
Weyl element $w_0$.  For $G = \GL_n$, there is a unique simple reflection in a reduced expression of 
$w_0$ that occurs exactly once. For our case, we
can choose
\[w_0=(s_{\alpha_1}\cdots s_{\alpha_{n-1}})(s_{\alpha_1}\cdots s_{\alpha_{n-2}})\cdots(s_{\alpha_1}s_{\alpha_2})s_{\alpha_1}\]
so that the simple reflection $s_{\alpha_{n-1}}$ for the simple root $\alpha_{n-1}$ is the special one. Then  the image of $b_{\alpha_i}$
is contained in $K$ except when $i = n-1$ (See Lemma \ref{Bruhat decomposition}). This property of $K$ enables us to refine the Blomer-Man argument for the maximal open compact subgroup by analyzing only the difference in the $\SL_2$-block corresponding to $b_{\alpha_{n-1}}$ (See the proof of 
Lemma \ref{lifting modulus}). 

This gives the non-trivial
bound for Kloosterman sums on $\GL_n$ with the special
level $K$ (See Theorem \ref{main theorem GL_n} and Corollary \ref{congruence subgroup gamma}).
The non-trivial bound for any split group of type A then follows readily. Applying Theorem \ref{main-intro}, we obtain the regularity for split groups of type $A$ (Theorem 
\ref{thm-type-A-intro}).

The second approach is from  Stevens \cite{Ste87}. Consider the principal congruence groups $K_d$ (Example \ref{special filtration}) which are admissible for large $d$. By considering the action
of the split torus $A_d = A \cap K_d$ on the Kloosterman set $X_d(\dot{w_0}a)$ given by $K_d$, the associated Kloosterman sum  can be written in
terms of Kloosterman sums on $\GL_2$ (Proposition \ref{stevens}). 
Applying the Weil bound for the Kloosterman sums on $\GL_2$ (Proposition
\ref{GL2 kloosterman sum}), the study is reduced to a counting
problem for the Kloosterman set (Corollary \ref{corollary of Stevens' method}).  In fact, in \cite{Ste87}, only the
case $G = \GL_n$ is considered. But it is not hard to generalize to arbitrary split groups. 

The counting is usually based on the Pl\"ucker coordinates. It
is given by matrix coefficients of fundamental representations $\sigma$ on $G$ and looks like
\[N \bs G \hookrightarrow F^{\sum_\sigma \dim\sigma}.\]
The counting is then reduced to the study of the family of conditions given by the Pl\"ucker coordinates which can be
very complicated for higher rank groups.  

For Stevens' approach, see \cite{Ste87} for $G = \GL_3$,
\cite{GSW21} for $G = \GL_4$, \cite{Man22,Man24} for $G = \Sp_4$ with $d=0$ and \cite{Mia24} for $G = \GL_n$ with $F = \BQ_p$.  Here, we consider the case $G = \Sp_4$ with large
enough $d$ (Theorem \ref{Sp_4 nontrivial bound}) which gives the regularity for groups of type
$C_2$.

\subsubsection{The local integrability}

The proof of Theorem \ref{main-intro} is based on several important previous works. Precisely, 
\begin{itemize}
	\item As we have mentioned above, Baruch showed that the restriction of the Bessel distribution for a generic representation $\pi$ to the open Bruhat cell is given by a
		smooth function $j_\pi^0$ (See Theorem \ref{restriction}). In particular, if $j_\pi^0$ is locally integrable, then
		the regularity for the Bessel distribution of $\pi$ holds.
	\item  By the works of Lapid-Mao and Chai, the function $j_\pi^0$ equals to another function $j_\pi$, which is given by 
		certain regularized integral of Whittaker functions over $N$ (See Theorem \ref{thm-regularized} and Theorem \ref{thm-chai}).
	\item  Lapid-Mao proved that $j_\pi$ is locally given by a Kloosterman integral (See Theorem \ref{locally orbital integral 2}). 
\end{itemize}

Therefore, the local integrability of Kloosterman integrals implies the regularity of Bessel distributions.

To give the definition of Kloosterman integrals and its germ expansion, we fix a system of representatives 
$\{w_I^0 \in N_G(A)\}_{I \subset \Delta}$ of $\{w_I \in W\}_{I \subset \Delta}$  which is relevant (with respect to $\psi_0$) 
in the sense that for each $I \subset \Delta$ 
    \begin{itemize}
        \item  $(w_I^0)^2\in A_I = Z(M_I)$.
        \item  $\psi_0\left(\Ad(w^0 w_I^0)u\right)=\psi_0(u)$ for any $u \in N_{M_I}$ with $w^0 = w_\Delta^0$.
    \end{itemize}
In fact, the Tits representatives of Weyl elements give such a system of representatives (See Lemma \ref{relevant weyl elements}).  

Consider the following action of $N \times N$ on $G$
\[g \cdot (n_1,n_2) = \ov{n_1}^{-1} g n_2, \quad \ov{n_1} = (w^0)^{-1}n_1w^0 \in \ov{N}.\]

Denote by $(N\times N)_g$ the stabilizer of $g$. An element $g \in G$ is called relevant if $\psi\left(n_1^{-1}n_2\right) = 1$ 
for any $(n_1,n_2) \in(N\times N)_g$. We have the following Bruhat decomposition for relevant elements (See Lemma \ref{relevant elements})
\[G_{\mathrm{rel}}=\bigsqcup_{I \subset\Delta}\ov{N} w_I^0 A_I N.\]

For a Schwartz function $f \in \CS(G)$ and $g \in G_{\mathrm{rel}}$ relevant, the Kloosterman integral is defined as
\[I(g,f)=\int_{(N\times N)_g\backslash N\times N}f(g \cdot (n_1,n_2))\psi^{-1}(n_1^{-1}n_2)dn_1dn_2.\]

\begin{remark}
    For the $GL_n$ case, Jacquet \cite{Jac16} considers the following
    action
    \[g \cdot (n_1,n_2)= n_1^{t} g n_2\]
    with $n_1^t$ the transport of $n_1$. 
    For general split groups, Lapid-Mao \cite{LM13} considers
    \[g \cdot (n_1,n_2) = n_1^{-1} g n_2.\]
    Our convention is close to the one in \cite{LM13}. In fact,
    our Kloosterman integral and the one in \cite{LM13} are equal
    up to the translate of  $w^0$ (See the proof of
    Theorem \ref{locally orbital integral 2}).
\end{remark}

Based on the counting formula by D\k{a}browski-Reeder for Kloosterman sets mentioned above, the Kloosterman integral $I(\cdot,f)$
is local integrable if  there exist some $\varepsilon > 0$ and $C>0$, such that  
\[ \Big|I(a,f)\delta^{\frac{1}{2}-\varepsilon}(a)\Big| \leq C, \quad a \in A.\] 
See Proposition \ref{prop-delta} for details.

\subsubsection{The Shalika germ expansion}
We consider the germ expansion for Kloosterman integrals.
This is a generalization of the work of  Jacquet-Ye \cite{JY99} and Jacquet \cite{Jac16} for $\GL_n$ to general
split groups.

To state the germ expansion, we need to introduce more notations.  Let $I \subset J$ be two subsets of $\Delta$.  
Consider the intersection
\[A_I^J = A_I \cap M_J^\der.\]
It is called the transverse torus in \cite{CST17} as there
is an isogeny (See Proposition \ref{prop-transverse})
\[A_J \times A_I^J \lra A_I, \quad (b,c) \mapsto bc.\]

For $G = \GL_n$, Jacquet defines $A_I^J$ as the subset
of $A_I$ cut out by certain sub-determinants. In fact, 
the two ways are the same (Remark \ref{remark-GL_n}).

A system of Shalika germs for Kloosterman integrals (See Definition \ref{defn-germ})  is a family of  functions 
\[\{K_I^J \in C^\infty(A_I^J)\}_{I \subset J \subset \Delta}\] 
such that $K_I^I=\delta_e$ for any $I$ and for each  $f \in \CS(G)$, there
    exist $\{\omega_J \in \CS(A_J)\}_{J \subset \Delta}$ such that for any $I \subset \Delta$
    \[I(w_I^0a,f)=\sum_{I \subset J}\left(K_I^J*\omega_{J}\right)(a), \quad
    a \in A_I\]
    with 
    \[\left(K_I^J*\omega_{J}\right)(a) = \sum_{a=bc, b \in A_I^J, c \in A_J} K_I^J(b)\omega_J(c).\]

We now give the result for the Shalika germ expansion
for Kloosterman integrals. This result may stand on its own.

\begin{thm}(Theorem \ref{shalika})
\label{shalika-intro}
There exists a system of Shalika germs for Kloosterman integrals on $G$ which is unique in a suitable sense.
\end{thm}

Once we have generalized the setting to arbitrary split groups, especially the action of $N^2$ on $G$, the Bruhat decomposition
for relevant elements and the transverse torus, the proof
of Theorem \ref{shalika-intro} is an analogue of the one  for $\GL_n$.

By the germ expansion, 
the local integrability of Kloosterman sums is
controlled by the germ functions $\{K_\emptyset^J\}_J$. In fact, from the proof of Theorem \ref{shalika}, we observe that
the germ function can be expressed as Kloosterman integrals
for admissible groups. Precisely, for a family of admissible subgroups $K_J \subset M_J$
with $J \subset \Delta$, there exists a
system of Shalika germs $\{K_I^J\}_{I \subset J}$ such that (Proposition \ref{prop-germ-adm})
\[K_I^J(a) \stackrel{.}{=} I(w_I^0a,1_{w_J^0K_J}), \quad a\in A_I^J.\]
Now, Theorem \ref{main-intro} follows.

\subsubsection{Future directions}
There are several directions one may pursue further after this work. The first one is the regularity of Bessel distributions on general split reductive group $G$. In Theorem \ref{main-intro}, we have reduced the regularity of Bessel distributions attached to irreducible generic representations to nontrivial bounds of  Kloosterman sums for admissible subgroups. We discussed the groups of type $A$ and $C_2$ based on this reduction, and the general case will be considered in our future work.

One may use degenerate Whittaker functional to replace non-degenerate Whittaker functional and consider the corresponding Bessel distributions. In this case, the result of Baruch (Theorem \ref{restriction}) is also true, that is, when restricted to the open Bruhat cell, the Bessel distribution is given by a smooth function. But this smooth function may fail to be locally integrable, and the regularity does not hold (See \cite[Remark 2.4]{Ba01}). The failure of this local integrability may be revealed by looking at the corresponding Shalika germ expansion.



\s{\bf Acknowledgements.} L. Cai is supported by the National Key R\&D Program of China No.
2023YFA1009702 and National Natural Science Foundation of
China, No. 12371012. J. Chai is supported by National Natural Science Foundation of China, No. 12571010 and No. 12671039.

\section{Kloosterman integrals}

In this section, we consider the germ expansion of Kloosterman integrals for split reductive groups. 

\subsection{The germ expansion}

\subsubsection{}\label{notations}

Let $F$ be a $p$-adic field, with ring of integers $\CO$. $|\cdot|$ is the $p$-adic absolute value on $F$ with $v$ the additive valuation. Let $q$ be the cardinality of its residue field $\mathbb{F}_q$. Use $\varpi$ to denote a fixed uniformizer of $F$, and let $\fp$ be the maximal ideal of $\CO$.

Let $G$ be a split (connected) reductive group over a $p$-adic field $F$. We shall make no distinction between an algebraic group and its $F$-points. We fix a Borel subgroup $B=AN$, where $A$ is a maximal split torus of $G$ and $N$ is the unipotent radical of $B$. The center of $G$ is denoted by $Z$. Denote by $n$ the rank of $A$.

Denote by $W = N_G(A)/A$ the Weyl group of $G$ with $N_G(A)$ the normalizer of $A$ in $G$.  Denote by $w_0\in W$ the longest Weyl element. We have the Bruhat decomposition
\[G=\bigsqcup_{w\in W}\overline{N}wAN, \quad \overline{N}:=w_0^{-1}Nw_0.\] 

\subsubsection{}

Denote by $X^*(A):=\Hom(A,F^\times)$ the group of $F$-rational characters of $A$. The Weyl group $W$ acts on $X^*(A)$ by $w \cdot x = x \circ \Ad(w)$, for
any $w \in W$ and $x \in X^*(A)$. Sometimes, we also write $w(x) = w \cdot x$.

Let $\Phi \subset X^*(A)$ be the set of roots of $G$ with respect to $A$. 
Then $\Phi=\Phi^+\sqcup\Phi^-$, where $\Phi^+$ is the set of positive roots and $\Phi^-$ is the set of negative roots
with respect to $B$. For each root $\alpha \in \Phi$, denote by $N_\alpha$ the
root subgroup of $\alpha$.

Denote by $\Delta=\{\alpha_i\}_{1 \leq i \leq m} \subset \Phi^+$, for some $m \leq n$,  the set of simple roots of $G$. For each simple root $\alpha \in \Delta$, let $A_\alpha = \Ker(\alpha)^0 \subset A$ and  $G_\alpha:=Z_G(A_\alpha)$ be its centralizer in $G$. Then $G_\alpha$ is a split reductive group with $A\subset G_\alpha$ a maximal torus. The derived group of $G_\alpha$ has rank $1$ so that the Weyl group $W(G_\alpha,A) \subset W$ of
$G_\alpha$ is isomorphic to $\BZ/2\BZ$. Denote by
$s_\alpha$ the generator of $W(G_\alpha,A)$.

Denote by $X_*(A):=\Hom(F^\times ,A)$ the group of cocharacters of $A$.
Let $\Phi^\vee \subset X_*(A)$ be the set of coroots of $G$ and 
$\Delta^\vee$ be the set of simple coroots. Then $\Phi^\vee=\Phi^{\vee,+}\sqcup\Phi^{\vee,-}$, where $\Phi^{\vee,+}$ is the set of positive coroots and $\Phi^{\vee,-}$ is the set of negative coroots with respect to $B$.

There exists a  perfect pairing
\[\left<\cdot,\cdot\right>\colon X^*(A)\times X_*(A)\to\BZ\]
given as follows: for any $\alpha \in X^*(A)$ and $\lambda \in X_*(A)$,
$\alpha\left(\lambda(t)\right) = t^{\langle \alpha,\lambda \rangle}$ for
any $t \in A$. We have the following decomposition
\[X^*(A)_\BQ=X^*(Z)_\BQ\bigoplus\BQ\Phi, \quad X_*(A)_\BQ=X_*(Z)_\BQ\bigoplus\BQ\Phi^\vee,\]
such that under the pairing $\langle \cdot,\cdot \rangle$,
$X_*(Z)_\BQ=(\BQ\Phi)^\bot$ and $X^*(Z)_\BQ=(\BQ\Phi^\vee)^\bot$.

For each 
$\alpha \in \Delta$, the simple coroot $\alpha^\vee \in \Delta^\vee$ is the unique cocharacter of $A$ such that
\[s_\alpha(x)=x-\left<x,\alpha^\vee\right>\alpha, \quad x\in X^*(A).\]


\subsubsection{}
The standard parabolic subgroups can be described in terms of subsets of $\Delta$. 
For each subset $I$ of $\Delta$, there is a unique standard parabolic subgroup $P_I$ of $G$ such that
\[P_I=\bigsqcup_{w\in W_I}BwB\]
where $W_I$ is the subgroup of $W$ generated by $s_\alpha,\alpha\in I$. 
We have the Levi decomposition 
\[P_I=M_IN_I, \quad N_I=\prod_{\alpha\in\Phi^+\backslash\Phi_I}N_\alpha, \quad
\Phi_I=\BZ I\cap\Phi.\] 
The group $W_I$ is the Weyl group for the Levi subgroup $M_I$. For each subset $I \subset \Delta$, denote by $w_I$ the longest Weyl element in $W_I$. Moreover, every standard parabolic subgroup $P$ of $G$ is of the form $P_I$ for a unique subset $I$ of $\Delta$.

The Levi subgroup $M_I$ has the Borel subgroup $B_{M_I} = B \cap M_I$.
It admits the Levi decomposition $B_{M_I} = AN_{M_I}$ with the unipotent subgroup
\[N_{M_I} = N \cap M_I = \prod_{\alpha \in \Phi^+ \cap \Phi_I} N_\alpha.\]
Denote by 
\[N_I^- = \prod_{\alpha \in \Phi^- \setminus \Phi_I} N_\alpha,
\quad N_{M_I}^- = w_I^{-1}N_{M_I}w_I = \prod_{\alpha \in \Phi^- \cap \Phi_I}
N_\alpha.\]
In particular,
\[N_IN_{M_I} = N, \quad N_I^-N_{M_I}^- = \ov{N}.\]

Denote by $A_I$ the center of $M_I$. Then $A_I=\bigcap_{\alpha\in I}\Ker(\alpha)$ (\cite[Proposition 21.7]{Mil17}) and 
$M_I$ is the centralizer of $A_I^\circ$ in $G$ (\cite[Proposition 21.91]{Mil17}). 

Let $I \subset J$ be two subsets in $\Delta$. Then
$A_I \supset A_J$ and $M_I \subset M_J$.

\subsubsection{}

We fix a non-trivial additive character $\psi_F$ on $F$. We also fix a family of isomorphisms 
\[x_\alpha: F \stackrel{\sim}{\lra} N_\alpha, \quad \alpha\in\Phi.\]
In particular, any element $u \in N$ can be written as
\[u=\prod_{\alpha\in\Phi^+}x_\alpha(u_\alpha), \quad u_\alpha\in F.\]

A character $\psi$ on $N$ is called \emph{generic} if its restriction to each  $N_\alpha$ with $\alpha \in \Delta$ is nontrivial. By \cite[page 15]{Kal22}, we have $N/[N,N]\cong\prod_{\alpha\in\Delta}N_\alpha$. In particular, any generic character is in the form of 
\[\psi(u)=\prod_{\alpha\in\Delta}\psi_F(c_\alpha u_\alpha), \quad c_\alpha \in F^\times, \quad \alpha \in \Delta.\]

We shall only consider the following generic character
\[\psi_0(u)=\prod_{\alpha\in\Delta}\psi_F(u_\alpha).\]

\begin{defn}\label{relevant with respect to psi_0}
    A system of representatives $\{w_I^0 \in N_G(A)\}_{I \subset \Delta}$ of $\{w_I \in W\}_{I \subset \Delta}$ is called relevant with respect to $\psi_0$ if it satisfies the following conditions: 
    \begin{enumerate}
        \item For any $I\subset\Delta$, we have $(w_I^0)^2\in A_I$.
        \item For any $I\subset\Delta$, we have
    $\psi_0\left(\Ad(w^0_\Delta w_I^0)u\right)=\psi_0(u)$ for any $u \in N_{M_I}$.
    \end{enumerate}
\end{defn}

\begin{lem}\label{relevant weyl elements}
    There exists a system of representatives $\{w_I^0 \in N_G(A)\}_{I \subset \Delta}$ which is relevant with respect to $\psi_0$.
    \begin{proof}
        We shall prove that the Tits representatives $\{\wt{w_I}\}_I$ of $\{w_I\}_I$ are relevant to $\psi_0$.
        We refer to Section 5 of \cite{AV16} for the definition and basic properties of Tits representatives. For each $w \in W$
        and $\alpha \in \Phi$, there exists a constant 
        $\kappa_\alpha(\wt{w})\in F^\times$ such that
        \[\Ad(\wt{w})x_\alpha(u_\alpha)=x_{w\cdot\alpha}(\kappa_\alpha(\wt{w})u_\alpha), \quad u_\alpha\in F.\]

        First, we prove that \[\psi_0\left(\Ad(\wt{w_0w_I})u\right)=\psi_0(u), \quad u \in N_{M_I}.\]
        For $u \in N_{M_I}$, write $u=\prod_{\alpha\in\Phi_I^+}x_\alpha(u_\alpha),u_\alpha\in F$. Then
        \[\Ad(\wt{w_0w_I})u=\prod_{\alpha\in\Phi_I^+}
    \Ad(\wt{w_0w_I})x_\alpha(u_\alpha)
            =\prod_{\alpha\in\Phi_I^+}x_{(w_0w_I)\cdot\alpha}(\kappa_{\alpha}(\wt{w_0w_I})u_\alpha).\]
        Hence
        \[\psi_0\left(\Ad(\wt{w_0w_I})u\right)=\prod_{\alpha\in S}\psi_F\left(\kappa_{\alpha}(\wt{w_0w_I})u_\alpha\right), \quad S=\left\{\alpha\in\Phi_I^+ \Big|(w_0w_I)\cdot\alpha\in\Delta\right\}.\] 
        It is clear that $I \subset S$. We claim that $I=S$. Indeed, if $\alpha\in S$, then $(w_0w_I)\cdot\alpha\in\Delta$ and $w_I\cdot\alpha\in-\Delta$. Moreover, we have $w_I\cdot\alpha\in-\Delta\cap\Phi_I^-=-I$, this implies $\alpha\in I$. Therefore
        \[\psi_0\left(\Ad(\wt{w_0w_I})u\right)=\prod_{\alpha\in I}\psi_F\left(\kappa_{\alpha}(\wt{w_0w_I})u_\alpha\right).\]
        For any $\alpha \in \Delta$ and $w \in W$, if $w \cdot \alpha \in \Delta$, then $\kappa_\alpha(\wt{w}) = 1$ (See \cite[Proposition 9.3.5]{Spr98}). We have $\kappa_{\alpha}(\wt{w_0w_I})=1$ and \[\psi_0\left(\Ad(\wt{w_0w_I})u\right)=\psi_0(u), \quad u \in N_{M_I}.\]
        
         By \cite[Lemma 5.4]{AV16}, one has $\wt{w_I}^2\in A_I$. Denote by $\ell$ the length function on $W$. For any $w\in W$, by \cite[Proposition 2.3.2 (ii)]{BB05}, we have $\ell(w_0)=\ell(w_0w)+\ell(w)$. In particular, $\ell(w_0) = \ell(w_0w_I) + \ell(w_I)$.  Since the Tits
         representative $\wt{w}$ of $w \in W$ is independent of 
         the choice of  reduced representations of $w$, we have
         \[\wt{w_0}=\wt{w_0w_I}\cdot\wt{w_I}\]
         by $\ell(w_0) = \ell(w_0w_I) + \ell(w_I)$.
         
         Therefore, we have 
        \[\psi_0(u) = \psi_0\left(\Ad(\wt{w_0}\wt{w_I}^{-1})u\right)
         = \psi_0\left(\Ad(\wt{w_0}\wt{w_I})u\right)
        , \quad u\in N_{M_I}.\]
    \end{proof}
\end{lem}

In the following, we fix a representative $\{w_I^0 \in N_G(A)\}_{I \subset \Delta}$ of $\{w_I \in W\}_{I \subset \Delta}$ which is relevant with respect to $\psi_0$.
We shall write $w^0 = w^0_\Delta$ for simple.

Consider the following action of $N\times N$ on $G$
\[g \cdot (u_1,u_2) = \ov{u_1}^{-1} g u_2, \quad \ov{u_1} = (w^0)^{-1} u_1w^0 \in \ov{N}.\]
Denote by $(N\times N)_g$  the stabilizer of $g$ with respect to the action of $N\times N$.

An element $g \in G$ is called {\em relevant} with respect to $w^0$ and $\psi_0$ if $\psi_0\left(u_1^{-1}u_2\right) = 1$ for any $(u_1,u_2) \in(N\times N)_g$.
Denote by $G_\rel= G_{_\rel}^{w^0}$ the set of relevant elements in $G$. In particular, $w^0$ is also a relevant element. 
\begin{lem}\label{relevant elements}
    We have
    \[G_{\rel}=\bigsqcup_{I \subset\Delta}\ov{N} w_I^0 A_I N.\]
\end{lem}

We shall prove a generalized version of this Lemma 
(See Lemma \ref{relevant elements for Levi}).

\subsubsection{}
Following Cogdell-Shahidi-Tsai \cite[Section 5.2.7]{CST17}, we consider the transverse torus
\[A_I^J =A_I\cap M_J^{\operatorname{der}},\]
where $M_J^{\operatorname{der}}$ is the derived subgroup of $M_J$ (Note that in general $A_I^J$ is not a torus). In particular, $A_\emptyset^\Delta=A\cap G^{\operatorname{der}}$ is the maximal torus of the derived subgroup $G^{\operatorname{der}}$. 


\begin{prop}\label{prop-transverse}
    We have the following isogeny map
    \[A_J\times A_I^J\to A_I,\quad(b,c)\to bc.\]
    \begin{proof}
        Any connected reductive group $G$ can be written as $G=Z(G)\cdot G^{\operatorname{der}}$ and  there is an isogeny map $f_G:  Z(G)\times G^{\operatorname{der}}\to G$. In particular, for every $J\subset\Delta$, we have an isogeny map
        \[f_{M_J}: A_J\times M_J^{\operatorname{der}}\to M_J,\quad (b,c)\mapsto bc.\]
        The map $A_J \times A_I^J \ra A_I$ is 
        the restriction of $f_{M_J}$ to 
         $A_J\times(A_I\cap M_J^{\operatorname{der}})$. 
         In particular, the kernel of the map is finite. 
        
        Now we prove that the map $A_J \times A_I^J \ra A_I$ is surjective (as algebraic groups). For any 
        $a \in A_I$, $a \in M_I \subset M_J$. We may
        write $a = bc$ with $b \in A_J$ and $c \in M_J^\der$. We have
        \[c = ab^{-1} \in A_I \cap M_J^\der = A_I^J.\]
        Therefore, the map is surjective and it is an 
        isogeny map.

    \end{proof}
\end{prop}

\begin{cor}\label{lem-finite}
Let $I \subset J$ be two subsets of $\Delta$, we have $A_I^J\cap A_J=A_J^J$ and the set $A_{J}^{J}$ is finite. For any $a \in A_I$, there are only finitely many  $b\in A_I^J$ and $c\in A_J$ such that $a=bc$.
\begin{proof}
    By the above Proposition, the kernel of the map $A_J\times A_I^J\to A_I$ is finite. This is equivalent to the fiber of every element $a$ in $A_I$ is finite, and equivalent to $A_J^J=A_I^J\cap A_J$ is finite.
\end{proof}
\end{cor}

Now we give a characterization of $A_I^J$ via the character of $A$.
\begin{lem}\label{lem-image-res}
    Fix $J\subset\Delta$. Consider the restriction map
    \[\operatorname{res}\colon X^*(M_J)\to X^*(A),\quad\chi\mapsto\chi|_{A}.\]
    Then we have
    \[\operatorname{res}(X^*(M_J))=\{\lambda\in X^*(A):\left<\lambda,\alpha^\vee\right>=0,\forall \alpha\in J\}.\]
    \begin{proof}
        
        Suppose $\lambda\in X^*(M_J)$, then $\lambda$ is trivial on the derived subgroup $M_J^\der$.  For any $\alpha \in J$, $\Im\alpha^\vee \in M_J^\der$. Hence, $\left<\lambda,\alpha^\vee\right>=0$
        

        Assume $\lambda\in X^*(A)$ such that $\left<\lambda,\alpha^\vee\right>=0$ for all $\alpha\in J$. Then $\lambda$ is orthogonal to  $\BZ\Phi_J^\vee=X_*(A\cap M_J^\der)$. Hence, $\lambda$ is trivial on $A\cap M_J^\der$.  We
        have a surjective morphism of algebraic groups
        \[\pi: A \lra M_J/M_J^\der.\]
        It induces an isomorphism
        \[A/(A\cap M_J^\der)\cong M_J/M_J^\der.\]
        In particular, $\lambda$ is the restriction
        of a character on $M_J$.

    \end{proof}
\end{lem}
\begin{prop}
    For any $I\subset J\subset\Delta$, we have
    \[A_I^J=\left\{a\in A_I\Big| \lambda(a)=1,\forall\lambda\in X^*(A),\left<\lambda,\alpha^\vee\right>=0,\forall\alpha\in J\right\}.\]
    \begin{proof}
        This follows from Lemma \ref{lem-image-res} as $M_J^\der=\bigcap_{\lambda\in X^*(M_J)}\ker\lambda$. 
    \end{proof}
\end{prop}
\begin{remark}\label{remark-GL_n}
    In \cite{Jac16}, for the case $G = \GL_n$, Jacquet considers the functions $\Delta_i(g) = \Delta_i(g_{i,i})$, $1 \leq i \leq n$ where the 
    sub-matrix $g_{i,i}$ of $g \in G$ is formed with the first $i$ rows and the first $i$ columns of $g$. The function $\Delta_i$ is a matrix coefficient for the
    representation corresponding to the dominant weight $\lambda_i = e_1+\cdots+e_i,1\leq i\leq n$.  
    Jacquet considers
    \[
    \begin{aligned}
    A_I^J &=\left\{a\in A_I\Big| \Delta_i(w_I^0a)=\Delta_i(w_J^0) \text{ for all $i$ with } \Delta_i(w_J^0)\neq 0\right\} \\
    &= \left\{a\in A_I\Big| \lambda_i(a)=
    \frac{\Delta_i(w_J^0)}{\Delta_i(w_I^0)} \text{ for all $1 \leq i \leq n-1$ with  $\alpha_i \not\in J$ and $i=n$} \right\}. 
    \end{aligned}
    \]
    Here, $w_I^0,I\subset\Delta$ denotes a representative for longest Weyl element $w_I\in W_I$. With the choice of $w_I^0$ as the Tits representative of $w_I$, we have $\Delta_i(w_J^0)/\Delta_i(w_I^0)=1$. Hence, this definition agrees with ours.
\end{remark}

\subsubsection{}
We denote by $\CS(G) = C_c^\infty(G)$ the space of Schwartz functions on $G$.

\begin{defn}\label{defn-Kl-G}
The Kloosterman integrals for $G$ with respect to $w^0$ and $\psi_0$ are defined as follows
\[I(g,f)=\int_{(N\times N)_g\backslash N\times N}f\left(g \cdot (u_1,u_2)\right)\psi_0^{-1}\left(u_1^{-1}u_2\right)du_1du_2, \quad g \in G_\rel, \quad f \in \CS(G).\]
\end{defn}

Note that if $g \in\ov{N}B$, then $(N\times N)_g = 1$ so that $g$ is relevant. Moreover, $I(\cdot,f)$ is smooth on $\ov{N}B$. 

For any $I\subset J\subset\Delta$, $f_1 \in C^\infty(A_I^J)$
and $f_2 \in C^\infty(A_J)$, consider the following function on
$A_I$
\[(f_1 * f_2)(a) = \sum_{(b,c)} f_1(b)f_2(c).\]
Here, the sum runs over all pairs $(b,c)$ with $b \in A_I^J$
and $c \in A_J$ such that $a = bc$. If
no such pair $(b,c)$ with $a = bc$, then we just set $(f_1 * f_2)(a) = 0$. By Corollary \ref{lem-finite}, the above sum is in fact
a finite sum.

\begin{defn}\label{defn-germ}
A {\em system of Shalika germs} for Kloosterman integrals in Definition \ref{defn-Kl-G}  is a family of  functions 
\[\{K_I^J \in C^\infty(A_I^J)\}_{I \subset J \subset \Delta}\] 
such that
\begin{itemize}
    \item $K_I^I=\delta_e$ for any $I$. Note that by Corollary
    \ref{lem-finite}, the set $A_I^I$ is finite and contains
    the identity $e$ of $G$.
    \item For each Schwartz function $f \in \CS(G)$, there
    exist Schwartz functions $\{\omega_J \in \CS(A_J)\}_{J \subset \Delta}$ such that for any $I \subset \Delta$
    \[I(w_I^0a,f)=\sum_{I \subset J}\left(K_I^J*\omega_{J}\right)(a), \quad
    a \in A_I.\]
\end{itemize}
\end{defn}
\begin{remark}
    The functions $\omega_J$  depend on the system of Shalika germs $\{K_I^J\}$ and $f$.
\end{remark}

\begin{thm}[The Shalika germ expansion for Kloosterman integrals]\label{The Shalika germ expansion}
    \label{shalika}
    Let $G$ be a  split reductive group over $F$. There exists a system of Shalika germs for Kloosterman integrals on $G$. The systems of Shalika germs satisfy the following uniqueness property. If $\{K_I^J\}_{I \subset J}$ is a system of Shalika germs, and 
    $\{t_I^J \in \CS(A_I^J)\}_{I \subset J}$ is a family of
    Schwartz functions with $t_I^I = \delta_e$ for all $I$,
    then the functions
    \[H_I^{J}=\sum_{I\subset I_1\subset J}K_I^{I_1}*t_{I_1}^{J}\]
    form another system of Shalika germs. Moreover, all systems of Shalika germs are obtained in this way from a given system.
\end{thm}

We shall give more properties of the germ expansion in
Section \ref{sec-prop-sha}.

\subsection{The proof}

Here, we give a proof of Theorem \ref{shalika}. We have fixed a
system of representatives $\{w_I^0 \in N_G(A)\}_I$ of
$\{w_I \in W\}_I$ relevant with respect to $\psi_0$.

For any $J\subset\Delta$, consider the following action of $N_{M_J}\times N_{M_J}$ on $M_J$
\[g \cdot (n_1,n_2) = \ov{n_1}^{-1} g n_2, \quad \ov{n_1} = (w_J^0)^{-1}n_1w_J^0 \in N_{M_J}^-.\]
Denote by $(N_{M_J}\times N_{M_J})_g$ the stabilizer of $g$ with respect to the action of $N_{M_J}\times N_{M_J}$.

An element $g \in M_J$ is called {\em relevant} with respect to $w_J^0$ and $\psi_0$ if $\psi_0\left(n_1^{-1}n_2\right) = 1$ for any $(n_1,n_2) \in(N_{M_J}\times N_{M_J})_g$.
Denote by  $(M_J)_{_\rel}^{w_J^0}$ the set of relevant elements in $M_J$. In particular, $w_J^0$ is also a relevant element.

The following lemma generalizes Lemma \ref{relevant elements}. 

\begin{lem}\label{relevant elements for Levi}
    A system of representatives $\{w_I^0 \in N_G(A)\}_{I \subset \Delta}$ of $\{w_I\}_{I \subset \Delta}$ relevant with respect to $\psi_0$ has the following properties:
    \begin{enumerate}
        \item For any subset $I\subset J$ and $n\in N_{M_I}$, we have
              \[\psi_0(\Ad(w_J^0w_I^0)n)=\psi_0(n).\]
        \item For any subset $J\subset\Delta$, we have
        \[(M_J)_{\rel}^{w_J^0}=\bigsqcup_{I \subset J}N_{M_J}^- w_I^0 A_I N_{M_J}.\]
    \end{enumerate}
    \begin{proof}
    (1) As $(w^0)^2 \in Z$, we have
    \[\psi_0(\Ad(w_J^0w_I^0)n)=\psi_0(\Ad(w_J^0w^0)\Ad(w^0w_I^0)n).\]
    Denote by $n_1=\Ad(w^0w_I^0)n$. Since $\Ad(w^0w_I^0)N_{M_I}\subset\Ad(w^0w_J^0)N_{M_J}$, there exists $n_2\in N_{M_J}$ such that 
    \[n_1=\Ad(w^0w_J^0)n_2=\Ad(w^0(w_J^0)^{-1})n_2.\]
    This implies that
    \[\Ad(w_J^0w^0)n_1=n_2.\]
    Therefore, as $\{w_I^0\}$ are relevant to $\psi_0$, we have
    \[\psi_0(\Ad(w_J^0w_I^0)n)=\psi_0(\Ad(w_J^0w^0)n_1)=\psi_0(n_2) =\psi_0(\Ad(w^0w_J^0)n_2) =\psi_0(n_1)
        =\psi_0(\Ad(w^0w_I^0)n)
        =\psi_0(n).\]

    (2) Since $(M_J)_{\rel}^{w_J^0}$ is left $N_{M_J}^-$-invariant and right $N_{M_J}$-invariant, for $g\in (M_J)_{\rel}^{w_J^0}$, we may assume $g\in wA$ with $w\in W_J$. We prove that $w=w_I$ for some $I \subset J$. If $w$ is not of the above form, then by \cite[Lemma 89]{Ste68} there exists $\alpha\in\Phi_J^+\backslash J$ such that $\beta=(w_Jw)^{-1}\cdot\alpha\in J$. Then $(N_{M_J}\times N_{M_J})_g$ would contain $\{(x_\alpha(t),x_\beta(ct)):t\in F\}$ for some $t\in F^\times$, in contradiction to the condition $g\in(G)_{\rel}^{w_J^0}$.
 
    Next we show that $g\in w_I^0A_I$. Fix $a\in A$, and for any $n_2\in N_{M_I}$, set
    \[n_1=\Ad(w_J^0w_I^0)\Ad(a)n_2\in N_{M_J}.\]
    Then
    \[(w_J^0)^{-1}n_1w_J^0w_I^0an_2^{-1}=w_I^0a, \quad
    (n_1,n_2) \in (N_{M_J}^2)_{w_I^0a}.\]
    If $w_I^0a$ is relevant, by (1) we have
    \[\psi_0(n_2)=\psi_0(n_1)=\psi_0\left(\Ad(w_J^0w_I^0)\Ad(a)n_2\right)=\psi_0(\Ad(a)n_2).\] 
    Since $n_2\in N_{M_I}$ is arbitrary, we must have $a\in A_I$.

    Now suppose $a\in A_I$ and $(n_1,n_2)\in(N_{M_J}\times N_{M_J})_{w_I^0a}$, then
    \[(w_J^0)^{-1}n_1^{-1}w_J^0w_I^0an_2=w_I^0a.\]
    This implies $an_2a^{-1}=(w_J^0w_I^0)^{-1}n_1(w_J^0w_I^0)\in N_{M_J}\cap(w_J^0w_I^0)^{-1}N_{M_J}(w_J^0w_I^0)$. For any $w \in W(M_J)$, 
    \[N_{M_J} \cap w^{-1}N_{M_J} w = \prod_{\alpha \in \Phi_J^+ \cap w^{-1}(\Phi_J^+)} N_{M_J,\alpha}.\]
    As $\Phi_I^+ = \Phi_J^+ \cap (w_J^0w_I^0)^{-1}(\Phi_J^+)$, we have
    \[N_{M_J}\cap(w_J^0w_I^0)^{-1}N_{M_J}(w_J^0w_I^0) = N_{M_I}.\]
    Since $a\in A_I$, we have $n_2\in N_{M_I}$, $an_2a^{-1} = n_2$
    and $n_1=\Ad(w_J^0w_I^0)n_2$.
    Therefore, by (1)
    \[\psi_0(n_1)=\psi_0(\Ad(w_J^0w_I^0)n_2)=\psi_0(n_2).\]
\end{proof}
\end{lem}

Consider the Kloosterman integral for Levi subgroups of $G$. For $J\subset\Delta$, the Kloosterman integrals for $M_J$ with respect to $w_J^0$ and $\psi_0$ are defined as follows
\[I(g,f)=\int_{(N_{M_J}\times N_{M_J})_g\backslash N_{M_J}\times N_{M_J}}f\left(g \cdot (n_1,n_2)\right)\psi_0^{-1}\left(n_1^{-1}n_2\right)dn_1dn_2, \quad g \in (M_J)_\rel^{w_J^0},\quad f \in \CS(M_J).\]

We transfer Kloosterman integrals on $G$ for functions supported
on the following open subset $\Omega_J$ of $G$ to Kloosterman integrals on $M_J$.

For each subset $J\subset \Delta$, consider the cell
\[\Omega_J:= \bigsqcup_{w \in W_J} \ov{N} wA N.\]

\begin{lem}\label{decomposition of omega_w}
    The set $\Omega_{J}$ is open in $G$ and the map
    \[N_J^-\times M_J\times N_J \lra \Omega_J, \quad (n_1,m,n_2)\mapsto n_1mn_2\]
    is an isomorphism of analytic varieties over $F$.
\end{lem}
    \begin{proof}
        We have the following Bruhat decomposition for the standard Levi subgroup $M_J$ of the standard parabolic subgroup $P_J$
        \[M_J=\bigsqcup_{w\in W_J}N_{M_J}^-wAN_{M_J}.\]
         The set
        \[N_J^-M_JN_J=\bigsqcup_{w\in W_J}N_J^-N_{M_J}^-wAN_{M_J}N_J=\bigsqcup_{w\in W_J}\overline{N}wAN = \Omega_J.\]

        If $n_1 m n_2 = n_1' m' n_2'$ with $n_1, n_1' \in N_J^-$, $m,m' \in M_J$ and $n_2,n_2' \in N_J$, then
        \[(n_1')^{-1}n_1 = m'n_2'(n_2)^{-1}m^{-1} \in N_J^-\cap P_J.\]
        As $N_J^-\cap P_J = \{1\}$, this gives $n_1 = n_1'$ and also
        $m = m'$, $n_2 = n_2'$.
    \end{proof}
The following lemma is easy to see.
\begin{lem}\label{compact supported functions on M_I}
    Let $f\in\CS(\Omega_J)$. We set
    \[h(m)=h_f(m)=\int_{N_{J}^-\times N_{J}}f(n_1^{-1}mn_2)\psi_0^{-1}(\ov{n_1}^{-1}n_2)dn_1dn_2,\quad m\in M_J.\]
    Here, $\ov{n} = (w^0)^{-1}nw^0$.
    Then $h\in\CS(M_J)$ and every $h\in\CS(M_J)$ can be obtained in this way.
\end{lem}
    
\begin{prop}\label{transform orbital integrals}
    Suppose $f_J\in\CS(\Omega_{J})$, then we have
    \[I(w_I^0a,f_J)=I(w_I^0a,h_{f_J}), \quad I \subset J, \quad  a \in A.\]
    Conversely, suppose $h\in\CS(M_J)$, then there exists $f_h\in\CS(\Omega_J)$ such that
    \[I(w_I^0a,h)=I(w_I^0a,f_h), \quad I \subset J, \quad a \in A.\]
    \begin{proof}
        Consider the integral
        \begin{align*}
        I(w_{I}^0a,f_{J})=\int_{(\overline{N}\times N)_{w_I^0a}\backslash(\overline{N}\times N)}f_J(u_1^{-1}w_I^0au_2)\psi_0^{-1}(\overline{u}_1^{-1}u_2)du_1du_2.
        \end{align*}
        
        We claim that
        \[ (\bar{N} \times N)_{w_I^0a} = (N_{M_J}^- \times N_{M_J})_{w_I^0 a}.\]
        In fact, if $(u_1'n_1)^{-1} w_I^0a (u_2'n_2) = w_I^0a$
        for $u_1' \in N_{M_J}^-$, $u_2' \in N_{M_J}$, $n_1 \in N_J^-$ and $n_2 \in N_J$, then 
        \[ n_1^{-1} \left((u_1')^{-1} w_I^0 a u_2'\right) n_2
        = w_I^0a \in \Omega_J \cong N_J^- \times M_J \times N_J.\]
        Therefore,
        \[n_1 = n_2 = 1, \quad (u_1')^{-1} w_I^0 a u_2' = w_I^0a.\]
        The claim now holds. It implies that
        \[\left[(N_{M_J}^- \times N_{M_J})_{w_I^0a} \bs
        (N_{M_J}^- \times N_{M_J})\right] \times (N_J^- \times N_J) 
        \stackrel{\sim}{\lra} (\ov{N} \times N)_{w_I^0 a}
        \bs \ov{N} \times N.\]
        
        We have
        \begin{align*}
        I(w_{I}^0a,f_{J})&=\int_{(\overline{N}\times N)_{w_I^0a}\backslash(\overline{N}\times N)}f_J(u_1^{-1}w_I^0au_2)\psi_0^{-1}(\overline{u}_1^{-1}u_2)du_1du_2\\
        &=\int_{(N_{M_{J}}^-\times N_{M_{J}})_{w_I^0a}\backslash (N_{M_{J}}^-\times N_{M_{J}})}\left(\int_{N_{J}^-\times N_{J}}f_J(n_1^{-1}u_1^{-1}w_I^0au_2n_2)\psi_0^{-1}(\overline{u_1n_1}^{-1}u_2n_2)dn_1dn_2\right)du_1du_2\\
        &=\int_{(N_{M_{J}}^-\times N_{M_{J}})_{w_I^0a}\backslash (N_{M_{J}}^-\times N_{M_{J}})}h_{f_J}(u_1^{-1}w_I^0au_2)\psi_0^{-1}((w^0)^{-1}u_1^{-1}w^0u_2)du_1du_2,
        \end{align*}
        where the function $h_{f_J} \in \CS(M_J)$ is  the one given in Lemma \ref{compact supported functions on M_I}.
       
        Since $(w^0)^2\in Z$, for $u_1\in N_{M_J}^-$ we have \[\psi_0(\Ad((w^0)^{-1})u_1)=\psi_0(\Ad(w^0)u_1)=\psi_0(\Ad(w^0w_J^0)\Ad((w_J^0)^{-1})u_1)=\psi_0(\Ad((w_J^0)^{-1})u_1)\] as $w_J^0$ is relevant. Hence, we have
        \begin{align*}
            I(w_{I}^0a,f_{J})&=\int_{(N_{M_{J}}^-\times N_{M_{J}})_{w_I^0a}\backslash (N_{M_{J}}^-\times N_{M_{J}})}h_{f_J}(u_1^{-1}w_I^0au_2)\psi_0^{-1}((w_J^0)^{-1}u_1^{-1}w_J^0u_2)du_1du_2\\
            &=\int_{(N_{M_{J}}\times N_{M_{J}})_{w_I^0a}\backslash (N_{M_{J}}\times N_{M_{J}})}h_{f_J}((w_J^0)^{-1}
            u_1^{-1} w_J^0w_I^0au_2)\psi_0^{-1}(u_1^{-1}u_2)du_1du_2 \\
            &= I(w_I^0a,h_{f_J}).
        \end{align*}

        The second assertion comes from the second part of Lemma \ref{compact supported functions on M_I}.
    \end{proof}
\end{prop}
We expect that only relevant parts of Bruhat cells contribute to the orbital integrals. All other cells should make no contribution. The following lemmas analyze the contributions from such non-relevant cells.
\begin{lem}\label{decomposition of functions}\
    \begin{enumerate}
        \item Denote by $\Omega$ the complement of the closed set $N_{M_I}^-w_I^0AN_{M_I}=w_I^0AN_{M_I}$ in $M_I$. Let $f \in \CS(M_I)$ such that the orbital integrals $I(w_I^0a,f)$ vanish for all $a\in A_I$.  Then, there is $f_0\in \CS(\Omega)$ such that $I(g,f)=I(g,f_0)$.
        \item Suppose that $\Omega_0\subset \Omega$ are open sets of $M_I$ which are invariant under the action of $N_{M_I}\times N_{M_I}$ and $A$. Suppose further the difference $\Omega-\Omega_0$ contains no set of the form $N_{M_I}^-w_J^0AN_{M_I}$ with $J\subset I$. Then for any $f\in \CS(\Omega)$, there is $f_0\in \CS(\Omega_0)$ such that $I(g,f)=I(g,f_0)$.
    \end{enumerate}
    \begin{proof}
        The proof is the same as \cite[Lemma 2.2, 2.3, 2.4]{Jac16}. 
    \end{proof}
\end{lem}

Next, we prove a proposition on the decomposition of functions. The corresponding result for $\GL_n$ is contained in the proof of \cite[Proposition 2.1]{Jac16}.
\begin{prop}\label{prop-decomp}
    Fix a subset $J\subset\Delta$. There exist functions $K_{I}^{J}\in C^\infty(A_{I}^{J}),I\subset J$ satisfying
    the following condition. For any $f\in\CS(M_J)$, there exists a function $\omega_J\in \CS(A_J)$ (depends on $f$) such that
    \[f=\sum_{I':|J\backslash I'|=1}f_{I'}+f_2+f_3,\]
    where 
    \begin{itemize}
    \item for each $I' \subset J$ with $|J \setminus I'| = 1$, $f_{I'}\in \CS(\Omega_{I'}\cap M_J)$;
    \item $f_2\in\CS(M_J)$ satisfies
    $I(w_I^0a,f_2)=(K_I^{J}*\omega_{J})(a)$ 
    for any $I\subset J$ and $a\in A_I$;
    \item $f_3\in\CS(M_J)$ satisfies $I(w_I^0a,f_3)=0$ for all $I\subset J$ and $a\in A_I$.
    \end{itemize}
    \begin{proof}
        First we have $w_J^0A_J\cap M_J^\der=w_J^0A_{J}\cap w_J^0M_J^{\der}=w_J^0(A_J\cap M_J^\der)=w_J^0A_J^J$ is a finite set. The subsets $w_J^0aN_{M_J}$ with $a\in A_{J}^{J}$ are closed and disjoint. Note that $(N_{M_J}\times N_{M_J})_{w_J^0a}\backslash(N_{M_J}\times N_{M_J})\simeq N_{M_J}$, so we have
        \[I(w_J^0a,f)=\int_{N_{M_J}}f(w_J^0au)\psi_0^{-1}(u)du,\]
        for $a\in A_J,f\in\CS(M_J)$. Thus, there exists $f_0\in\CS(M_J)$ such that for $a\in A_{J}^{J}$,
        \[I(w_J^0a,f_0)=\delta_e(a).\]
        For $I\subset J$, we define a function $K_{I}^{J}$ on $A_{I}^{J}$ by
        \[K_{I}^{J}(a)=I(w_I^0a,f_0).\]
        In particular, we have $K_{J}^{J}=\delta_e$.

        Let $f\in\CS(M_J)$, we define a function $\omega_{J}^f$ on $A_{J}$ by the formula
        \[\omega_{J}^f(a)=I(w_J^0a,f).\]
        Since the set $N_{M_J}^-w_J^0A_{J}N_{M_J}$ is closed in $M_J$, the function $\omega_{J}$ is indeed in $\CS(A_{J})$. 
        
        We define a function
        \[f_2(g)=\sum f_0(g_1)I(w_J^0c,f)=\sum f_0(g_1)\omega_J^f(c)\]
        for $g\in M_J$, where the sum is over all pairs $(g_1,c),g_1\in M_J^\der,c\in A_{J}$ such that $g=g_1c$. For a given $g$, the sum is finite. If such decomposition does not exist, we define $f_2(g)=0$. 

        We now study the orbital integral $I(w_I^0a,f_2)$ for $I\subset J$ and $a\in A_I$, the integral
        \[I(w_I^0a,f_2)=\int f_2(\overline{u}_1^{-1}w_I^0au_2)\psi_0^{-1}(u_1^{-1}u_2)du_1du_2.\]
        We must consider all possible decompositions
        \[\overline{u}_1^{-1}w_I^0au_2=g_1c,\]
        with $g_1\in M_J^\der$ and $c\in A_{J}$. Since $c$ is in the center of $M_J$, we can write
        \[\overline{u}_1^{-1}w_I^0ac^{-1}u_2=g_1\]
        and thus
        \[g_1=\overline{u}_1^{-1}w_I^0bu_2,\]
        where $b\in A_{I}$ verifies $a=bc$. Since $g_1\in M_J^\der$, we have $b_1\in M_J^\der$, that is, $b\in A_{I}^{J}$. Hence,
        \[f_2(\overline{u}_1^{-1}wau_2)=\sum f_0(\overline{u}_1^{-1}wbu_2)I(w_J^0c,f),\]
        where the sum is over all $(b,c)$ with $b\in A_I^{J},c\in A_{J}$ and $a=bc$. Note that $(N_{M_J}\times N_{M_J})_a=(N_{M_J}\times N_{M_J})_b$ Then we find
        \[I(w_I^0a,f_2)=\sum I(w_I^0b,f_0)I(w_J^0c,f)=K_{I}^{J}*\omega_{J}^f(a)\]
        for $a\in A_I$. In particular, we have $I(w_J^0a,f_2)=I(w_J^0a,f)$ for $a \in A_J$. 
        
        Let $\Omega$ be the complement of the set $N_{M_J}^-w_J^0AN_{M_J}$ in $M_J$. By Lemma \ref{decomposition of functions} (1), there exists $f_2'\in \CS(\Omega)$ such that
        \[I(g,f-f_2)=I(g,f_2'),\]

        Now consider $I'\subset J$ with $|J\backslash I'|=1$ and suppose $J\backslash I'=\{\alpha_j\}$. Suppose $x\in N_{M_J}^-w_J^0AN_{M_J}\cap(\Omega_{I'}\cap M_J)$, then there exists some $w\in W_{I'}$ such that $x\in\ov{N}wAN$. This implies $x\in\ov{N}wAN\cap M_J$. Next we prove that
        \[\ov{N}wAN\cap M_J\subset(\ov{N}\cap M_J)wA(N\cap M_J)=N_{M_J}^-wAN_{M_J}.\]
        Suppose $x=\ov{n}wan\in\ov{N}wAN\cap M_J$, then we have $x\in N_{M_J}^-w'AN_{M_J}$ for some $w'\in W_J$ by the Bruhat decomposition of $M_J$. Note that $N_{M_J}^-w'AN_{M_J}\subset\overline{N}w'AN$ and we must have $w=w'\in W$. This implies $x\in N_{M_J}^-wAN_{M_J}$. Now $x$ belongs simultaneously to two Bruhat cells of $M_J$:
        \[x\in N_{M_J}^-w_J^0AN_{M_J}\cap N_{M_J}^-wAN_{M_J},\]
        this forces $w=w_J^0$. But $w_J^0\notin W_{I'}$. So no such $x$ exists.
        
        From above we see that $N_{M_J}^-w_J^0AN_{M_J}\cap(\Omega_{I'}\cap M_J)=\emptyset$. This implies the open set
        \[\Omega_0=\bigcup_{I'\subsetneq J,|J\backslash I'|=1}\Omega_{I'}\cap M_J\]
        is contained in $\Omega$. Thus, we may apply Lemma \ref{decomposition of functions} to the pair $(\Omega_0,\Omega)$. It follows that $I(g,f_2')=I(g,f_2'')$ for some $f_2''\in\CS(\Omega_0)$. Thus, we may take $f_2'\in \CS(\Omega_0)$. Using a partition of unity, we may write further
        \[f_2'=\sum_{I'\subsetneq J,|J\backslash I'|=1}f_{I'}.\]
        with $f_{I'}\in\CS(\Omega_{I'}\cap M_J)$. So finally we have
        \[f=\sum_{I':|J\backslash I'|=1}f_{I'}+f_2+f_3,\]
        where $f_3$ satisfies $I(w_I^0a,f_3)=0$ for all $I\subset J,a\in A_I$.
    \end{proof}
\end{prop}

Now, we give the proof for the existence of Shalika germs. The argument is similar to the case of $\GL_n$, as treated in \cite[Proposition 3.1]{Jac16}. 

Consider the Kloosterman integral with respect to the decomposition in Proposition \ref{prop-decomp}. Let $f\in\CS(G)$. For $I\subset\Delta$ and $a\in A_{I}$, we have
\[I(w_{I}^0a,f)=\sum_{I':I'\subsetneq\Delta:|\Delta\backslash I'|=1}I(w_{I}^0a,f_{I'})+(K_{I}^{\Delta}*\omega_{\Delta})(a).\]
Note that in the first sum the integral $I(w_{I}^0a,f_{I'})=0$ unless $w_I^0A_I$ intersects $\Omega_{I'}$ which happens only if $I\subset I'$. Thus above formula can be written in the form
\[I(w_{I}^0a,f)=\sum_{I':I\subset I'\subsetneq\Delta:|\Delta\backslash I'|=1}I(w_{I}^0a,f_{I'})+(K_{I}^{\Delta}*\omega_{\Delta})(a).\]
By Proposition \ref{transform orbital integrals}, there exists $h_{I'}=h_{f_{I'}}\in\CS(M_{I'})$ such that
\[I(w_{I}^0a,f_{I'})=I(w_{I}^0a,h_{I'}),\]
where the latter orbital integral is for the group $M_{I'}$. 

Moreover, we can apply  Proposition \ref{prop-decomp} for Schwartz functions on $M_{I'}$. We have a
function $\omega_{I'}\in\CS(A_{I'})$, and, for each $I\subset I'$, a function $K_{I}^{I'}$. For each $I''$ with $|I'\backslash I''|=1$, there is a function $_{I''}h_{I'}\in \CS(\Omega_{I''}\cap M_{I'})$ such that, for $I\subset I'$,
\[I(w_{I}^0a,h_{I'})=\sum_{I\subset I''\subsetneq I':|I'\backslash I''|=1}I(w_I^0a, {_{I''}}h_{I'})+(K_{I}^{I'}*\omega_{I'})(a).\]

By Lemma \ref{compact supported functions on M_I}, we can choose a function $_{I''}f_{I'}$ supported on the set $\Omega_{I'}$ such that, for any $I$ with $I\subset I''$,
\[I(w_I^0a, {_{I''}}h_{I'})=I(w_I^0a, {_{I''}}f_{I'}).\]

We can furthermore that $_{I''}f_{I'}$ is supported on the set $\Omega_{I''}$. To prove this fact, we need some $N^2$-invariant continuous functions in Lemma \ref{properties of Delta_i} analogue the functions $\Delta_i$ for the case $G = \GL_n$ in \cite{Jac16}. If $\Delta_i(w_{I''}^0)\neq 0$, there are constants $C_{\Delta_i}$ and $D_{\Delta_i}$ such that $C_{\Delta_i}\leq |\Delta_i(g)|\leq D_{\Delta_i}$ for $g$ in the support of $_{I''}h_{I'}$.  This implies that
\[I(w_I^0a, {_{I''}}h_{I'})\neq 0\implies C_{\Delta_i}\leq|\Delta_i(w_I^0a)|\leq D_{\Delta_i},\]
and thus
\[I(w_I^0a, {_{I''}}f_{I'})\neq 0\implies C_{\Delta_i}\leq|\Delta_i(w_I^0a)|\leq D_{\Delta_i}.\]

Denote by $\phi$ the characteristic function of the set $\{t\in F^\times:C_{\Delta_i}\leq|t|\leq D_{\Delta_i}\}$.
Multiplying $_{I''}f_{I'}$ by $\phi\circ\Delta_i$, we may assume $_{I''}f_{I'}$ is supported on the set defined by $C_{\Delta_i}\leq|\Delta_i(g)|\leq D_{\Delta_i}$. Repeating this construction for each $\Delta_i$ with $\Delta_i(w_{I''}^0)\neq 0$, we may assume that $_{I''}f_{I'}$ is supported on the set $\Omega_{I''}$.

We have 
\begin{align*}
    I(w_{I}^0a,f)&=\sum_{I':I\subset I'\subsetneq\Delta:|\Delta\backslash I'|=1}\sum_{I\subset I''\subsetneq I':|I'\backslash I''|=1}I(w_I^0a,{_{I''}}f_{I'})+\sum_{I':I\subset I'\subsetneq\Delta:|\Delta\backslash I'|=1}(K_{I}^{I'}*\omega_{I'})(a)+(K_{I}^{\Delta}*\omega_{\Delta})(a)\\
    &=\sum_{I\subset I''\subsetneq I':|I'\backslash I''|=1}I\left(w_I^0a,\sum_{I':I\subset I'\subsetneq\Delta:|\Delta\backslash I'|=1}{_{I''}f_{I'}}\right)+\sum_{I':I\subset I'\subsetneq\Delta:|\Delta\backslash I'|=1}(K_{I}^{I'}*\omega_{I'})(a)+(K_{I}^{\Delta}*\omega_{\Delta})(a)
\end{align*}
We set
\[f_{I''}=\sum_{I':I\subset I'\subsetneq\Delta:|\Delta\backslash I'|=1}{_{I''}f_{I'}}\in\CS(\Omega_{I''}).\]
Now we have the relation
\[I(w_{I}^0a,f)=\sum_{I\subset I''\subsetneq\Delta:|\Delta\backslash I''|=2}I(w_I^0a,f_{I''})+\sum_{I':I\subset I'\subsetneq\Delta:|\Delta\backslash I'|=1}(K_{I}^{I'}*\omega_{I'})(a)+(K_{I}^{\Delta}*\omega_{\Delta})(a).\]
Repeating this construction, we obtain the existence of germ expansions.

We shall omit the proof of the uniqueness of the system of Shalika germs as one can follow the  argument for  the $\GL_n$ case (See \cite[page 930]{JY96}). The proof of Theorem \ref{shalika} is complete.

\subsection{Shalika germs as Kloosterman integrals}\label{sec-prop-sha}

In the following, we shall express the Shalika germs as some special Kloosterman integrals. See \cite[Section 2]{JY99} for the $\GL_n$ case.

We consider the cone of dominant characters
\[X^*(A)^+=\left\{\lambda\in X^*(A)\Big| \left<\lambda,\alpha^\vee\right>\geq 0,\forall\alpha\in\Delta\right\}.\]

Let $\lambda_i\in X^*(A)^+,1 \leq i \leq m$ such that $\left<\lambda_i,\alpha_j^\vee\right>=0$ if and only if $i\neq j$.  For $I\subset J\subset\Delta$, consider
\[A_I^J(r,\lambda):=\left\{a\in A_I^J \Big| |\lambda_{i}(a)|\leq r, \forall \alpha_i\in J\backslash I\right\}.\]

\begin{remark}
    For the case $G = \GL_n$, consider the functions
    $\Delta_i$ in Remark \ref{remark-GL_n}.  In \cite{Jac16}, Jacquet considers
    \[
    \begin{aligned}
    A_I^J(r) &=\left\{a\in A_I^J\Big| \Delta_i(w_I^0a)\leq r \text{ for all $i$ with } \Delta_i(w_I^0)\neq 0,\Delta_i(w_J^0)=0\right\} \\
    &=\left\{a\in A_I^J\Big| \lambda_i(a)\leq r\cdot\Delta_i(w_I^0)^{-1} \text{ for all $i$ with } \alpha_i\in J\backslash I\right\}.
    \end{aligned}
    \]
     With the choice of $w_I^0$ as the Tits representative of 
     $w_I$, we have $\Delta_i(w_I^0)=1$ for $\Delta_i(w_I^0)\neq 0$. Hence, this definition agrees with ours.
\end{remark}

\begin{lem}
    Let $\lambda=(\lambda_1,\cdots,\lambda_m)$ and $\lambda'=(\lambda_1',\cdots,\lambda_m')$ be two such
    dominant characters. For fixed $r>0$, there exists $r_1,r_2>0$ such that
    \[A_I^J(r_1,\lambda')\subset A_I^J(r,\lambda)\subset A_I^J(r_2,\lambda').\]
    \begin{proof}
        
        Denote by $\omega_i\in\Phi_\BQ,1\leq i \leq m$ the fundamental weights of $G$, that is,  $\left<\omega_i,\alpha_j^\vee\right>=\delta_{ij},1\leq i,j\leq m$. Then every $\lambda_i$ can be written as
\[\lambda_i=k_i\omega_i+\chi_i, \quad k_i\in\BQ_{>0}, 
\quad \chi_i\in X^*(Z)_\BQ.\]

        We choose $N,N'>0$ such that
        \[N\lambda_i=Nk_i\omega_i+N\chi_i,N'\lambda_i'=N'k_i'\omega_i+N'\chi'_i,1\leq i\leq m,\]
        where $Nk_i\omega_i,N'k_i'\omega_i\in X^*(A)$ and $N\chi_i,N'\chi_i'\in X^*(Z)$.
        
        If $a\in A_I^J(r_1,\lambda')$, then $|\lambda_i'(a)|=|N'\lambda_i'(a)|^{1/N'}=|(N'k_i'\omega_i)(a)|^{1/N'}\leq r_1$ for $\alpha_i\in J\backslash I$. We have
        \[|\lambda_i(a)|=|(N\lambda_i)(a)|^{\frac{1}{N}}=|(Nk_i\omega_i)(a)|^{\frac{1}{N}}=|(N'k_i'\omega_i)(a)|^{\frac{k_i}{N'k_i'}}\leq r_1^{\frac{k_i}{k_i'}}.\]
        If we take $r_1$ such that $r_1^{k_i/k_i'}\leq r$, then we have $A_I^J(r_1,\lambda')\subset A_I^J(r,\lambda)$.

        For $A_I^J(r,\lambda)\subset A_I^J(r_2,\lambda')$, the proof is similar.
    \end{proof}
\end{lem}

For simplicity, we use $A_I^J(r)$ to denote $A_I^J(r,\lambda)$ if $\lambda$ is fixed.

\begin{prop}\label{germ function as Kloosterman sum}
    There exists a system of Shalika germs $\{K_I^J\}_{I \subset J}$ such that
    \[K_I^J(a)=I(w_I^0a,f_J),\]
    where $f_J$ is any Schwartz function on $M_J$ such that $I(w_J^0a,f_J)=\delta_e(a)$ for any $a \in A_J^J$. In particular,
\begin{enumerate}
    \item the function $K_I^J$ is supported on $A_I^J(r)$ for some $r>0$.
    \item for any $J' \subset \Delta$,  the family of functions
    $\{K_I^J\}_{I \subset J \subset J'}$ gives a system of Shalika germs
    for the group $M_{J'}$.
\end{enumerate}
\end{prop}
    
\begin{remark}
    In fact, the above Proposition can be strengthened as following.
    \begin{itemize}
        \item Using the uniqueness of Shalika germs, we can prove that
        (1) and (2) in the above Proposition holds for any system
        of Shalika germs.
        \item  
        For each $J \subset \Delta$, let $f_J \in \CS(M_J)$ such that
        $I(w_J^0a,f_J)=\delta_e(a)$ for any $a \in A_J^J$. 
        Then, for any $r>0$, we can construct a system of Shalika germs $\{K_I^J\}_{I \subset J}$ such that
    \[K_I^J(a) = I(w_I^0a,f_{J}), \quad I\subset J\subset\Delta,a \in A_I^J(r')\]
    for some $r'>0$ and $K_I^J$ is supported on $A_I^J(r)$. See
    \cite[Proposition 2.6]{JY99} for $G = \GL_n$.
    \end{itemize}
\end{remark}

To prove this Proposition, we introduce some $N^2$-invariant
functions analogue the functions $\Delta_i$ for the case $G = \GL_n$ (See Remark \ref{remark-GL_n}).  

For this, we choose a basis $\{\lambda_i^\circ\}_{m+1 \leq i \leq n}$ of $X^*(Z)_\BQ \subset X^*(A)_\BQ$.  In particular, 
\[\{\omega_1,\cdots,\omega_m,\lambda_{m+1}^\circ,\cdots,\lambda_n^\circ\}\] 
gives a basis of $X^*(A)_\BQ=X^*(Z)_\BQ\bigoplus\BQ\Phi$. We also choose $d_i\in\BZ_{>0},1\leq i\leq n$ such that $\lambda_i:=d_i\lambda_i^\circ,1\leq i\leq n$ are dominant weights in $X^*(A)$.

By the highest weight theory, for each $1 \leq i \leq n$, there exists a unique (up to isomorphism) irreducible (algebraic) representation $\pi_i$ of $G$ such that its highest weight is $\lambda_i$ (\cite[Theorem 22.2]{Mil17}). Consider the matrix coefficient
\[\Delta_i(g)=\left<\pi_i(g)\varepsilon_i,\eta_i^*\right>, \quad g \in G.\]
Here, $\varepsilon_i$ is a nonzero highest weight vector in $\pi_i$ and $\eta_i^*$ is a nonzero lowest weight vector in the dual representation of $\pi_i$. 

We record basic properties for $\Delta_i$. The following lemma
may stand on its own.

\begin{lem}\label{properties of Delta_i}
The functions $\Delta_i$, $1 \leq i \leq n$ satisfy the following basic properties.
    \begin{enumerate}
        \item If $g=u_1\dot{w}au_2$ with $u_1 \in \ov{N}$, $\dot{w}$ is representative element of $w \in W$, $a \in A$ and $u_2 \in N$, then $\Delta_i(g)=\lambda_i(a)\Delta_i(\dot{w})$. This means $\Delta_i(\cdot)$ are all $N\times N$-invariant functions on $G$.
        \item For any representative element $\dot{w}$ of $w\in W$, 
        $\Delta_i(\dot{w})\neq 0$ if and only if $w\cdot \lambda_i =\lambda_i$.
        \item For $i>m$, we always have $\Delta_i(g)\neq 0$
        for any $g \in G$.
        \item Let $I$ be a subset of $\Delta$. 
        \begin{enumerate}
        \item Let $w\in W$. We have
        $w\cdot \lambda_i=\lambda_i$ for all  $i \leq m$ with $\alpha_i\notin I$ if and only if $w\in W_I$.
        \item Let $i \leq m$. We have $w_I \cdot \lambda_i=\lambda_i$ if and only if $\alpha_i\notin I$.
        \end{enumerate}
    \end{enumerate}
\end{lem}
\begin{proof}
    (1) Let $1 \leq i \leq n$. Since $\varepsilon_i$ is the highest weight vector and $\eta_i^*$ is the lowest weight vector in the dual representation, we have
    \[\pi_i(n)\varepsilon_i=\varepsilon_i,\quad \pi_i(\overline{n})\eta_i^*=\eta_i^*\]
    for $n\in N,\overline{n}\in\overline{N}$.
    Thus,
    \[\Delta_i(g)=\Delta_i(wa)=\left<\pi_i(wa)\varepsilon_i,\eta_i^*\right>=\lambda_i(a)\left<\pi_i(w)\varepsilon_i,\eta_i^*\right> = \lambda_i(a)\Delta_i(w).\]

    (2) For any representative element $\dot{w}$ of $w\in W$, $\pi_i(\dot{w})\varepsilon_i\in V_{w\cdot\lambda_i}$. Here, $V_{w\cdot\lambda_i}$ is the $(w\cdot \lambda_i)$-eigenspace. 
    For a vector $x \in \pi_i$ which is $A$-eigen, $\langle x,\eta_i^* \rangle \not=0$ implies that $x \in V_{\lambda_i}$.
    In particular, $\left<\pi_i(\dot{w})\varepsilon_i,\eta_i^*\right>\neq 0$ if and only if $w \cdot \lambda_i=\lambda_i$. 

    (3) Let $i>m$. By (2), we shall prove that for any $w \in W$, 
    $w\cdot \lambda_i = \lambda_i$. For any $\alpha\in\Delta$, we have
    \[s_\alpha(\lambda_i)=\lambda_i-\left<\lambda_i,\alpha^\vee\right>\alpha=\lambda_i\]
     since $X^*(Z)_\BQ=(\BQ\Phi^\vee)^\bot$. As $W$ is generated by these
     $s_\alpha$,  $w\cdot\lambda_i=\lambda_i$ for any $w \in W$.

    (4a)  For each dominant weight $\lambda \in X^*(A)$, denote
    by
    \[W_\lambda = \{w \in W | w \cdot \lambda = \lambda\}.\]
    For any $1 \leq i,j \leq m$,  we have $s_{\alpha_j}
    \in W_{\lambda_i}$ if and only if $j\neq i$. By
    \cite[Proposition 20.11]{Bum13},  for any $w \in W_{\lambda_i}$, if $w = s_{\alpha_{j_1}} \cdots s_{\alpha_{j_\ell}}$ is a reduced representation of $w$,
    then each $s_{\alpha_{j_k}} \in W_{\lambda_i}$, $1 \leq k \leq \ell$. Therefore, 
    \[W_{\lambda_i}=W_{\Delta\backslash\{\alpha_i\}}.\]
    In particular,
    \[W_I=\bigcap_{\alpha_i\notin I}W_{\Delta\backslash\{\alpha_i\}}=\bigcap_{\alpha_i\notin I}W_{\lambda_i}
    = \{w\in W:w\cdot\lambda_i=\lambda_i\text{ for all }\alpha_i\notin I\}.\]

    (4b) Suppose $\alpha_i\notin I$. For any $\alpha\in I$, we have
    \[s_\alpha(\lambda_i)=\lambda_i-\left<\lambda_i,\alpha^\vee\right>\alpha=\lambda_i.\]
    This implies $w_I\cdot\lambda_i=\lambda_i$. 
    
    On the other hand, if $\alpha_i\in I$, we have
    \[\left<w_I\cdot\lambda_i, \alpha_i^\vee\right>=\left<\lambda_i,w_I\cdot\alpha_i^\vee\right>\leq 0,\]
    since $w_I\cdot\Phi_I^{\vee,+}=\Phi_I^{\vee,-}$. This implies $w_I\cdot \lambda_i\neq\lambda_i$.
\end{proof}
\begin{proof}[The proof of \ref{germ function as Kloosterman sum}]
    (1)  For any $I\subset J\subset\Delta$, in the above proof of the existence of
    Shalika germ expansion (see the beginning in the proof of Proposition \ref{prop-decomp}), we choose 
    \[K_I^J(a)=I(w_I^0a,f_J)\]
    for any function $f_J\in\CS(M_J)$ such that $I(w_J^0a,f_J)=\delta_e(a)$ for any $a\in A_J^J$.
        
    For a fixed $J\subset\Delta$, the function $K_I^J(a)$ is the Kloosterman integral for $f_J\in\CS(M_J)$. For any $g \in M_J$ in the support of function $f_J\in\CS(M_J)$ and  any $i$, $\Delta_i(g)$ is contained in a compact support of $F$. Thus, for each $i$, there exists $C_i$ such that $K_I^J(a) \not=0$ implies that 
    \[|\Delta_i(w_I^0a)| = |\Delta_i(w_I^0)|\cdot |\lambda_i(a)| \leq C_i.\]
    By (2) and (4b) in 
    Lemma \ref{properties of Delta_i}, $\Delta_i(w_I^0) \not=0$
    if and only if $\alpha_i \not\in I$. Therefore, $K_I^J(a)$ is supported on $A_I^J(r)$ for some
    $r > 0$.

    (2) By definition of $\{K_I^J\}_{I\subset J\subset J'}$, the functions $\{K_I^J\}_{I\subset J\subset J'}$ is a system of Shalika germs for $M_J$.
\end{proof}

In the following, we shall give a way to construct the above $f_J$.

\begin{lem}
    Denote by $Z^\der$ the center of $G^\der$. Then as $F$-points
    of algebraic groups, $Z^\der = Z \cap G^\der$.
\end{lem}
\begin{proof}
    As algebraic groups, $Z^\der = Z \cap G^\der$. Consider their
    $F$-points,
    \[Z(G^\der(F)) = Z(G^\der)(F) = (Z \cap G^\der)(F) = Z(F) \cap 
    G^\der(F).\]
\end{proof}

\begin{defn}\label{defn-adm}
    An open compact subgroup $K\subset G$ with
    $\psi_0|_{N \cap K} = 1$ is called admissible (with
    aspect to $\psi_0$)
    if for any $z \in Z^\der$ and $n \in N$
    with $zn \in K$, we have $z = 1$ and $n \in K$.
\end{defn}   

We give several examples for admissible subgroups.

\begin{example}
    Suppose the center of $G$ is trivial. Then any
    open compact subgroup $K$ of $G$ with $\psi_0|_{N \cap K} = 1$ is admissible.
\end{example}

\begin{example}\label{GL_n example}
    Suppose $G=\GL_n$ and we assume $\psi$  is trivial on $N(\CO)$. Consider the following subgroup of maximal compact subgroup $\GL_n(\CO)$:
    \[K=\left\{g=\begin{pmatrix}
    g' & v\\
    z & t  
\end{pmatrix}\in\GL_n(\CO)\Big|g' \in M_{n-1,n-1}(\CO), v \in M_{n-1,1}(\CO), z \in M_{1,n-1}(\fp^\ell),  t\in1+\fp^\ell\right\},\]
    Here, $\ell$ is a sufficiently large positive integer such that $(1+\fp^\ell) \cap \mu_n = 1$ where
    $\mu_n = Z \cap G^\der$ is the group of $n$th roots of unity. Then $K$ is admissible.
\end{example}

\begin{example}\label{special filtration}
   For each $d \geq 0$, consider the following open compact subgroups of $A$ and $N_{\alpha,d}$, $\alpha \in \Phi$
    \[A_d:=\{a\in A:\chi(a)\in (1+\fp^d) \cap \CO^\times \text{ for all }\chi\in X^*(A)\}, \quad N_{\alpha,d}=x_\alpha(\fp^d).\]
    Let $K_d$ be the subgroup of $G$ generated by $A_d$ and all $N_{\alpha,d}$ for $\alpha\in\Phi$. Then $\{K_d\}_{d\geq 1}$ is a filtration of compact open subgroup of $G$ with $\bigcap_{d\geq 1}K_d=\{1\}$ and the Iwahori decomposition $K_d=\ov{N}_dA_dN_d$.

    As $Z^\der$ is finite, we can take $d>0$ sufficiently large such that
    $Z^\der \cap A_d = 1$. Therefore, by the Iwahori decomposition, for
    $d>0$ sufficiently large, $K_d$ is an admissible group of $G$.
\end{example}

\begin{prop}\label{prop-germ-adm}
    For every $J\subset\Delta$, let $K_J$ be an admissible
    open compact subgroup of $M_J$. Denote by 
    \[f_{J} = \frac{1}{\Vol(N_{M_J}\cap K_{J})}\cdot1_{w_J^0K_{J}} \in \CS(M_J).\]
    Then, there exists a system of Shalika germs $\{K_I^J\}_{I\subset J}$ such that
    \[K_I^J(a)=I(w_I^0a,f_{J}).\]
\end{prop}
\begin{proof}
    We can assume $J=\Delta$ and $M_J=G$. Consider the function $1_{w^0K_\Delta}\in\CS(G)$. Let $z \in Z^\der$. Then the stabilizer
    $N^2_z = \Delta N$ so that
    \[I(w^0 z, 1_{w^0 K_\Delta}) = \int_N 1_{w^0K_\Delta} (w^0 zn)
    \psi_0^{-1}(n)dn.\]
    The admissible condition on $K_\Delta$ implies that
    \[I(w^0z,1_{w^0 K_\Delta}) = \Vol(N \cap K_\Delta) \delta_e(a).\]
    Therefore, by Proposition \ref{germ function as Kloosterman sum}, for any $I\subset\Delta$ we can choose the germ
    $K_I^\Delta(a) =I(w_I^0a,f_{\Delta})$. 
\end{proof}

\section{Bessel distributions}

In this section, based on the germ expansion of Kloosterman 
integrals, we relate the regularity of Bessel distributions to
the nontrivial bound of Kloosterman sums (See Theorem \ref{main theorem}).

\subsection{The regularity}
Let $G$ be a split (connected) reductive group over a $p$-adic field $F$. As before, we fix a Borel subgroup $B = AN$ of $G$. Let $\Phi \subset X^*(A)$ be the set of roots of $G$ with respect to $A$ and use $\Delta$ to denote the set of simple roots. For each $\alpha\in\Phi$, denote by $N_\alpha$ the corresponding root subgroup. Denote by $w_0$ the
longest Weyl element in the Weyl group $W = N_G(A)/A$ of $G$. Denote by $\Omega = Bw_0B$ the open Bruhat cell in $G$. Let $\psi$ be a generic character on $N$. 

Let $\pi$ be an irreducible smooth admissible representation on $G$. Let $\pi^\vee$ be the  contragredient representation of $\pi$. 
Denote by $\pi^*$ and $(\pi^\vee)^*$ the linear dual of $\pi$ and $\pi^\vee$ respectively. The representation $\pi$ is called \emph{generic} if $\Hom_N(\pi,\psi)\neq 0$. For a generic representation $\pi$, one has $\dim\Hom_N(\pi,\psi)=1$. A Whittaker functional is an element $\ell\in\Hom_N(\pi,\psi)$. If $\pi$ is generic, then so is $\pi^\vee$. Let $\ell_1 \in\pi^*$ and $\ell_2 \in(\pi^\vee)^*$ be nonzero Whittaker functionals with respect to $\psi$ and $\psi^{-1}$, respectively. In particular, for any Schwartz function $f \in \CS(G)$, $\pi^\vee(f)\ell_2 \in (\pi^\vee)^\vee = \pi$. The Bessel distribution for $\pi$ is defined as
\[B_\pi(f) = \ell_1 \left( \pi^\vee(f) \ell_2 \right), \quad f \in \CS(G).\]

We believe that the Bessel distribution $B_\pi$ satisfies the following
property. 

\begin{defn}
	The Bessel distribution $B_\pi$ is called regular if there is a unique smooth function $j_\pi \in C^\infty(\Omega)$ which is
	locally integrable on $G$ such that
	\[B_\pi(f) = \int_G j_\pi(g)f(g)dg, \quad f \in \CS(G).\]
\end{defn}

By the uniqueness of Whittaker functionals, the regularity of the
Bessel distribution for $\pi$ is independent of the choice of 
$\ell_1$ and $\ell_2$.

\begin{lem}
   Let $\psi' = \psi(\Ad(a)\cdot)$ for some $a \in A$. Let $\pi$ be an
   irreducible smooth admissible representation on $G$ which is generic
   with respect to $\psi$. Then $\pi$ is also generic with respect to $\psi'$. The regularity of the Bessel distribution for $\pi$
   with respect to $\psi$ implies the regularity of the Bessel distribution
   for $\pi$ with respect to $\psi'$.
    \begin{proof}
        Let $\ell_1 \in\pi^*$ and $\ell_2 \in(\pi^\vee)^*$ be nonzero Whittaker functionals with respect to $\psi$ and $\psi^{-1}$. Consider
        \[\ell_1'(v)=\ell_1(\pi(a)v)\in\Hom_N(\pi,\psi'), \quad \ell_2'(v^\vee)=\ell_2(\pi^\vee(a^{-1})v^\vee)\in\Hom_N(\pi^\vee,\psi'^{-1})\]
        For $f\in\CS(G)$, we have
        \[B_\pi'(f)=\ell_1'(\pi^\vee(f)\ell_2')=\ell_1(\pi(a)\pi^\vee(f)\pi^\vee(a^{-1})\ell_2) = \ell_1
        (\pi^\vee(f_a)\ell_2), \quad f_a(g) = f(a^{-1}ga).\]
        Hence, the regularity of $B_\pi$ implies regularity of $B_\pi'$.
    \end{proof}
\end{lem}

When the regularity holds, generic representations
are determined by their Bessel functions.
\begin{prop}\label{prop-FLO}
    Let $\pi_1,\pi_2$ be two generic smooth irreducible representations of $G$. Assume the distributions $B_{\pi_1}$ and $B_{\pi_2}$ are regular, say represented by
    functions $j_{\pi_1}$ and $j_{\pi_2}$ respectively.  If there exists a non-zero constant $c$, such that for any $g\in \Omega$, we have
    \[j_{\pi_1}(g)=cj_{\pi_2}(g).\]
    Then $\pi_1\cong\pi_2$.
    \begin{proof}
        By \cite[Lemma 2.2(2)]{FLO12}, the Bessel distributions are linear independent for two inequivalent smooth irreducible generic representations.  Note that \cite[Lemma 2.2(2)]{FLO12} only proved in the case $G=\GL_n$. However, this proof can be generalized to arbitrary split reductive groups. This Proposition now follows directly from the regularity of Bessel distributions.
    \end{proof}
\end{prop}

\subsection{Local integrability}
\begin{thm}[Baruch, Theorem 2.3 in \cite{Ba01}]\label{restriction} There exists a smooth function $j_\pi^0 \in C^\infty(\Omega)$ such that
	\[B_\pi(f) = \int_G j_\pi^0(g)f(g)dg, \quad f \in \CS(\Omega).\] 
\end{thm}

\begin{prop}\label{locally integrable implies regular}
	If $j_\pi^0$ is locally integrable on $G$, then $B_\pi$ is regular. 
\end{prop}
\begin{proof}
	Assume $j_\pi^0$ is locally integrable. Consider the distribution
	\[B_\pi^1(f) = \int_G j_\pi^0(g)f(g)dg, \quad f \in \CS(G).\]
	Then by the above Theorem of Baruch, the support for the difference $B_\pi - B_\pi^1$ is contained in $G - \Omega$. Note that by Theorem 1.1 and Theorem 1.3 in \cite{Chai19}, both $B_\pi$ and $B_\pi^1$ are $\fz$-finite, where $\fz$ denotes the Bernstein center of $G$. Now by
	results in \cite[Theorem A]{AGS15} and \cite[Corollary B and C]{AGK15}, $B_\pi - B_\pi^1 = 0$. 
\end{proof}

For any $v \in \pi$, denote by $W_v(g) = \ell_1(\pi(g)v)$
the corresponding Whittaker function. The Whittaker model
$\CW(\pi,\psi)$ of $\pi$ is the subspace of $C^\infty(G)$ spanned by these $W_v$, $v \in \pi$.

\begin{thm}[Lapid-Mao, Theorem 2 in \cite{LM13}]\label{thm-regularized}
    Consider the following filtration of compact open subgroups of $N$
    \[\CN: N_1 \subset N_2 \subset \cdots, \quad N = \bigcup_i N_i.\]
    For any $g \in \Omega$ and 
    $W \in \CW(\pi,\psi)$, consider the limit
    \[\int_N^\reg W(gn)\psi^{-1}(n)dn = \lim_{i \ra +\infty} \int_{N_i} W(gn)\psi^{-1}(n)dn.\]
    Then, the limit exists and is independent of the choice of the filtration $\CN$. 
\end{thm}

In particular, by the uniqueness of Whittaker functionals, for each $g \in \Omega$, there is a scalar $j_\pi(g)$ such that
\[\int_N^\reg W(gn)\psi^{-1}(n)dn = j_\pi(g)W(e), \quad W \in \CW(\pi,\psi).\]

In the following, we fix a family of isomorphisms 
$x_\alpha:F \stackrel{\sim}{\lra} N_\alpha$, $\alpha \in \Phi$ and
consider the generic character 
\[\psi = \psi_0\] 
given before Definition \ref{relevant with respect to psi_0}.
\begin{thm}[Chai, Theorem 1.1 in \cite{Chai19}]\label{thm-chai}
    We have $j_\pi^0(g) = j_\pi(g)$ for any $g \in \Omega$. 
\end{thm}
In particular, the Bessel distribution $B_\pi$ is regular if and only if $j_\pi$ is locally integrable.

The following result relates $j_\pi$ to Kloosterman integrals. As before, we fix a system of representatives $\{w_I^0\}_{I \subset \Delta}$
of $\{w_I\}_{I \subset \Delta}$ which is relevant with respect to $\psi_0$.

Let $\omega$  be a character on the center $Z$ of $G$. We consider an analogue of the Kloosterman integral. For $f \in \CS(G)$,
\[I_\omega(g,f) = \int_{N\times Z\times N}f(zg\cdot(n_1,n_2)))\omega^{-1}(z)\psi_0^{-1}(n_1^{-1}n_2)dn_1dzdn_2, \quad g \in \ov{N}B,\]
where
\[g \cdot (u_1,u_2) = \ov{u_1}^{-1} g u_2, \quad \ov{u_1} = (w^0)^{-1} u_1w^0 \in \ov{N}.\]
The above integral converges absolutely and defines a locally constant function on $\overline{N}B$.

\begin{thm}[Lapid-Mao, Theorem 4 in \cite{LM13}] \label{locally orbital integral 2}
    Let $\pi$ be an irreducible admissible generic representation on $G$ with central character $\omega$. For any compact subset $U\subset G$, 
    there exists a Schwartz function $f\in \CS(G)$ such that
    \[j_\pi(g)=I_\omega((w^0)^{-1}g,f), \quad g \in U \cap \Omega.\]
\end{thm}
\begin{proof}
    By \cite[Theorem 4]{LM13}, for any compact subset $U\subset G$, there exists a Schwartz function $f\in\CS(G)$ such that
    \[j_\pi(g)=J_f^\psi(g),\quad g\in U.\]
    Here
    \[J_f^\psi(g)=\int_{N\times Z\times N}f(n_1^{-1}zgn_2)\omega^{-1}(z)\psi_0^{-1}(n_1^{-1}n_2)dn_1dn_2,\quad g\in\Omega.\]
    If $g\in\Omega$, then
    \begin{align*}
        J_{f}^\psi(g)&=\int_{N\times Z\times N}f(n_1^{-1}zgn_2)\omega^{-1}(z)\psi_0^{-1}(n_1^{-1}n_2)dn_1dzdn_2\\
        &=\int_{N\times Z\times N}f(w^0\ov{n_1}^{-1}z(w^0)^{-1}gn_2)\omega^{-1}(z)\psi_0^{-1}(n_1^{-1}n_2)dn_1dzdn_2\\
        &=I_\omega((w^0)^{-1}g,f_{w^0}),
    \end{align*}
    where $f_{w^0}(g)=f(w^0g),g\in G$. 
\end{proof}
\begin{lem}
   For any $f \in \CS(G)$, $I_\omega(\cdot,f)$ is locally integrable if $I(\cdot,f)$ is locally integrable.
\end{lem}
\begin{proof}
	Let $Q$ be a compact set in $G$. We must show that $\int_{Q}|I_\omega(g,f)|dg$ is bounded on $Q$ in $G$. We have
    \begin{align*}
        \int_{Q}|I_\omega(g,f)|dg&=\int_{Q}\Big|\int_ZI(gz,f)\omega(z)^{-1}dz\Big|dg\\
        &\leq\int_Q\int_Z|I(zg,f)\omega(z)^{-1}|dzdg\\
        &=\int_Z\left(\int_Q|I(zg,f)|dg\right)|\omega(z)^{-1}|dz
    \end{align*}
    Consider the functions $\Delta_i, m+1\leq i\leq n$ in Lemma
\ref{properties of Delta_i}. If $g\in Q,z\in Z$ and $gz$ is in the support of the orbital integral $I(g,f)$, we have $\Delta_i(gz)=\Delta_i(g)\Delta_i(z)$, $m+1 \leq i \leq n$, is in some fixed compact set in $F^\times$. As $g$ is in the compact set $Q$, we know that $z$ is in a  compact set  in $Z$. As $I(\cdot,f)$ is locally integrable, $I_\omega(\cdot,f)$ is locally integrable.

\end{proof}

In particular, the local integrability of Kloosterman integrals for $G$ implies  the regularity of Bessel distributions for
all irreducible admissible generic representations on $G$.

Denote by $\delta$ the modulus character of $B$. Following \cite{DR98}, we consider a maximal compact subgroup $K_0$ of $G$. Denote by $\CO$ the ring of integers in $F$. Consider the maximal open compact subgroup of $A$ and $N_\alpha$, $\alpha \in \Phi$
\[A_0=\{a\in A:\chi(a)\in\CO^\times\text{ for all }\chi\in X^*(A)\}, \quad N_{\alpha,0} = x_\alpha(\CO).\]
Let $K_0$ be the subgroup of $G$ generated by $A_0$ and all $N_{\alpha,0}$ for $\alpha\in\Phi$. Then $K_0$ is a maximal compact subgroup of $G$.

For each $a\in A$, there exists a unique $\lambda_a\in X_*(A)$ such that 
\[a=a_0\lambda_a(\varpi), \quad a_0 \in A\cap K_0.\] 
This decomposition gives an isomorphism
\[(A \cap K_0) \times X_*(A) \stackrel{\sim}{\lra} A, \quad
(a_0, \lambda) \mapsto a_0\lambda(\varpi).\]

\begin{prop}\label{prop-delta}
	Let $\phi$ be a $(N \times N)$-invariant function on $G$ such that there exist some $\varepsilon > 0$ and $C>0$ satisfying  
	\[ \Big|\phi(a)\delta^{\frac{1}{2}-\varepsilon}(a)\Big| \leq C, \quad a \in A.\] 
    Then $\phi$ is locally integrable on $G$.
\end{prop}

The proof of this Proposition is based on the following result.

\begin{thm}[D\k{a}browski-Reeder, Theorem 0.3 in \cite{DR98}]\label{DR1998}
    Let $a \in A$. Consider the integral
    \[\CO(a) = \int_{N^2} 1_{K_0}( a \cdot(n_1,n_2))dn_1dn_2.\]
    Here $1_{K_0}$ is the characteristic function of the maximal open compact subgroup $K_0$ of $G$. 
    
    If $\lambda_a \notin\BZ_{\geq 0}\Phi^{\vee,+}$, then  $\CO(a) = 0$. If $\lambda_a \in \BZ_{\geq 0}\Phi^{\vee,+}$, then up to a constant,
    \[\CO(a) \stackrel{.}{=} \delta^{-\frac{1}{2}}(a)\sum_{\ov{r}}\left(1-\frac{1}{q}\right)^{\kappa(\ov{r})}.\]
    Here, $\overline{r}= (r_\gamma)_{\gamma \in \Phi^{\vee,+}}$ runs over all possible decompositions for $\lambda_a$ with respect to $\Phi^{\vee,+}$, that is,
    all the decompositions $\lambda_a = \sum_{\gamma \in \Phi^{\vee,+}} r_\gamma \gamma$. For each $\overline{r} = (r_\gamma)_\gamma$, $\kappa(\overline{r})$ is the number of 
    $\gamma \in \Phi^{\vee,+}$ with $r_\gamma > 0$.
\end{thm}

 \begin{remark}\label{results in DR}
    We have $\CO(a)\stackrel{.}{=}\#X(w^0a)$ (See the proof of Lemma \ref{lem-int-sum}), where
    \[X(w^0a)=N\cap K_0\backslash Nw^0aN\cap K_0/N\cap K_0\]
    is the Kloosterman set corresponding to $w^0a$ considered in \cite[Theorem 0.3]{DR98}.
\end{remark}

\begin{proof}[Proof of Proposition \ref{prop-delta}]
    It is sufficient to prove that
    \[\int_G|\phi(g)|1_{K_0g_0}(g)dg<+\infty\]
    for any $g_0 \in B$. The above integral equals
    \[\int_G |\phi(gg_0)|1_{K_0}(g)dg.\]
    As $\phi$ is $(N \times N)$-invariant. We can take $g_0 \in A$. Write $R(g_0)\phi = \phi(\cdot g_0)$.
    Then $R(g_0)\phi$ is also $(N \times N)$-invariant.

    We may assume the Haar measures on $G$ with those on $A$ and $N$ satisfying the following
    integration formula
    \[\int_G f(g)dg = \int_N \int_A \int_N f(\ov{n_1} a n_2) \delta(a)dn_1 da dn_2, \quad f \in C_c(G).\]
    
    By the $(N \times N)$-invariance property of $R(g_0)\phi$, 
    we have
    \[\int_G|R(g_0)\phi(g)|1_{K_0}(g)dg =\int_A|R(g_0)\phi(a)|\CO(a)\delta(a)da \leq C\delta^{1/2+\varepsilon}(g_0)\int_A\delta^{\frac{1}{2}+\varepsilon}(a) \CO(a)da.\]

    Note that $\delta$ and $\CO$ are invariant under the multiplication by $A\cap K_0$. We have
    \[\int_A\delta^{\frac{1}{2}+\varepsilon}(a) \CO(a)da
     = \Vol(A\cap K_0) \sum_{\lambda \in X_*(A)}\delta^{\frac{1}{2}+\varepsilon}(\lambda(\varpi)) \CO(\lambda(\varpi)).\]

     By the result for the support of $\CO$ in Theorem \ref{DR1998},
     \[\sum_{\lambda \in X_*(A)}\delta^{\frac{1}{2}+\varepsilon}(\lambda(\varpi)) \CO(\lambda(\varpi)) = 
     \sum_{\lambda \in \BZ_{\geq 0}\Phi^{\vee,+}}\delta^{\frac{1}{2}+\varepsilon}(\lambda(\varpi)) \CO(\lambda(\varpi)).\]

     Denote by 
     $\Delta^\vee = \{\alpha_i^\vee\}_{1 \leq i \leq m}$ 
     the set of simple coroots. Each $\lambda \in \BZ_{\geq 0} \Phi^{\vee,+}$ can be written uniquely as
     $\lambda = \sum_{i=1}^m r_i \alpha_i^\vee$
     and this gives an isomorphism
     \[\BZ_{\geq 0} \Phi^{\vee,+} \stackrel{\sim}{\lra} \BZ_{\geq 0}^m, \quad \lambda \mapsto (r_i)_{1 \leq i \leq m}.\]

     By the following Lemma \ref{lem-delta-formula}
     \[\delta(\lambda(\varpi)) = q^{-2\sum_i r_i}, \quad 
     \lambda = \sum_{i=1}^m r_i \alpha_i^\vee \in 
     \BZ_{\geq 0} \Phi^{\vee,+}.\]
     On the other hand, 
     by the following Lemma \ref{estamite on Kostant partition function}, there exists a polynomial 
     $P(x_1,\ldots,x_m) \in \BZ_{\geq 0}[x_1,\ldots,x_m]$ such that
     \[\CO(\lambda(\varpi)) \leq \delta^{-1/2}(\lambda(\varpi))P(r_1,\ldots,r_m), \quad 
     \lambda = \sum_{i=1}^m r_i \alpha_i^\vee \in 
     \BZ_{\geq 0} \Phi^{\vee,+}.\]

      Therefore,
    \[\sum_{\lambda \in \BZ_{\geq 0}\Phi^{\vee,+}}\delta^{\frac{1}{2}+\varepsilon}(\lambda(\varpi)) \CO(\lambda(\varpi))
    \leq \sum_{(r_i)_i \in \BZ_{\geq 0}^m} 
    q^{-2\varepsilon\sum_i r_i} P(r_1,\ldots,r_m).\]

   Consider the following function on $\BR^m$
   \[f(x_1,\cdots,x_m):=q^{-2\varepsilon\sum_{1\leq i\leq m}x_i}\cdot P(x_1,\cdots,x_m).\]
   It is a Schwartz function on $\BR_{\geq 0}^m$.
   The local integrability of $\phi$ follows from 
   \[\sum_{(r_i)_i \in \BZ_{\geq 0}^m} f(r_1,\ldots,r_m) < \infty.\]
\end{proof}

\begin{lem}\label{lem-delta-formula}
    For each coroot  $\lambda = \sum_{i=1}^m r_i \alpha_i^\vee$, 
    we have
     \[\delta(\lambda(\varpi)) = q^{-2\sum_i r_i}.\]
\end{lem}
\begin{proof}
   By definition, we have
   \[\delta(\lambda(\varpi)) = q^{-\langle 2\rho,\lambda \rangle}
   = q^{-\sum_i \langle 2\rho,\alpha_i^\vee \rangle r_i}\]
   with $2\rho = \sum_{\beta \in \Phi^{+}} \beta$. Let $\alpha$
   be a simple root. Then the simple reflection $s_\alpha$ satisfies
   \[s_\alpha(\Phi^+) = (\Phi^+ \setminus \{\alpha\}) \cup \{-\alpha\}.\]
   Therefore, 
   \[s_\alpha(2\rho) = \sum_{\beta \in \Phi^+} s_\alpha(\beta) = 
   \sum_{\beta \in \Phi^+ \setminus \{\alpha\}} s_\alpha(\beta) - \alpha = \sum_{\beta \in \Phi^+ \setminus \{\alpha\}} \beta - \alpha = 2\rho-2\alpha.\]
   As $s_\alpha(2\rho) = 2\rho  - \langle 2\rho,\alpha^\vee \rangle \alpha$, we have $\langle 2\rho,\alpha^\vee \rangle = 2$ and
   the Lemma holds.
\end{proof}

Consider the Kostant partition function
\[R(\lambda)=\#\left\{\ov{r}=(r_\gamma)_{\gamma\in\Phi^{\vee,+}}:\lambda=\sum_{\gamma\in\Phi^{\vee,+}}r_\gamma\gamma\right\}, \quad \lambda \in \BZ_{\geq 0}\Phi^{\vee,+}.\]
\begin{lem}\label{estamite on Kostant partition function}
    The Kostant partition function $R$ is bounded by a polynomial 
    $P(x_1,\ldots,x_m) \in \BZ_{\geq 0}[x_1,\ldots,x_m]$, that is,
    \[R(\lambda)\leq P(r_1,\cdots,r_m), \quad \lambda = \sum_{i=1}^m r_i \alpha_i^\vee \in 
     \BZ_{\geq 0} \Phi^{\vee,+}.\]
     In particular, by Theorem \ref{DR1998},  
     \[\CO(\lambda(\varpi)) \leq \delta^{-1/2}(\lambda(\varpi))P(r_1,\ldots,r_m), \quad 
     \lambda = \sum_{i=1}^m r_i \alpha_i^\vee \in 
     \BZ_{\geq 0} \Phi^{\vee,+}.\]
    \begin{proof}
        In fact, we can take the following polynomial
        \[P(x_1,\cdots,x_m)=\left(1+\sum_{1\leq i\leq m}x_i\right)^{|\Phi^{\vee,+}|}.\] 
        To see this, write each $\gamma \in \Phi^{\vee,+}$ as $\gamma=\sum_{1\leq i\leq m}c_{\gamma,i}\alpha_i^\vee$ with $c_{\gamma,i}\in\BZ_{\geq 0}$. Therefore, for each 
        $\lambda = \sum_i r_i\alpha^\vee_i \in \BZ_{\geq 0} \Phi^{\vee,+}$, a partition
        $\lambda = \sum_\gamma r_\gamma \gamma$ satisfies the following equations
        \[\sum_{\gamma\in\Phi^{\vee,+}}c_{\gamma,i}r_\gamma=r_i, \quad 1\leq i\leq m.\]
        Then we have
        \[\sum_{\gamma\in\Phi^{\vee,+}}\left(\sum_{1\leq i\leq m}c_{\gamma,i}\right)r_\gamma=\sum_{1\leq i\leq m}r_i.\]
        Since at least one $c_{\gamma,i}\geq 1$ for every $\gamma$, we have $\sum_{1\leq i\leq m}c_{\gamma,i}\geq 1$ and hence
        \[0 \leq r_\gamma\leq\left(\sum_{1\leq i\leq m}c_{\gamma,i}\right)r_\gamma \leq \sum_{1\leq i\leq m}r_i.\]
        This implies
        \[R(\lambda)\leq\left(1+\sum_{1\leq i\leq m}r_i\right)^{|\Phi^{\vee,+}|}.\]
        Therefore, the Kostant partition function $R$ is bounded 
        by a polynomial $P$.

        Now, for any $\lambda = \sum_{i=1}^m r_i \alpha_i^\vee \in 
     \BZ_{\geq 0} \Phi^{\vee,+}$, by Theorem \ref{DR1998}, we have
     \[\delta^{1/2}(\lambda(\varpi))\CO(\lambda(\varpi)) = \sum_{\ov{r}}\left(1-\frac{1}{q}\right)^{\kappa(\ov{r})} \leq
     \sum_{\ov{r}} 1 =
     R(\lambda) \leq P(r_1,\ldots,r_m).\]
     We are done.
     
    \end{proof}
\end{lem}
\begin{cor}\label{bound of orbital integral implies regularity}
	Assume that for any $f\in \CS(G)$, there exist some $\varepsilon > 0$ and $C>0$, such that  
	\[ \Big|I(a,f)\delta^{\frac{1}{2}-\varepsilon}(a)\Big| \leq C, \quad a \in A.\] 
    Then for any irreducible admissible generic representation $\pi$ of $G$, the Bessel distribution $B_\pi$ is regular.
\end{cor}
\begin{proof}
   Apply Proposition \ref{prop-delta}, the orbital integral $I(g,f)$ is a locally integrable function. Hence by Theorem \ref{locally orbital integral 2}, the function $j_\pi(g)$ is locally integrable function for any $\pi$. This implies $B_\pi$ is regular by Proposition \ref{locally integrable implies regular}.
\end{proof}
\begin{prop}\label{K implies I(a,f)}
    Assume $\{K_I^J\}_{I\subset J\subset\Delta}$ such that for any $J\subset\Delta$,
    \begin{enumerate}
        \item The functions $K_\emptyset^J$ is supported on the set $A_\emptyset^J(r_J)$ for some $r_J>0$.
        \item For $r>0$ small enough, there exists $\varepsilon_J,C_J$ such that
        \[\left|K_\emptyset^J(a)\delta_{B_{M_J}}^{\frac{1}{2}-\varepsilon_J}(a)\right|\leq C_J,\]
        for any $a\in A_\emptyset^J(r)$.
    \end{enumerate}
    Then for any $f\in\CS(G)$, there exists $\varepsilon>0,C>0$ such that
    \[\left|I(a,f)\cdot\delta^{\frac{1}{2}-\varepsilon}(a)\right|\leq C, \quad a \in A.\]
    \begin{proof}
        By the conditions, we have the germ expansion
        \[I(a,f)=\sum_{J\subset\Delta}(K_\emptyset^J*\omega_J)(a),\quad a\in A,\]
        where the function $K_\emptyset^J$ is supported on the set $A_\emptyset^J(r_J)$ for some $r_J>0$. For each $J \subset \Delta$, note that for any $r' > r > 0$, the set 
        \[\left\{a \in A_\emptyset^J \Big| r \leq |\lambda_i(a)| \leq r', \forall
        \alpha_i \in J\right\}\]
        is compact. Therefore, there exists some $\varepsilon_J>0$ such that $\Big|K_\emptyset^J(\cdot)\delta_{B_{M_J}}^{\frac{1}{2}-\varepsilon_J}(\cdot)\Big|$ is  bounded on the set $A_\emptyset^J$. Then we take $\varepsilon=\min_{J\subset\Delta}\varepsilon_J$ and we have
        \[\Big|I(a,f)\delta^{\frac{1}{2}-\varepsilon}(a)\Big|\leq\sum_{J\subset\Delta}\Big|(K_\emptyset^J*\omega_J)(a)\delta^{\frac{1}{2}-\varepsilon_J}(a)\Big|.\]
        For each $J\subset\Delta$,
        \[\Big|(K_\emptyset^J*\omega_J)(a)\delta^{\frac{1}{2}-\varepsilon_J}(a)\Big|\leq\sum_{a=bc}|K_\emptyset^J(b)\delta_{B_{M_J}}^{\frac{1}{2}-\varepsilon_J}(b)|\cdot\Big|\omega_J(c)\delta^{\frac{1}{2}-\varepsilon_J}(c)\Big|,\]
        where $b\in A_\emptyset^J,c\in A_J$. Here we use the fact that for $J\subset\Delta$, the function $\delta|_{A_\emptyset^J}=\delta_{B_{M_J}}$ on $A_\emptyset^J=A\cap M_J^\der$. Since $\omega_J\in\CS(A_J)$, the function $|\omega_J(c)\delta^{\frac{1}{2}-\varepsilon_J}(c)\Big|$ is bounded on $A_J$. Combine with
        \[\Big|K_\emptyset^J(a)\delta_{B_{M_J}}^{\frac{1}{2}-\varepsilon_J}(a)\Big| \leq C_J,\]
        there exists constant $C$ such that
        \[\Big|I(a,f)\delta^{\frac{1}{2}-\varepsilon}(a)\Big|\leq C,\quad a\in A.\]
    \end{proof}
\end{prop}

\subsection{Kloosterman sums}
\begin{defn}\label{defn-sys-Kl}
    Let $G$ be  a split reductive group. We fix a Borel subgroup $B=AN$. Consider the following triple
    $\left(\psi,\dot{w_0},K\right)$:
    \begin{itemize}
        \item $\psi$ - a generic character on $N$;
        \item $\dot{w_0}$ - a representative for the longest
        Weyl element $w_0 \in W$;
        \item $K$ is a compact open subgroup of $G$ such that $\psi|_{N\cap K}=1$.
    \end{itemize}
    The Kloosterman sum on $G$ with respect to
        the triple $\left(\psi,\dot{w_0},K\right)$ is the function
        \[\mathrm{Kl}(a) =\Kl(a;\psi,\dot{w_0},K) = \sum_{x \in X_K(\dot{w_0}a)} \psi(u(x))\psi(u'(x)), \quad a \in A.\]
        Here, we consider the finite set
        \[X_K(\dot{w_0}a) = (N\cap K)\bs \left(N\dot{w_0}a N \cap K \right)/(N\cap K)\]
        with the following two maps
        \[u: X_K(\dot{w_0}a) \ra (N\cap K)\bs N, \quad u': X_K(\dot{w_0}a) \ra N/(N\cap K).\]
        Here, if $[n_1\dot{w_0}an_2] \in X_K(\dot{w_0}a)$ with $n_1, n_2 \in N$ and $a \in A$, then 
        \[u([n_1\dot{w_0}an_2]) = [n_1], \quad u'([n_1\dot{w_0}an_2]) = [n_2].\]
        The Kloosterman sum $\mathrm{Kl}$ is said to have a nontrivial bound if there exists $\varepsilon>0$ and $C>0$ such that
        \[\left|\mathrm{Kl}(a)\cdot\delta^{1/2-\varepsilon}(a)\right| \leq C, \quad a \in A^\der\]
        where $A^\der = A_\emptyset^\Delta = A\cap G^\der$ is the maximal torus of $G^\der$.
\end{defn}

\begin{example}\label{DR example}
    Suppose $G$ is a general split reductive group, $\dot{w_0} = w^0$ the longest Weyl element in $W$ defining the Kloosterman integrals and $\psi=\psi_0$. Let $K_0$ be the maximal open compact subgroup in Remark \ref{results in DR}. 
    For each $a \in A$, $X(\dot{w_0}a)$ is the Kloosterman set with respect to $K_0$. In this example, we consider the Kloosterman sum $\mathrm{Kl}$ with respect to the triple $(\psi,\dot{w_0},K_0)$.

    By  Theorem \ref{DR1998} (\cite[Theorem 0.3]{DR98}), we have
    \[\#X(\dot{w_0a})=\begin{cases}
        0,&\text{ if }\lambda_a \notin \BZ_{\geq 0}\Phi^{\vee,+}\\
        \delta^{-\frac{1}{2}}(a)\sum_{\ov{r}}\left(1-\frac{1}{q}\right)^{\kappa(\ov{r})}&\text{ if }\lambda_a \in \BZ_{\geq 0}\Phi^{\vee,+}
    \end{cases}.\]
    In the above, the terms $\overline{r},\kappa(\overline{r})$ are given in Theorem \ref{DR1998}.
    

    Moreover, if $\lambda_a \in \BZ_{\geq 0}\Phi^{\vee,+}$, then $\lambda_a$ can be written uniquely as $\lambda_a= \sum_{i=1}^m r_{a,i}\alpha_i^\vee$ and by Lemma \ref{estamite on Kostant partition function}
    \[\sum_{\ov{r}}\left(1-\frac{1}{q}\right)^{\kappa(\ov{r})}< R(\lambda_a)\leq \left(1+\sum_{1\leq i\leq m}r_{a,i}\right)^{|\Phi^{\vee,+}|}.\]
    Hence,  
    \[\left|\mathrm{Kl}(a)\cdot\delta^{1/2}(a)\right| \leq \left(1+\sum_{1\leq i\leq m}r_{a,i}\right)^{|\Phi^{\vee,+}|}, \quad a \in A, \quad
    \lambda_a \in \BZ_{\geq 0} \Phi^{\vee,+}\]
    may be regarded as a trivial bound for the corresponding Kloosterman sum.
\end{example}

We list some basic facts for Kloosterman sums. Some proofs
will be omitted.

\begin{lem}\label{lem-int-sum}
    Consider $\dot{w_0} =w^0$ with $w^0$ the
    longest Weyl element in $W$ defining the Kloosterman 
    integrals, $\psi = \psi_0^{-1}$ and an open compact subgroup $K$. Then up to a constant, we have
    \[I(a,1_{w^0K}) \stackrel{.}{=} \mathrm{Kl}((w^0)^{-2}a; \psi,w^0,K), \quad
    a \in A.\]
\end{lem}
\begin{proof}
    We have
    \begin{align*}
        I(a,1_{w^0K})&=\int_{N\times N}1_{w^0K}(\ov{n_1}^{-1}an_2)\psi_0^{-1}(n_1^{-1}n_2)dn_1dn_2\\
        &=\int_{N\times N}1_{w^0K}(\ov{n_1}an_2)\psi_0^{-1}(n_1n_2)dn_1dn_2
    \end{align*}
    Note that $\ov{n_1}an_2\in w^0K$ if and only if
    $n_1w^0((w^0)^{-2}a)n_2 \in K$. This implies that up to a constant
    depending on $K$,
    \[I(a,1_{w^0K}) \stackrel{.}{=} \sum_{x\in X_K(w^0(w^0)^{-2}a)}\psi_0^{-1}(u(x))\psi_0^{-1}(u'(x))=\mathrm{Kl}((w^0)^{-2}a).\]
\end{proof}

Let $G_1,G_2$ be two split reductive groups. We fix a Borel subgroup $B_i = A_i N_i$ for $G_i$ with $i=1,2$. Suppose we are given triples $(\psi_1,\dot{w_{1,0}},K_1)$ and $(\psi_2,\dot{w_{2,0}},K_2)$ for $G_1$ and $G_2$ respectively.

Consider $G = G_1\times G_2$. The group $B=B_1\times B_2=AN$ is a Borel subgroup of $G$ with $A=A_1\times A_2$ and  $N=N_1\times N_2$.  Let $\psi = \psi_1 \boxtimes \psi_2$
be a generic character of $N$. The triple $(\psi,(\dot{w_{1,0}},\dot{w_{2,0}}),K_1\times K_2)$ defines a triple for $G$.

\begin{lem}\label{Kloosterman sum of product group}
    We have
    \[\Kl(a;\psi,(\dot{w_{1,0}},\dot{w_{2,0}}),K_1\times K_2)=\Kl(a_1;\psi_1,\dot{w_{1,0}},K_1)\cdot\Kl(a_2;\psi_2,(\dot{w_{2,0}},K_2)), \quad a=(a_1,a_2)\in A.\]
\end{lem}

The following lemma gives the dependence of $\psi$ and $\dot{w_0}$ for the Kloosterman sums.
\begin{lem}\label{properties for Kloosterman sums}\
    \begin{enumerate}
        \item Let $a' \in A$. Let $\psi$ be a generic character on $N$ with $\psi' = \psi(\Ad(a')\cdot)$ and $\dot{w_0} =w^0$ the
    longest Weyl element in $W$ defining the Kloosterman 
    integrals. Consider the triples $(\psi',\dot{w_0},K)$ and $(\psi,\dot{w_0},(a')^{-1}Ka')$. Then we have the relation
        \[\mathrm{Kl}(a;\psi',\dot{w_0},K)=\mathrm{Kl}\left(\ov{a'}^{-1}aa';\psi,\dot{w_0},(a')^{-1}Ka'\right), \quad a \in A.\]
        \item Let $a' \in A$ and $\dot{w_0}'=\dot{w_0}a'$ be a representative of $w_0$. Consider the  triples $(\psi,\dot{w_0},K)$ and $(\psi,\dot{w_0}',K)$. Then we have the relation
        \[\mathrm{Kl}(a;\psi,\dot{w_0},K)=\mathrm{Kl}(a'a;\psi,\dot{w_0}',K), \quad a \in A.\]
    \end{enumerate}
\end{lem}

\begin{proof}
     We give a proof of (1). We have
     \[\mathrm{Kl}(a;\psi',\dot{w_0},K)=\sum_{x\in X_K(\dot{w_0}a)}\psi'(u(x))\psi'(u'(x))=\sum_{x\in X_K(\dot{w_0}a)}\psi(\Ad(a')u(x))\psi(\Ad(a')u'(x)).\]
     The condition $x=u(x)\dot{w_0}\ov{a'}^{-1}aa'u'(x)\in a'^{-1}Ka'\cap N\dot{w_0}\ov{a'}^{-1}aa'N$ is equivalent to
     \[\Ad(a')u(x)\dot{w_0}a\Ad(a')u'(x)\in K\cap N\dot{w_0}aN.\]
     Hence,
     \[\sum_{x\in X_K(\dot{w_0}a)}\psi(\Ad(a')u(x))\psi(\Ad(a')u'(x))=\sum_{x\in X_{(a')^{-1}Ka}(\dot{w_0}\ov{a'}^{-1}aa')}\psi(u(x))\psi(u'(x))\]
     and this is equal to $\mathrm{Kl}(\ov{a'}^{-1}aa';\psi,\dot{w_0},(a')^{-1}Ka')$.

\end{proof}

We consider Kloosterman sums on a split group $G$ with
those on its derived group $G^\der$. 
Let $\psi$ be a generic character of $N$.  Let
$\dot{w_0} \in G^\der$ be a representative of the longest Weyl element $w_0 \in W$ (for example, we may choose the
Tits representative of $w_0$).

\begin{lem}\label{kloosterman sum for derived subgroup}
    We have the following.
    \begin{enumerate}
        \item For an admissible group $K$ of $G$ (with respect to $\psi$), the subgroup $K\cap G^\der$ is an admissible group of $G^\der$.  Conversely, given an admissible group $K$ of $G^\der$, there exists an admissible group $\wt{K}$ of $G$ such that $\wt{K}\cap G^\der= K$.
    \item Let $K$ be an admissible group of $G$. Consider the Kloosterman sums $\Kl_G(\cdot;\psi,\dot{w_0},K)$ and $\Kl_{G^\der}(\cdot;\psi,\dot{w_0},K\cap G^\der)$ on $G$ and $G^\der$. We have
    \[\Kl_G(a;\psi,\psi',\dot{w_0},K)=\Kl_{G^\der}(a;\psi,\psi',\dot{w_0},K\cap G^\der),\quad a\in A^\der.\]
    \end{enumerate}
    \begin{proof}
        We only prove the existence of $\wt{K}$ in (1). Since $Z^\der = Z\cap G^\der$ is finite, we can choose a small enough open compact subgroup $L\subset Z$ such that
        \[L \cap G^\der = L\cap(Z\cap G^\der)=\{1\}.\]
        Consider $\wt{K} = KL$. Then $\wt{K}$ is an 
        open compact subgroup of $G$. We claim $\wt{K}$
        is admissible. In fact, let $z \in Z^\der$ and 
        $n \in N$ with $zn \in \wt{K}$. Then, there exists
        $\ell \in L$ such that $\ell z n \in K$. Then
        $\ell \in L \cap G^\der = \{1\}$ and
        $zn \in K$. As $K$ is admissible, $z = 1$ and $n \in K$. Therefore, $\wt{K}$ is admissible.

    \end{proof}
\end{lem}

Now, we consider the relationship of Kloosterman sums under
isogeny maps. 

Let $G'$ and $G$ be two split groups over $F$ with an isogeny
\[\pi: G' \lra G.\]
We fix a Borel subgroup $B' = A'N'$ of $G'$. Denote by
$A = \pi(A')$ and $N = \pi(N')$. Then $A$ is a maximal
torus of $G$, $\pi|_{N'}$ gives an isomorphism of $N'$ and $N$
and $B = AN$ is a Borel subgroup of $G$. 

Let $\psi'$ be a generic character of $N'$, then $\psi'$ 
induces a generic character $\psi$ on $N$. Let $\dot{w_0}'$
be a representative of the longest Weyl element in the
Weyl group $N_{G'}(A')/A'$ of $G'$. Then $\dot{w_0} = \pi(\dot{w_0}')$ gives a representative of the longest 
Weyl element of $G$.

\begin{lem}\label{Kloosterman sum under isogeny map}
    Assume $G'$ is semisimple (so that $G$ is also semisimple).
    We have the following
    \begin{enumerate}
    \item Let $K'$ be an admissible group of $G'$ (with respect 
    to $G'$). Then $K = \pi(K')$ is an admissible group
    of $G$ and $\pi|_{K'}$ gives an isomorphism from $K'$ to $K$. 
    \item  Let $a\in A$ such that the Kloosterman set $X_{K}(\pi(w_0')a)\neq\emptyset$. Then there exist $a'\in A'$ such that $\pi(a')=a$. In this case, we have
    \[\Kl_{G'}(a';\psi',\dot{w_0}',K')=\Kl_{G}(a;\psi,\dot{w_0},K).\]
    \end{enumerate}
    \begin{proof}
        
(1) The isogeny map $\pi$ is an \etale morphism so that
it induces a locally analytic isomorphism between the
$F$-points of $G'$ and $G$. In particular, $K$ is an open compact subgroup
of $G$. As $\ker\pi \subset Z'$, $K' \cap \ker\pi \subset K' \cap Z' = 1$ as we require $G'$ is semisimple and $K'$ is admissible. Therefore, $\pi|_{K'}$ gives an isomorphism of
$K'$ and $K$. 

We now prove that $K$ is admissible. Let $z \in Z$, $n \in N$
with $zn \in K$. We want to show that $z = 1$ and $n \in K$.
As $\pi|_{N'}$ is an isomorphism (of varieties and also of
$F$-points of varieties), there exists $n' \in N'$ such that
$\pi(n') = n$. Then $z \in \pi(G') \cap Z$. Note that,
as $F$-points, $\pi(G') \cap Z = \pi(Z')$ so that there
exists $z' \in Z'$ such that $z = \pi(z')$. As
$zn \in K$, there exists $u \in \ker\pi$ 
such that $uz'n' \in K'$. As $K'$ is admissible, $uz' = 1$
and $n' \in K'$. In particular, 
$z = \pi(z') = \pi(uz') = 1$ and $n = \pi(n') \in K$.

(2) Assume $a \in A$ with nonempty $X_K(\dot{w_0}a)$. In 
particular, $N\pi(\dot{w_0})aN \cap K \not= \emptyset$. 
Therefore, as $F$-points, $a \in \pi(G')\cap A = \pi(A')$
so that there exists $a' \in A'$ with $a = \pi(a')$. In this
case, the
bijection $\pi|_{K'}$ induces a bijection between the
two Kloosterman sets $X_{K'}(\dot{w_0}'a')$ and 
$X_K(\dot{w_0}a)$ and we have
 \[\Kl_{G'}(a';\psi',\dot{w_0}',K')=\Kl_{G}(a;\psi,\dot{w_0},K).\]
    \end{proof}
\end{lem}

\subsection{Kloosterman sums and the regularity}

The following is one of our main result. The nontrivial bound for
Kloosterman sums on each Levi subgroup of $G$ associated to admissible open compact subgroups (See Definition \ref{defn-adm}) implies the regularity of Bessel distributions
on $G$.

\begin{thm}\label{main theorem}
    Let $G$ be a split reductive group. We fix a Borel subgroup $B=AN$. Moreover, we fix 
    \begin{itemize}
        \item a family of isomorphisms $x_\alpha: F \stackrel{\sim}{\ra} N_\alpha$, $\alpha \in \Phi$ and consider the generic character $\psi_0$;
        \item a system of representatives $\{\dot{w_I}\}_{I \subset \Delta}$ of $\{w_I\}_{I \subset \Delta}$;
        \item for each standard Levi subgroup $M_I,I\subset\Delta$, an admissible compact open subgroup $K_I$ of $M_I$ (with
        respect to $\psi_0|_{N_{M_I}}^{-1}$).
    \end{itemize}

    Assume that for each $I \subset \Delta$, the 
    Kloosterman sum $\Kl_I$ on $M_I$ with respect to the triple 
    \[(\psi_0|_{N_{M_I}}^{-1},\dot{w_I},K_{I})\]
    has a nontrivial bound. 

    Then, for any irreducible smooth admissible representation $\pi$ on $G$ which is generic with respect to $\psi_0$, the Bessel distribution $B_\pi$ is regular.
\end{thm}
\begin{proof}
    By Lemma \ref{properties for Kloosterman sums} (2), we can assume 
    Kloosterman sum $\Kl_I$ on $M_I$ with respect to the triple 
    \[(\psi_0|_{N_{M_I}}^{-1},w_I^0,K_{I})\]
    has a nontrivial bound, where $\{w_I^0\}_{I\subset\Delta}$ is an system of representatives of $\{w_I\}_{I \subset \Delta}$ relevant to $\psi_0$ (for example, the Tits representatives in the proof of
    Lemma \ref{relevant weyl elements}).
    
    By Corollary \ref{bound of orbital integral implies regularity}, we need to prove that for any $f \in \CS(G)$, there exist some $\varepsilon > 0$ and $C>0$, such that  
	\[ \Big|I(a,f)\delta^{\frac{1}{2}-\varepsilon}(a)\Big| \leq C, \quad a \in A.\]
    
    Apply Theorem \ref{shalika}, we have the following germ expansion of Kloosterman integrals  
    \[I(a,f) =\sum_{J\subset\Delta}(K_\emptyset^{J}*\omega_J)(a),\quad a\in A,\]
    where $\omega_J\in\CS(A_J)$ for every $J\subset\Delta$ and $\{K_I^J\}_{I\subset J}$ is a system of Shalika germs.
    
    By Proposition \ref{prop-germ-adm}, for each $J$, we can choose  the germ function $K_\emptyset^J$ such that
    \[K_\emptyset^{J}(a) =I(a,1_{w_J^0K_{J}}) \stackrel{.}{=} \Kl_J((w_J^0)^{-2}a),\quad a \in A_\emptyset^J.\]
    In particular, $K_\emptyset^J$ is  supported on $A_\emptyset^J(r_J)$ for some $r_J>0$. 
    
    By the above germ expansion and Proposition \ref{K implies I(a,f)}, the local integrability
    of $I(\cdot,f)$ holds provided that for every $J\subset\Delta$, for any $r>0$ small enough, there exist some constants $\varepsilon > 0$ and $C_J>0$, such that
    \[|K_\emptyset^J(a)\delta_{B_{M_J}}^{\frac{1}{2}-\varepsilon}(a)|\leq C_J,\quad a\in A_\emptyset^J(r).\]
    
    By our assumption on the Kloosterman sum $\Kl_J$, such 
    a nontrivial bound holds for $K_\emptyset^J$. The proof is done.
\end{proof}

\section{Kloosterman sums for groups of type $A$}

\subsection{Kloosterman sums for $\GL_2$}

Here, we consider Kloosterman sums for $\GL_2$ over a
$p$-adic field $F$. This result is well known to experts and is included here for the convenience of the reader.

We keep notations as in the beginning of Section \ref{notations}. We fix an additive nontrivial character $\psi_F$ on $F$. We assume that $\psi_F$ is unramified.

Let $d$, $r$ be two integers with $0 \leq d < r$. For any $\nu,\nu'\in\CO-\{0\}$, consider the restricted Kloosterman sum
\[\Kl(\nu,\nu';d,r)=\sum_{x\in  (1+\fp^{d})\cap \CO^\times /1+\fp^{r}}\psi_F\left(\frac{\nu x+\nu'x^{-1}}{\varpi^{r}}\right).\]

Firstly, we consider the case $d$ is sufficiently large. In particular, we assume that the map 
\[1 + \fp^d \lra 1+ 2\fp^d, \quad z \mapsto z^2\]
is an analytic bijection. The inverse bijection will be denoted by a square root.


\begin{prop}\label{GL2 kloosterman sum}
    For $d > 0$ sufficiently large, we have the following Weil bound
    \[|\Kl(\nu,\nu';d,r)| \leq q^{\frac{k}{2}}q^{\frac{\min\{v(\nu),v(\nu')\}+r}{2}}.\]
    Here, we fix $k \geq 0$ such that $q^{-k} \leq |2|^2$.
\end{prop}

This Weil bound can be obtained using the existence of the Weil constant (See \cite[Section 3]{JY99}). 

For any $a \in F^\times$, there is a unique constant $\gamma(a,\psi)$, called the Weil constant associated to $a$ and $\psi$, satisfying  
    \[\int_F \Phi(x)\psi\left(\frac{ax^2}{2}\right)dx=|a|^{-\frac{1}{2}}\gamma(a,\psi)\int_F \wh{\Phi}(x)\psi\left(-\frac{x^2}{2a}\right)dx, 
\quad \Phi \in \CS(F).\]
Here, $\wh{\Phi}(x) = \int_F \Phi(y)\psi(-xy)dy$ is the Fourier transform of $\Phi$.   Moreover, we have $|\gamma\left(a,\psi\right)|=1$.

\begin{prop}[Jacquet-Ye, Lemma 3.2, Lemma 3.3, Lemma 3.4 in \cite{JY99}]\label{stationary phase}
    For $d > 0$ sufficiently large, consider
    \[T^d(a,b)=\int_{1+\mathfrak{p}^d}\psi\left(ax+\frac{b}{x}\right)dx,\]
    here the Haar measure $dx$ on $F$ is normalized with the volume of $\CO$ equal to $1$.
    Assume $(a,b)\in\operatorname{supp}T^d$. Then the following properties hold
    \begin{enumerate}
        \item For any $k \geq 0$, we have $|b|\leq q^{2d+k}$ if and only if $|a|\leq q^{2d+k}$. Under this condition, we have $|a-b|\leq q^{d+k}$.
        \item Let $k\geq0$ such that $q^{-k}\leq|2|^2$. If $|a|,|b|>q^{2d+k}$, 
		then $a=bu^2$ with $u\in1+\mathfrak{p}^d$. Furthermore, the integral
        \[T^d(a,b)=|2a|^{-1/2}\psi\left(\frac{2a}{u}\right)\gamma\left(\frac{a}{u},\psi\right),\]
         In particular,
        \[|T^d(a,b)|=|2a|^{-1/2}.\]
    \end{enumerate}
\end{prop}

\begin{proof}[The proof of Proposition \ref{GL2 kloosterman sum}]
    By the symmetry of $x$ and $x^{-1}$, it is 
    enough to show that 
    \[|\Kl(\nu,\nu';d,r)| \leq q^{\frac{k}{2}}q^{\frac{v(\nu')+r}{2}}.\]
    
    We have
        \[\Kl(\nu,\nu';d,r) = q^rT^d\left(\frac{\nu}{\varpi^{r}},\frac{\nu'}{\varpi^{r}}\right).\]
     
        We fix $k>0$ such that $q^{-k}\leq|2|^2$. Then for $|\frac{\nu}{\varpi^{r}}|>q^{2d+k}$, by Lemma \ref{stationary phase}(2), we have $|\nu|=|\nu'|$ and
        \[|\Kl(\nu,\nu';d,r)|=q^r\left|T^d\left(\frac{\nu}{\varpi^{r}},\frac{\nu'}{\varpi^{r}}\right)\right|=q^r\left|\frac{2\nu'}{\varpi^{r}}\right|^{-\frac{1}{2}}=q^{\frac{r}{2}}|2|^{-\frac{1}{2}}|\nu'|^{-\frac{1}{2}}\leq q^{\frac{r+k}{2}}|\nu'|^{-\frac{1}{2}}.\]
        On the other hand, for $|\frac{\nu}{\varpi^{r}}|\leq q^{2d+k}$, by Lemma \ref{stationary phase}(1), we have
        $|\nu'|^{-1/2}\geq q^{-\frac{d-r+k}{2}}$ so that
        \[|\Kl(\nu,\nu';d,r)|=q^{r}\left|T^d\left(\frac{\nu}{\varpi^{r}},\frac{\nu'}{\varpi^{r}}\right)\right|
            \leq q^{r-d} \leq q^{\frac{r+k}{2}}\cdot|\nu'|^{-\frac{1}{2}}.\]
     
\end{proof}

We also consider the case $d=0$
\[\mathrm{Kl}(\nu,\nu';r)= \mathrm{Kl}(\nu,\nu';0,r) = \sum_{x\in\CO^\times/1+\fp^{r}}\psi_F\left(\frac{\nu x+\nu'x^{-1}}{\varpi^r}\right).\]
\begin{cor}\label{classical kloosterman sum}
    There is a constant $C$ only depending on $F$ such that for any $r > 0$ sufficiently large
    \[|\mathrm{Kl}(\nu,\nu';r)| \leq C q^{\frac{\min\{v(\nu),v(\nu')\}+r}{2}}\]
    \begin{proof}
        Choose a constant $d>0$ sufficiently large as above.
        For any $r > d$, by Proposition \ref{GL2 kloosterman sum}, we have 
        \begin{align*}
            |\mathrm{Kl}(\nu,\nu';r)|&=\Big|\sum_{x\in\CO^\times/1+\fp^{r}}\psi_F\left(\frac{\nu x+\nu'x^{-1}}{\varpi^r}\right)\Big|\\
            &\leq\sum_{y\in\CO^\times/1+\fp^d}\Big|\sum_{x\in 1+\fp^d/1+\fp^{r}}\psi_F\left(\frac{\nu xy+\nu'(xy)^{-1}}{\varpi^{r}}\right)\Big|\\
            &\leq (q-1)q^{d+k/2-1} q^{\frac{\min\{v(\nu),v(\nu')\}+r}{2}}.
        \end{align*}
    \end{proof}
\end{cor}
\subsection{The stratification of D\k{a}browski-Reeder}

Here, we consider the stratification of D\k{a}browski-Reeder
for Kloosterman sets on $G = \GL_n$, especially the refinement
in the paper of Blomer-Man \cite{BM24}. 
Let $B=AN$ be the standard upper triangular Borel subgroup of $G$ with $A$ the diaganal torus and $N$ the upper triangular unipotent matrices.

We shall follow the convention of \cite{BM24}. Note that this is slightly different from \cite{DR98} (and Theorem
\ref{DR1998}). For example, for an open compact subgroup $K$ of $G$, the Kloosterman set here looks like
\[N \cap K \bs NawN \cap K / N\cap K\]
with the Weyl element $w$ in the right side of $a \in A$.

Firstly, we consider the stratification for $K = K_0 = G(\CO)$.
Denote by $N_0 = K_0 \cap N$.

Let $\Delta=\{\alpha_{i}:1\leq i\leq n-1\}$ be the set of simple roots
and let $\Phi^+=\{\alpha_{ij}:=\sum_{k=i}^j\alpha_k:1\leq i \leq j\leq n-1\}$
be the set of positive roots.  For each positive root $\beta$ there is an embedding
\[\phi_\beta\colon\SL_2\to G.\]

For a positive root $\beta$, consider the following map
\[b_\beta: F \lra G\]
given by
\[b_\beta(a)=\begin{cases}
    \phi_\beta\begin{pmatrix}
        0 & -1\\
        1 & a
    \end{pmatrix},&\text{ if }a\in\CO;\\
    \phi_\beta\begin{pmatrix}
        c^{-1} & \\
        \varpi^m & c
    \end{pmatrix},&\text{ if }a=\varpi^{-m}c,c\in\CO^\times,m>0.
\end{cases}\]

Consider the following reduced representation of the longest
Weyl element as products of simple reflections
\[w_0=(s_{\alpha_1}\cdots s_{\alpha_{n-1}})(s_{\alpha_1}\cdots s_{\alpha_{n-2}})\cdots(s_{\alpha_1}s_{\alpha_2})s_{\alpha_1}.\]
Write
\[\underline{\beta}=(\alpha_1,\cdots,\alpha_{n-1},\alpha_1,\cdots,\alpha_{n-2},\cdots,\alpha_1,\alpha_2,\alpha_1)\in(\Phi^+)^{\frac{n(n-1)}{2}}.\]
For $\underline{a}=(a_{11},\cdots,a_{1,n-1},a_{22},\cdots,a_{n-1,n-1})\in F^{\frac{n(n-1)}{2}}$, we define
\[b_{\underline{\beta}}(\underline{a}):=b_{\alpha_1}(a_{11})\cdots b_{\alpha_{n-1}}(a_{1,n-1})\cdots b_{\alpha_1}(a_{n-1,n-1}).\]
We have the following injective map (\cite[Proposition 3.1]{DR98})
\[b_{\underline{\beta}}: F^{\frac{n(n-1)}{2}} \lra N_0\backslash\left(Nw_0AN\cap K_0\right).\] 

For $\underline{m}=(m_{11},\cdots,m_{1,n-1},m_{22},\cdots,m_{n-1,n-1})\in\BZ_{\geq 0}^{\frac{n(n-1)}{2}}$, consider
\[Y_{\underline{\beta}}(\underline{m})=\left\{\underline{a}=(a_{11},\cdots,a_{1,n-1},a_{22},\cdots,a_{n-1,n-1})\in F^{\frac{n(n-1)}{2}}\Big|\mu(a_{ij})=m_{ij},1\leq i\leq j\leq n-1\right\}\]
where for $a \in F$, we denote
\[\mu(a)=\begin{cases}
    0, \quad &\text{if } a\in\CO;\\
    -v(a),\quad &\text{otherwise }.
\end{cases}\]

Cocharacters on $A$ can be viewed as elements of $A$ by the following injection
\[X_*(A) \lra A, \quad \lambda \mapsto \lambda(\varpi).\]

Let $\lambda \in X_*(A)$. Consider
\[Y_\lambda = \bigsqcup_{\underline{m}\in M_\lambda}Y_{\underline{\beta}}({\underline{m}})\]
with the index set
\[M_\lambda=\left\{\underline{m}\in\BZ_{\geq 0}^{\frac{n(n-1)}{2}}\Big| \lambda=-\sum_{1\leq i\leq j\leq n-1}m_{ij}\alpha_{ij}^\vee\right\}.\]

\begin{prop}[Proposition 3.2 and 3.4 in \cite{DR98}]
For each $\underline{m}$, the set $Y_{\underline{\beta}}({\underline{m}})$ can be endowed with an action of $N_0$ such that
\[b_{\underline{\beta}}(\underline{a} * u)=b_{\underline{\beta}}(\underline{a})u.\]
Let $\lambda\in X_*(A)$.  Then $b_{\underline{\beta}}$ descends to a bijection between $Y_\lambda/N_0$ and the Kloosterman set $X_{K_0}(\lambda w_\ell)$. Here, the Kloosterman set 
    \[X_{K_0}(\lambda w_\ell)=N_0\backslash N\lambda w_\ell N\cap K_0/N_0\]
    with $w_\ell$ the following representative of the longest Weyl element $w_0$ 
\[w_\ell=\begin{pmatrix}
    & & & & & (-1)^{n-1}\\
    & & & & \iddots &\\
    & & & -1 & & \\
    & & 1 & & & \\
    & -1 & & & & \\
    1 & & & & &
\end{pmatrix}.\]
\end{prop}

\begin{lem}[Lemma 3 in \cite{BM24}]\label{Bruhat decomposition}
    Let $\underline{a}=(a_{11},\cdots,a_{n-1,n-1})\in F^{\frac{n(n-1)}{2}}$.   Write  $b_{\underline{\beta}}(\underline{a}) =u a w_\ell u' \in N_0 \bs N_1 Aw_0 N_2 \cap K_0$ with $u,u'\in N$, $a \in A$. Denote by $a_{ij}=\varpi^{-m_{ij}}c_{ij}$ with $m_{ij} \geq 0$ and $c_{ij} \in \CO$ such that $(c_{ij},\varpi^{m_{ij}}) = 1$. On the first off-diagonal, we have
    \begin{align*}
        &u_{i,i+1}=\sum_{1\leq\delta \leq i}\frac{c_{\delta(i-1)}\prod_{t=1}^{\delta-1}\varpi^{m_{t(i-1)}}}{c_{\delta i}\prod_{k=1}^{\delta}\varpi^{m_{ki}}},\quad 1\leq i\leq n-1;\\
        &u_{i,i+1}'=\sum_{n-i\leq\partial\leq n-1}\frac{c_{(n-i)\partial}\prod_{k=\partial+1}^{n-1}\varpi^{m_{(n+1-i)k}}}{c_{(n+1-i)\partial}\prod_{t=\partial}^{n-1}\varpi^{m_{(n-i)t}}},\quad 1\leq i\leq n-1.
    \end{align*}
    Here, we set $c_{ij}=1$ and $m_{ij}=0$ if the condition $1\leq i\leq j\leq n-1$ is not satisfied. As a convention, when $m_{ij}=0$ for $1\leq i \leq j \leq n-1$, we define $c_{ij}^{-1}=0$ as a formal symbol.
\end{lem}

\begin{lem}[Lemma 4 in \cite{BM24}]
    For $\underline{m}\in\BZ_{\geq 0}^{\frac{n(n-1)}{2}}$, a complete system of coset representatives for $Y_{\underline{\beta}}(\underline{m})/N_0$ is given by
    \[\CC(\underline{m}) = \left\{(\varpi^{-m_{ij}}c_{ij})_{1\leq i\leq j\leq n-1}\Big| c_{ij} \in \CO/ \fp^{\sum_{k=j}^{n-1}m_{ik}},\quad (c_{ij},\varpi^{m_{ij}})=1\right\}.\]
\end{lem}

In the following, we shall focus on Kloosterman set associated to the admissible subgroup in Example \ref{GL_n example}. For 
a positive integer $\ell$ sufficiently large such that
$(1+\fp^\ell) \cap \mu_n = 1$, consider
\[K=\left\{g=\begin{pmatrix}
    g' & v\\
    z & t  
\end{pmatrix}\in K_0 \Big|g' \in M_{n-1,n-1}(\CO), v \in M_{n-1,1}(\CO), z \in M_{1,n-1}(\fp^\ell),  t\in1+\fp^\ell\right\}.\]
Note that $K \cap N = N_0$. 

For $\lambda \in X_*(A)$, consider the Kloosterman set
\[X_K(\lambda w_\ell) = N_0 \bs N\lambda w_\ell N \cap K / N_0 \subset
X_{K_0}(\lambda w_\ell).\]

\begin{lem}\label{complete system of coset representatives}
    Let $\lambda\in X_*(A)$. Then $b_{\underline{\beta}}$
    induces a bijection
    \[\bigsqcup_{\underline{m} \in M_\lambda^K} \CC^K(\underline{m}) \stackrel{\sim}{\lra} X_K(\lambda w_\ell).\]
    Here, we denote by
    \[M_\lambda^{K}=\left\{\underline{m}\in M_\lambda \Big| m_{1,n-1}\geq\ell \right\}\]
    and for each $\underline{m} \in M_{\lambda}^K$,
    \[\CC^K(\underline{m}) = \left\{(\varpi^{-m_{ij}}c_{ij})_{1\leq i\leq j\leq n-1} \in \CC(\underline{m}) \Big| c_{1,n-1}\in1+\fp^\ell \right\}.\]
    \begin{proof}
       Recall for any simple root $\beta\in\Delta$,
        \[b_\beta(a)=\begin{cases}
        \phi_\beta\begin{pmatrix}
        0 & -1\\
        1 & a
        \end{pmatrix},&\text{ if }a\in\CO;\\
        \phi_\beta\begin{pmatrix}
        c^{-1} & \\
        \varpi^m & c
        \end{pmatrix},&\text{ if }a=\varpi^{-m}c,c\in\CO^\times,m>0,
        \end{cases}\]
        and the embedding $\phi_{\alpha_i}$ maps $\begin{pmatrix}
            x & y\\
            z & w
        \end{pmatrix}$ to the matrix
        \[\begin{pmatrix}
            I_{i-1} & & \\
            & x & y &\\
            & z & w &\\
            & & & I_{n-1-i}\\
        \end{pmatrix}.\]
       Therefore, for each $1 \leq i \leq n-2$, 
       $b_{\alpha_i}(a) \in K$ for any $a \in F$ while
       $b_{\alpha_{n-1}}(a) \in K$ if and only if $a = \varpi^{-m}c$
       with $m \geq \ell$ and $c \in 1+\fp^\ell$.
       
       The reduced expression for $w_0$ is chosen such that there is exactly one factor $s_{\alpha_{n-1}}$ in this expression.
       The map $b_{\underline{\beta}}$ is a 
       product of $b_{\alpha_i}$ with respect to the reduced expression for $w_0$. Therefore, the factors
       of $b_{\underline{\beta}}(\underline{a})$ are always in $K$
       except $b_{\alpha_{n-1}}(a_{1,n-1})$. In particular,
       $b_{\underline{\beta}}(\underline{a}) \in K$ if and only if
       $c_{1,n-1} \in 1+\fp^\ell$ and $m_{1,n-1} \geq \ell$.
    \end{proof}
\end{lem}

\subsection{Kloosterman sums for  $\GL_n$}

We keep notations in the last section. 

Let $\psi_F$ be a nontrivial additive character on $F$ with
$\psi_F|_\CO = 1$. We consider two generic characters on $N$ 
\[\psi(n)=\psi_F\left(\sum_{j=1}^{n-1}\psi_jn_{j,j+1}\right),\psi'(n)=\psi_F\left(\sum_{j=1}^{n-1}\psi_j'n_{j,j+1}\right)\quad n=(n_{i,j})_{1\leq i,j\leq n},\]
where $\psi_j,\psi_j'\in\CO,1\leq j\leq n-1$ are nonzero. In particular,
$\psi|_{N_0} = \psi'_{N_0} = 1$.

Following Blomer-Man \cite{BM24}, consider the following Kloosterman sum on $G$ with respect to the quadruple $\left(\psi,\psi',w_\ell,K\right)$ 
\[\Kl'(a) = \Kl'(a;\psi,\psi',w_\ell,K) = \sum_{x \in X_K(aw_\ell)} \psi(u(x))\psi'(u'(x)), \quad a \in A.\]
Here, $K$ is the admissible open compact subgroup of $G$ given
in Example \ref{GL_n example} associated to
a sufficiently large  $\ell$ with $(1+\fp^\ell)\cap\mu_n = 1$
and the Kloosterman set $X_K(aw_\ell)$ is the one
in Lemma \ref{complete system of coset representatives}.

The following is the main result in this section.
\begin{thm}\label{main theorem GL_n}
    There exists a constant $\delta_n>0$ such that for any $a=\diag(\varpi^{-r_1},\varpi^{r_1-r_2},\cdots,\varpi^{r_{n-1}})$ with $r_i\geq 0$, $1 \leq i \leq n-1$, we have
    \[\Kl'(a)\ll(\max_{1\leq j\leq n-1}|\psi_j|^{-\frac{1}{2}})\cdot q^{(1-\delta_n)(r_1+\cdots+r_{n-1})}.\]
    Here, the notion $A \ll B$ means that $A \leq CB$ for
    a constant $C$ depending only on $F$.
\end{thm}

\begin{prop}
    Let $a=\diag(\varpi^{-r_1},\varpi^{r_1-r_2},\cdots,\varpi^{r_{n-1}})$ with $r_i\geq 0$. We have
    \[\mathrm{Kl}'(a)=\sum_{\underline{m}\in\CM^K(r)}\mathrm{Kl}'(\underline{m}).\]
    Here,
    \[\CM^K(r)=\left\{\underline{m}=(m_{ij})_{1\leq i\leq j\leq n-1}\in\BZ_{\geq 0}^{\frac{n(n-1)}{2}}\Big| \sum_{i\leq k\leq j}m_{ij}=r_k,1\leq k\leq n-1, m_{1,n-1} \geq \ell \right\}.\]
    For each $\underline{m} \in \CM^K(r)$, the partial
    Kloosterman sum $\Kl'(\underline{m})$ is
    \begin{align*}
    \sum_{\underline{c}\in\CC^K(\underline{m})}&\psi_F\Bigg(\sum_{j=1}^{n-1}\sum_{i=1}^j\psi_jc_{i(j-1)}c_{ij}^{-1}\varpi^{\sum_{k=1}^{i-1}m_{k(j-1)}-\sum_{k=1}^im_{kj}}\\
    &+\sum_{i=1}^{n-1}\sum_{j=1}^i\psi_i'c_{(n-i)(n-j)}c_{(n+1-i)(n-j)}^{-1}\varpi^{\sum_{k=1}^{j-1}m_{(n+1-i)(n-k)}-\sum_{k=1}^jm_{(n-i)(n-k)}}\Bigg).
\end{align*}
Here we employ the convention that $c_{ij}=1$ and $m_{ij}=0$ for $i>j$. Moreover, for
$1\leq i\leq j\leq n-1$, we set $c_{ij}^{-1}=0$ when $m_{ij}=0$. 
    \begin{proof}
    Denote by $\lambda=-\sum_{j=1}^{n-1}r_j\alpha_j^\vee$. Then
    $\lambda$ is the cocharacter associated to $a$. We have
    \[M_\lambda =\left\{\underline{m}=(m_{ij})_{1\leq i\leq j\leq n-1}\in\BZ_{\geq 0}^{\frac{n(n-1)}{2}}\Big| \sum_{i\leq k\leq j}m_{ij}=r_k,1\leq k\leq n-1 \right\} \]
    and $M_\lambda^K = \CM^K(r)$. By Lemma \ref{complete system of coset representatives},
    \[X_K(aw_\ell) = \bigsqcup_{\underline{m} \in \CM^K(r)}
    \CC^K(\underline{m})\]
    and
    \[\Kl'(a) = \sum_{\underline{m} \in \CM^K(r)} 
    \Kl'(\underline{m}), \quad 
    \Kl'(\underline{m}) = \sum_{\underline{a} \in \CC^K(\underline{m})}
    \psi(u(b_{\underline{\beta}}(\underline{a})))\psi'(u'(b_{\underline{\beta}}(\underline{a}))).\]
    We now obtain the Proposition by Lemma \ref{Bruhat decomposition}.
    \end{proof}
\end{prop}

Note that, for each $\underline{m} \in \CM^K(r)$,  $\Kl'(\underline{m})$ is independent on the choice of a system of representatives for $\CC^K(\underline{m})$. See also remarks following
\cite[Equation (1.6)]{BM24}.

We now give the proof of Theorem \ref{main theorem GL_n}. The argument is slightly different from that in \cite[Section 5.1 and 5.2]{BM24}.

Fix $\underline{m}\in \CM^K(r)$.

For each index $(i,j)$ with $1 \leq i \leq j \leq n-1$, 
denote by 
\[
M_{ij} =\sum_{k=j}^{n-1} m_{ik}, \quad  C_{ij}=\varpi^{M_{ij}}
\]
and
\[A_{ij}=\psi_{j+1}c_{i,j+1}^{-1}\varpi^{a_1(i,j)}+\psi_{n-i}'c_{i+1,j}^{-1}\varpi^{a_2(i,j)}, \quad
    B_{ij}=\psi_jc_{i,j-1}\varpi^{b_1(i,j)}+\psi_{n+1-i}'c_{i-1,j}\varpi^{b_2(i,j)},\]
with
\begin{align*}
    a_1(i,j)=\sum_{k=1}^{i-1}m_{kj}-\sum_{k=1}^im_{k,j+1},\quad a_2(i,j)=\sum_{k=j+1}^{n-1}m_{i+1,k}-\sum_{k=j}^{n-1}m_{ik},\\
    b_1(i,j)=\sum_{k=1}^{i-1}m_{k,j-1}-\sum_{k=1}^im_{kj},\quad b_2(i,j)=\sum_{k=j+1}^{n-1}m_{ik}-\sum_{k=j}^{n-1}m_{i-1,k}.
\end{align*}
Here, we use the following conventions: if $m_{ij} = 0$ for $1 \leq i \leq j \leq n-1$, then $c_{ij}^{-1} = 0$;
If $j < i$, then $m_{ij} = 0$, $c_{ij} = c_{ij}^{-1} = 1$ with $C_{ij} = 1$; if none of the above cases apply
with $i < 0$ or $j > n-1$, then $c_{ij} = c_{ij}^{-1} = m_{ij} = 0$ and $C_{ij} = 1$.

\begin{lem}\label{lifting modulus}
    Fix an index $(i,j)$ with $1 \leq i \leq j \leq n-1$. 
    Assume $m_{ij} \not=0$. 
    Then for any system of representatives $c_{uv}$ for $\CO/\fp^{M_{uv}}$ with $(u,v) \not= (i,j)$, we have
    \[|\Kl'(\underline{m})|\ll q^{M_{ij}}\sum_{c_{uv},(u,v)\neq(i,j)} W(A_{ij},B_{ij}).\]
    Here, for $A,B\in F$, we denote by $W(A,B)=\min \{1,|A|^{-1/2},|B|^{-1/2} \}$.
\end{lem}

\begin{proof}
    For any index $(u,v) \not=(i,j)$, we fix a system of representatives $c_{uv}$ for $\CO/\fp^{M_{uv}}$ (and
    other conditions as in the definition of $\CC^K(\underline{m})$). 
    
    Firstly, we consider the case $(i,j)\neq(1,n-1)$. Fix a complete representatives system $\CR$ of $(\CO/\fp^{M_{ij}})^\times = (\CO/C_{ij}\CO)^\times$. We can write $\Kl'(\underline{m})$ as
    \[\sum_{x \in \CR}\sum_{c_{uv},(u,v)\neq(i,j)}\psi_F(xA_{ij}+x^{-1}B_{ij}+G_{ij})\]
    where $G_{ij}$ denotes the contributions from $c_{uv}$
    with $(u,v) \not\in \{(i,j+1),(i+1,j),(i,j-1),(i-1,j)\}$.
    Note that, $\Kl'(\underline{m})$ is independent on the
    choice of $\CR$.

    We choose a positive integer $\alpha$ which is large
    enough such that for any $A_{ij}, B_{ij}$ in the above
    \[v(A_{ij}) + \alpha, v(B_{ij}) + \alpha \geq 0.\]
     For any $h\in\CO$, denote by  
    \[
    \CR_h=\left\{x+hC_{ij}\Big| x\in \CR  \right\}.
    \]
    Then each $\CR_h$ is still a system of  representatives
    for $(\CO/C_{ij}\CO)^\times$  and the set $\bigsqcup_{h\in\CO/\fp^\alpha}\CR_h$ is exactly a system of  representatives for $(\CO/C_{ij}\varpi^\alpha\CO)^\times$. As $\Kl'(\underline{m})$ is independent on the
    choice of $\CR$ 
    \[\begin{aligned}
        \Kl'(\underline{m}) &=
        q^{-\alpha} \sum_{x \in \bigsqcup_{h\in\CO/\fp^\alpha}\CR_h} \sum_{c_{uv},(u,v)\neq(i,j)}\psi_F(xA_{ij}+x^{-1}B_{ij}+G_{ij}) \\
        &= q^{-\alpha} \sum_{c_{uv},(u,v)\neq(i,j)}\psi_F(G_{ij}) \sum_{x\in(\CO/C_{ij}\varpi^\alpha\CO)^\times}\psi_F(xA_{ij}+x^{-1}B_{ij})
    \end{aligned}\]
    Here, we use the fact that $C_{ij}\varpi^\alpha A_{ij},C_{ij}\varpi^\alpha B_{ij}\in \CO$.

     The inner sum is the Kloosterman sum $\Kl(C_{ij}\varpi^\alpha A_{ij}, C_{ij}\varpi^\alpha B_{ij};M_{ij}+\alpha)$ in Corollary \ref{classical kloosterman sum}. 
     By Corollary \ref{classical kloosterman sum} (
     if $A_{ij}, B_{ij} \in \CO$, we consider the trivial bound) 
    \[
    |\Kl(C_{ij}\varpi^\alpha A_{ij}, C_{ij}\varpi^\alpha B_{ij};M_{ij}+\alpha)|\ll q^{M_{ij}+\alpha}W(A_{ij},B_{ij}).
    \]
    Hence
    \[|\Kl'(\underline{m})|\ll q^{M_{ij}}\sum_{c_{uv},(u,v)\neq(i,j)} W(A_{ij},B_{ij}).\]

    Consider the case $(i,j)=(1,n-1)$. By our choice of the open compact subgroup $K$, $c_{1,n-1} \in 1+\fp^\ell$ and $M_{1,n-1} = m_{1,n-1} \geq \ell$. 
    
    Similarly as the case $(i,j) \not= (1,n-1)$, we have
    \[\Kl'(\underline{m})= q^{-\alpha} \sum_{c_{uv},(u,v)\neq(i,j)}\psi_F(G_{ij}) \sum_{x\in 1+\varpi^\ell\CO/1+C_{ij}\varpi^\alpha\CO}\psi_F(xA_{ij}+x^{-1}B_{ij}).
    \]
    The inner sum is the Kloosterman sum $\Kl(C_{ij}\varpi^\alpha A_{ij}, C_{ij}\varpi^\alpha B_{ij};\ell,M_{ij}+\alpha)$ in Proposition \ref{GL2 kloosterman sum}.
    By Proposition \ref{GL2 kloosterman sum}
    \[
    |\Kl(C_{ij}\varpi^\alpha A_{ij}, C_{ij}\varpi^\alpha B_{ij};\ell,M_{ij}+\alpha)|\ll q^{M_{ij}+\alpha}W(A_{ij},B_{ij}).
    \]
    Hence, in this case, we also have
    \[|\Kl'(\underline{m})|\ll q^{M_{ij}}\sum_{c_{uv},(u,v)\neq(i,j)} W(A_{ij},B_{ij}).\]
\end{proof}

The following lemma plays a role similar to that of \cite[Lemma 7]{BM24}.

\begin{lem}\label{linear average}
    Let $r> 0$ be a positive integer and  $a\in F^\times$.
    Then for any complete representatives system $\CR$ of $(\CO/\fp^r)^\times$ and $D \in F$ we have
    \[
    \sum_{u\in\CR} \min\{1,|au+D|^{-\frac{1}{2}}\}\leq 4(q^{r+\frac{1}{2}\min \{0,v(a)\}}+1).\]
\end{lem}

\begin{proof}

    If $v(a)\ge0$, then each term is at most $1$, and the conclusion follows trivially.

    Assume $-n = v(a)<0$.  Write
    \[au+D=\varpi^{-n}(a_0u+D_0),\quad a_0\in\CO^\times.\]
    If $D_0\notin \CO$, then
    \[
    \min\{1,|au+D|^{-\frac{1}{2}}\}\le q^{\frac{1}{2}v(a)}
    \]
    holds for all $u$, and the conclusion follows.
    
    Now suppose $D_0\in\CO$.  For each $u \in \CR$, denote by
    \[m(u) = v(a_0u+D_0), \quad w(u) = m(u)-n \geq -n, \quad
    S = \sum_{u \in \CR} \min\{1,q^{w(u)/2}\}.\]
    The following inequality is key to the whole proof
    \[\min\{1,q^{w(u)/2}\} \leq \sum_{h=0}^n q^{-h/2} \delta(w(u) \geq -h).\]
    Here,  the notion $\delta(w(u) \geq -h)=1$ if $w(u) \geq -h$ and $=0$ otherwise.
    In fact, in the right hand side, if $w(u) \geq 0$, the term $h = 0$ contributes $1$  while if $0 > w(u) \geq -n$, the term $h = -w(u) \leq n$ contributes $q^{w(u)/2}$.

    Therefore, 
    \[S \leq \sum_{h=0}^n q^{-h/2} M_{n-h} = q^{-n/2}\sum_{j=0}^n q^{j/2}M_j.\]
    Here, for each $0 \leq h \leq n$, 
    \[M_h = \#\{u \in \CR: m(u) \geq h\} = \#\{u \in \CR: u \equiv a_0^{-1} D_0 \pmod{\fp^h}\}.\]
    As there is at most one solution for the equation $u \equiv a_0^{-1} D_0 \pmod{\fp^h}$ with $u \in \CR$,
    \[M_h \leq \max\{q^{r-h},1\} \leq q^{r-h} + 1.\]

    In particular,
    \[\sum_{j=0}^n q^{j/2}M_j \leq q^r\sum_{j=0}^n q^{-j/2} + 
    \sum_{j=0}^n q^{j/2} < \frac{q^{1/2}}{q^{1/2} - 1} (q^r + q^{n/2})\]
    and
    \[S \leq \frac{q^{1/2}}{q^{1/2} - 1} (q^{r-n/2} + 1).\]
    
    In summary, for all the cases, we have 
    \[\sum_{u\in\CR} \min\{1,|au+D|^{-\frac{1}{2}}\}\leq \frac{q^{1/2}}{q^{1/2} - 1}(q^{r+\frac{1}{2}\min \{0,v(a)\}}+1) \leq 4(q^{r+\frac{1}{2}\min \{0,v(a)\}}+1).\]

\end{proof}

\begin{lem}\label{refined bound}
    For any index $(i,j)$ with $1 \leq i \leq j \leq n-1$, we have
    \[
    |\Kl'(\underline{m})|\ll Q(\psi) q^{\rho-m_{ij}/2+(n-1)\mu_{ij}}.
    \]
    Here, $Q(\psi)=\max_{1\leq j\leq n-1}|\psi_j|^{-1/2}$, $\rho=\sum_{k=1}^{n-1}r_k$ and $\mu_{ij}=\max_{(i',j')<(i,j)}m_{i'j'}$  with the following order for the indices
    \[(1,n-1)<(1,n-2)<\cdots<(1,1)<(2,n-1)<\cdots<(2,2)<\cdots<(n-1,n-1).\]
\end{lem}

\begin{proof}
     Note that
    \[\rho=\sum_{k=1}^{n-1}\sum_{i\le k \le j}m_{ij}=\sum_{1\le i\le j\le n-1}(j-i+1)m_{ij} = \sum_{1 \leq i \leq j \leq n-1} M_{ij}.\]
    We have the trivial bound
   \[
    |\Kl'(\underline{m})|\le q^\rho.
    \]
  
    If $m_{ij}=0$, then the trivial bound is enough.  We will assume $m_{ij}>0$ in the following. By Lemma
    \ref{lifting modulus}, 
     \[|\Kl'(\underline{m})|\ll q^{M_{ij}}\sum_{c_{uv},(u,v)\neq(i,j)} W(A_{ij},B_{ij}).\]
     
     Suppose $i<j$. By Lemma \ref{linear average} with $r = M_{i.j-1}$, $a = \psi_j\varpi^{b_1(i,j)}$ and $D = \psi_{n+2-i}'c_{i-1,j}\varpi^{b_2(i,j)}$, we have
     \[\sum_{c_{i,j-1}} W(A_{ij},B_{ij})
     \leq \sum_{c_{i,j-1}} \min\{1,|B_{ij}|^{-1/2}\}
     \ll q^{M_{i,j-1} + \frac{1}{2}\min\{0,v(\psi_j) + b_1(i,j)\}}+1.\]
     We have 
    \[
    b_1(i,j) = -m_{ij}+\sum_{k=1}^{i-1}m_{k,j-1} - \sum_{k=1}^{i-1}m_{kj}  
    \le  -m_{ij}+(i-1)\mu_{ij}.\]
    In particular,
    $\min\{0,v_p(\psi_j)+b_1(i,j)\}\le v_p(\psi_j)-m_{ij}+(i-1)\mu_{ij}$ and 
    \[\sum_{c_{i,j-1}} W(A_{ij},B_{ij}) \ll
    Q(\psi) q^{M_{i,j-1} -m_{ij}/2 +(n-1)\mu_{ij}}+1.\]
    Since $m_{ij}\leq M_{i,j-1}$, we have $q^{-M_{i,j-1}}\leq q^{-m_{ij}}\leq q^{-m_{ij}/2+(n-1)\mu_{ij}}$ and hence
    \begin{align*}
        |\Kl'(\underline{m})|&\ll Q(\psi) q^{\rho-m_{ij}/2+(n-1)\mu_{ij}}+q^\rho\cdot q^{-M_{i,j-1}}\leq Q(\psi) q^{\rho-m_{ij}/2+(n-1)\mu_{ij}}+Q(\psi) q^{\rho-m_{ij}/2+(n-1)\mu_{ij}}\\
        &\leq 2Q(\psi)q^{\rho-m_{ij}/2+(n-1)\mu_{ij}}
    \end{align*}

    Suppose $i=j$.  Then $c_{i,i-1} = 1$ and $M_{i,i-1} = 0$ by our convention with 
    \[B_{ii}=\psi_i\varpi^{b_1(i,i)}+\psi_{n+1-i}'c_{i-1,i}\varpi^{b_2(i,i)}.\]    
    
    If
    \[Q(\psi)q^{-\frac{1}{2}m_{ii}+(n-1)\mu_{ii}}\geq 1,\]
    then the trivial bound $|\Kl'(\underline{m})|\le q^\rho$ proves the result. Now suppose $Q(\psi) < q^{1/2m_{ii}-(n-1)\mu_{ii}}$. This
    implies that
    \[v(\psi_i) <  m_{ii}-2(n-1)\mu_{ii}.\]
     
    Note that 
    \[b_1(i,i)\le-m_{ii}+(i-1)\mu_{ii}, \quad
    b_2(i,i) \geq -(n-1)\mu_{ii}.\]
    We have
    \[v(\psi_i\varpi^{b_1(i,i)})<-(n-1)\mu_{ii}, \quad 
    v(\psi_{n+1}'c_{i-1,i}\varpi^{b_{2}(i,i)})\geq b_2(i,i)\geq -(n-1)\mu_{ii}\]
    so that
    \[v(B_{ii})=v(\psi_i)+b_1(i,i)<0.\]
    This implies
    \[W(A_{ii},B_{ii})\leq\min\{1,|B_{ii}|^{-\frac{1}{2}}\}=|B_{ii}|^{-\frac{1}{2}}=q^{\frac{1}{2}(v(\psi_i)+b_1(i,i))}\]
    and hence again by Lemma \ref{lifting modulus}, 
    \begin{align*}
        |\Kl'(\underline{m})|\ll Q(\psi) q^{\rho-m_{ij}/2+(n-1)\mu_{ij}}.
    \end{align*}
\end{proof}

\begin{lem}
    Denote by $N = \frac{n(n-1)}{2}$, $M_* = \max_{(i,j)} m_{ij}$ and 
    \[
    \varepsilon_n=\frac{2n-3}{2((2n-2)^N-1)}.
    \]
    Then there exists an index $(i,j)$ with $1 \leq i \leq j \leq n-1$  such that
    \[
    m_{ij} - 2(n-1)\mu_{ij} \ge\varepsilon_n M_*.
    \]
\end{lem}

\begin{proof}
    According to the given order of indices, we write $m_{ij}$ as $m_1,...,m_N$. Assume for all $1\leq i\leq N$, $m_i-2(n-1)\mu_i<\varepsilon_n M_*$. Let $R_t:=\max_{1\le s\le t} m_s, R_0:=0$. Since $\mu_t\le R_{t-1}$, we have
    \[
    R_t\le 2(n-1)R_{t-1}+\varepsilon_n M_*
    \]
    for all $1\leq t\leq N$. Inductively, we will have
    \begin{eqnarray*}
        R_N &\le &\varepsilon_nM_*\sum_{k=0}^{N-1} (2(n-1))^k  \\
        &=&\varepsilon_nM_* \frac{(2n-2)^N-1}{2n-3}=\frac{1}{2}M_*.
    \end{eqnarray*}
    This is a contradiction since $R_N=M_*$.
 
\end{proof}

\begin{proof}[The proof of Theorem \ref{main theorem GL_n}]
Choose an index $(i,j)$ satisfying the above lemma. 
Then Lemma \ref{refined bound} gives
\[
    |\Kl'(\underline{m})| \ll Q(\psi)q^{\rho-\frac{1}{2}(m_{ij}-2(n-1)\mu_{ij})}  
    \le  Q(\psi)q^{\rho -\varepsilon_n M_*/2}.\]
Note that 
\[\rho  = \sum_{1 \leq i \leq j \leq n-1} (j-i+1)m_{ij} \le(n-1)NM_*.\] 
We have
\[|\Kl'(\underline{m})| \ll Q(\psi)q^{(1-\frac{\varepsilon_n}{2(n-1)N})\rho}.\]

Finally, since
    \[\Kl'(a)=\sum_{\underline{m}\in \CM^K(r)}\Kl'(\underline{m}), \quad 
    \# \CM^K(r)\le (\rho +1)^N,\]
we have
\[|\Kl'(a)| \ll Q(\psi) (\rho +1)^{N+1} q^{(1-\frac{\varepsilon_n}{2(n-1)N})\rho}\le Q(\psi)q^{(1-\frac{\varepsilon_n}{4(n-1)N})\rho}\]
and take $\delta_n=\frac{\varepsilon_n}{4(n-1)N}$. We are done.
\end{proof}

We denote the following Kloosterman sum on $G$ with respect to the quadruple $\left(\psi,\psi',w_\ell,K\right)$ 
\[\Kl(a;\psi,\psi',w_\ell,K) = \sum_{x \in X_K(w_\ell a)} \psi(u(x))\psi'(u'(x)), \quad a \in A.\]
Note that the only difference between $\Kl(a;\psi,\psi',w_\ell,K)$
and $\Kl'(a;\psi,\psi',w_\ell,K)$ is their associated
Kloosterman sets $X_K(w_\ell a)$ and $X_K(a w_\ell)$ respectively.
Moreover, if $\psi=\psi'$, then $\Kl(a;\psi,\psi',w_\ell,K)$ equal to $\Kl(a;\psi,w_\ell,K)$ in Definition \ref{defn-sys-Kl}.

\begin{cor}\label{corollary on Kloosterman sum}
    There exists a constant $\delta_n>0$ such that
    \[\mathrm{Kl}(a;\psi,\psi',w_\ell,K)\ll (\max_{1\leq j\leq n-1}|\psi_j|^{-\frac{1}{2}})\cdot q^{(1-\delta_n)(r_1+\cdots+r_{n-1})}\]
    for $a=\diag(\varpi^{r_1},\varpi^{r_2-r_1},\cdots,\varpi^{r_{n-1}})$ with $r_i\geq 0$, $1\leq i\leq n-1$.
    \begin{proof}
        Note that $a^{w_\ell}=w_\ell aw_\ell^{-1}=\diag(\varpi^{-r_{n-1}},\varpi^{r_{n-1}-r_{n-2}},\cdots,\varpi^{r_1})$. By Theorem \ref{main theorem GL_n}, there exists a $\delta=\delta_n>0$ such that
        \[\mathrm{Kl}'(a^{w_\ell};\psi,\psi',w_\ell,K)\ll (\max_{1\leq j\leq n-1}|\psi_j|^{-\frac{1}{2}})q^{(1-\delta_n)(r_1+\cdots+r_{n-1})}.\]
        Hence,
        \[\mathrm{Kl}(a;\psi,\psi',w_\ell,K)=\mathrm{Kl}'(a^{w_\ell};\psi,\psi',w_\ell,K)\ll (\max_{1\leq j\leq n-1}|\psi_j|^{-\frac{1}{2}})\cdot q^{(1-\delta_n)(r_1+\cdots+r_{n-1})}.\]
    \end{proof}
\end{cor}
\begin{cor}\label{congruence subgroup gamma}
    The Kloosterman sum $\Kl(\cdot;\psi_0^{-1},w_\ell,K)$ has
    a nontrivial bound: there exist $\varepsilon>0$ and $C>0$ such that
    \[\left|\mathrm{Kl}(a;\psi_0^{-1},w_\ell,K)\cdot\delta^{1/2-\varepsilon}(a)\right| \leq C, \quad a \in A^\der\]
    with $A^\der = A\cap G^\der$ the maximal torus of $\SL_n$.
    \begin{proof}
        Recall that for each $a\in A$, there exists a unique $\lambda_a\in X_*(A)$ such that 
        \[a=a_0\lambda_a, \quad a_0 \in A\cap K_0.\]

        Note that $X_{K}(w_\ell a)\subset X_{K_0}(w_\ell a)$. By Theorem \ref{DR1998},  if $\lambda_a\notin\BZ_{\geq0}\Phi^{\vee,+}$, then
        $X_K(w_\ell a) = \emptyset$. We may assume $\lambda_a\in \BZ_{\geq0}\Phi^{\vee,+}$ and write $a\in A^\der$ as
        \[a=a_0\lambda_a, \quad a_0=\diag(v_1,\cdots,v_n),\quad \lambda_a=\diag(\varpi^{r_1},\varpi^{r_2-r_1},\cdots,\varpi^{-r_{n-1}})\]
        where for each $i$, $r_i$ is a positive integer and $v_i \in \CO^\times$ with $v_1\cdots v_n=1$. Then
        \[\delta(a) = q^{-2(r_1 + \cdots + r_{n-1})}.\]

        Assume $n_1w_\ell a n_2 \in N_1w_\ell a N_2 \cap K$.
        Write $a_0=b_0c_0$ with
        \[b_0=\diag(v_1,\cdots,v_{n-1},1),\quad c_0=\diag(1,\cdots,1,v_n).\]
        Then
        \[n_1w_\ell a n_2 = c_0^{w_\ell}\left(\Ad((c_0^{w_\ell})^{-1})
        n_1\right)w_\ell\lambda_a\left(\Ad(b_0)n_2\right)b_0
        \in K.\]
        As $b_0 \in K$ and $c_0^{w_\ell}=\diag(v_n,1,\cdots,1)\in K$, we have
        \[\left(\Ad((c_0^{w_\ell})^{-1})
        n_1\right)w_\ell\lambda_a\left(\Ad(b_0)n_2\right) \in K.\]
        
        This implies
        \begin{align*}
            \mathrm{Kl}(a;\psi_0^{-1},w_\ell,K)&=\sum_{x\in X_K(w_\ell a)}\psi_0^{-1}(u(x))\psi_0^{-1}(u'(x))\\
            &=\sum_{x\in X_K(w_\ell\lambda_a)}\psi_0^{-1}(\Ad((c_0^{w_\ell})) u(x))\psi_0^{-1}(\Ad(b_0^{-1})u'(x))\\
            &=\Kl(\lambda_a;\psi_1,\psi_2,w_\ell,K)
        \end{align*}
        where $\psi_1=\psi_0^{-1}(\Ad(c_0^{w_\ell})\cdot),\psi_2=\psi_0^{-1}(\Ad(b_0^{-1})\cdot)$ are generic characters of $N$ trivial on $N(\CO)$. Moreover, $|\psi_{1,j}| = 
        |\psi_{2,j}| = 1$ for any $1 \leq j \leq n-1$.
        
        By Corollary \ref{corollary on Kloosterman sum}, there exists a $\delta_n>0$ such that
        \[\Kl(\lambda_a;\psi_1,\psi_2,w_\ell,K)\ll 
        q^{(1-\delta_n)(r_1+\cdots+r_{n-1})}.\]
        Therefore, the Kloosterman sum $\Kl(\cdot;\psi_0^{-1},w_\ell,K)$ has a nontrivial 
        bound.
    \end{proof}
\end{cor}
\begin{cor}
    For any irreducible smooth admissible representation $\pi$ on $G=\GL_n$ which is generic with respect to $\psi_0$, the Bessel distribution $B_\pi$ is regular.
    \begin{proof}
        By Theorem \ref{main theorem}, it is enough to
        prove that for any subset $I \subset \Delta$, the 
        Kloosterman sum $\Kl_I$ on $M_I$ with respect to
        the triple $(\psi_0|_{N_{M_I}}^{-1},w_\ell,K_I)$
        has a nontrivial bound. Here, $K_I$ is an admissible
        compact open subgroup of $M_I$. The Levi 
        subgroup $M_I$ of $G = \GL_n$ is the  product of $\GL_{m_i}$ for some $m_i$ with $n = \sum_i m_i$. The Kloosterman sum
        on $M_I$ is the product of Kloosterman sums on 
        $\GL_{m_i}$ (by considering $K_I$ as a product of
        admissible subgroups on each $\GL_{m_i}$). It
        is enough to prove the existence of nontrivial 
        bound for $I = \Delta$. The nontrivial bound for
        $\Kl_\Delta$ with $K_\Delta = K$ is given in 
        Corollary \ref{congruence subgroup gamma}. Therefore,
        we obtain the regularity.
    \end{proof}
\end{cor}
\subsection{Groups of type $A$}
In this section, we prove the regularity of Bessel distributions for type $A$ split reductive groups.

Let $G$ be a split reductive group over $F$. As before, we fix a Borel subgroup $B =AN$ of $G$ and $\psi = \psi_0$ a generic character on $N$. 

Denote by $\{G_i\}_{1 \leq i \leq s}$ the set of minimal nontrivial normal smooth connected subgroups of $G^\der$. By \cite[Theorem in Section 27.5]{H80}, the groups $G_i$ are simple and the multiplication homomorphism
\[G_1 \times \cdots \times G_s \lra G^\der\]
is an isogeny map. These groups $G_i$ are called simple
factors of $G$. The root system $\Phi$ of $G$ is decomposed
as $\Phi = \sqcup_i \Phi_i$ where $\Phi_i$ is the root system of $G_i$, $1 \leq i \leq s$. 

The group $G$ is called of type $A$ if every irreducible component $\Phi_i$ of $\Phi$ is an irreducible root system
of type $A$. If $G$ is of type $A$, then all the (standard) Levi subgroups of $G$ are still of type $A$. In fact, the Levi
subgroups of $G$ are in the form of $M_I$ for a subset
$I$ of $\Delta$. The root system of $M_I$ is $\Phi_I = \BZ I
\cap \Phi$, which is a subsystem of $\Phi$. The irreducible
components of $\Phi_I$ are all of type $A$.

\begin{thm}\label{thm-type-A}
    Suppose $G$ is a type $A$ split reductive group. For any irreducible admissible generic representation $\pi$ of $G$, the Bessel distribution $B_\pi$ is regular.
    \begin{proof}
        By Theorem \ref{main theorem}, it is enough to
        prove that for any subset $I \subset \Delta$, the 
        Kloosterman sum $\Kl_{M_I}$ on $M_I$ with respect to
        the triple $(\psi_0|_{N_{M_I}}^{-1},\dot{w_I},K_I)$
        has a nontrivial bound. Here, $K_I$ is some admissible group of $M_I$ and $\{\dot{w_I}\}_{I\subset\Delta}$ is a system of representatives of $\{w_I\}_{I\subset\Delta}$.

        For every $I$, we have the following  isogeny map
        \[\wt{M_{I,1}}\times\cdots\times\wt{M_{I,s_I}}\to M_I^\der\]
        where $\wt{M_{I,i}},1\leq i\leq s_I$ is the universal cover of the simple factor $M_{I,i}$ of $M_I^\der$. 
        
        For each $i$, the group $\wt{M_{I,i}}$ is simple, simply connected and of type $A$. Therefore, 
        $\wt{M_{I,i}} = \SL_n$ for some $n > 0$. 
        Consider the following admissible group of $\GL_n$ in Example \ref{GL_n example}
        \[K_n=\left\{g=\begin{pmatrix}
        g' & v\\
        z & t  
        \end{pmatrix}\in \GL_n(\CO) \Big|g' \in M_{n-1,n-1}(\CO), v \in M_{n-1,1}(\CO), z \in M_{1,n-1}(\fp^\ell),  t\in1+\fp^\ell\right\}\]
         with $\ell$ a sufficiently large  positive
        integer. By Corollary \ref{congruence subgroup gamma} and Lemma \ref{kloosterman sum for derived subgroup}, there exists $\varepsilon>0$ and $C>0$ such that for any $a$ in the maximal torus of $\SL_n$
        \[\left|\mathrm{Kl}_{\SL_n}\left(a;\psi_0^{-1},w_\ell,K_n\cap\SL_n\right)\cdot\delta^{1/2-\varepsilon}(a)\right| \leq C.\]
        
        In particular, for every $i,1\leq i\leq s$, we can find an admissible group $K_{I,i}$ of $\wt{M_{I,i}}$ such that the Kloosterman sum $\Kl_{\wt{M_{I,i}}}(\cdot;\psi_0^{-1}|_{N_i},\dot{w_{0,i}},K_i)$ has a nontrivial bound.
        
        By Lemma \ref{Kloosterman sum under isogeny map} (1), 
        the product of $K_{I,i}$ gives an admissible group $K_I$ of $M_I^\der$. By Lemma \ref{Kloosterman sum of product group} and 
        Lemma \ref{Kloosterman sum under isogeny map}, for a representatives $\dot{w_I}$ of the longest Weyl element in $W_I$, the Kloosterman sums $\Kl_{M_I^\der}(\cdot;\psi_0^{-1}|_{N_I},\dot{w_I},K_I)$ has a nontrivial bound. Here, we use the fact that the modular characters are compatible under isogeny maps. 

        By  Lemma \ref{kloosterman sum for derived subgroup}, there exists an admissible group $K_{M_I}$ of $M_I$ and a representative $\dot{w_I}$ of the longest Weyl element in $W_I$ such that the Kloosterman sums $\Kl_{M_I}(\cdot;\psi_0^{-1}|_{N_I},\dot{w_I},K_{M_I})$ has a nontrivial bound. Therefore, we obtain the regularity.
    \end{proof}
\end{thm}

\section{Stevens' approach}

In this section, applying the main result (Theorem \ref{main theorem}), we obtain the regularity of Bessel distributions  for $\Sp_4$ following the approach of Stevens \cite{Ste87}.

\subsection{The general argument}
We now consider Kloosterman sums on split reductive group $G$ over $F$. As before, we fix a Borel subgroup $B$ of $G$. We have the Levi decomposition $B = AN$ with $A$ the maximal torus and $N$ its unipotent radical. Let $\Phi \subset X^*(A)$ be the set of roots of $G$ with respect to $A$ and use $\Delta=\{\alpha_1,\cdots,\alpha_m\}$ to denote the set of simple roots. Denote by $\delta = \delta_B$ the modulus character of $B$.

We fix 
\begin{itemize}
\item a generic character $\psi_0$ on $N$ as the one given before
Definition \ref{relevant with respect to psi_0}.
\item a representative $\dot{w_0}$ for the longest Weyl element $w_0$.
\item a filtration $\{K_d\}_d$ of open compact subgroups in $G$ as in the Example \ref{special filtration}. In particular, 
$\bigcap_{d\geq 1}K_d=\{1\}$ and the Iwahori decomposition $K_d=\ov{N}_dA_dN_d$. In particular, for $d > 0$ sufficiently large,
$K_d$ is admissible.
\end{itemize}

Consider the family of Kloosterman sums on $G$ 
\[\Kl_d(a) =\Kl(a;\psi,\dot{w_0},K_d) = \sum_{x \in X_d(\dot{w_0}a)} \psi_0(u(x))\psi_0(u'(x)), \quad a \in A,\]
corresponding to the triple $(\psi_0,\dot{w_0},K_d)$ (See Definition \ref{defn-sys-Kl}).  Here, we write
$X_d(\dot{w_0}a)$ for the Kloosterman set $X_{K_d}(\dot{w_0}a)$.

We choose $\ell > d$ such that for any $x \in X_d(\dot{w_0}a)$, we have $u(x) \in N_d\bs N_{-\ell}$ and $u'(x)  \in N_{-\ell}/N_d$.  
For any $x = [n_1\dot{w_0}an_2] \in X_d(\dot{w_0}a)$, denote by $\kappa_i(x) = (n_1)_{\alpha_i} \in  \fp^{-\ell}/ \fp^d$ and 
$\kappa_i'(x) = (n_2)_{\alpha_i} \in \fp^{-\ell}/\fp^d$. 

Consider the action of $A_d$ on $X_d(\dot{w_0}a)$
\[t \cdot [n_1\dot{w_0}an_2] = [ tn_1\dot{w_0}an_2 (\mathrm{Ad}(\dot{w_0})t)^{-1} ].\]
In particular, 
\[\kappa_i(t \cdot [x]) = \alpha_i(t) \kappa_i([x]), \quad \kappa_i'(t \cdot [x]) =(w_0\cdot\alpha_i)(t) \kappa_i'([x])=\alpha_{\sigma(i)}(t)^{-1} \kappa_i'([x]),\quad 1\leq i\leq m.\]
Here, $\sigma$ is the permutation on the $\{1,\cdots,m\}$ induced by the action of $w_0$ on $\Delta$.
\begin{lem}\label{simple root action is surjective}
    Denote by $C$ the cokernel of the following  map
    \[f\colon X_*(A)\to\BZ^m,\quad f(\lambda)=(\left<\alpha_i,\lambda\right>)_i.\]
    Assume that $p\nmid|C|$. Then for any $d>\frac{v(p)}{p-1}$, the map
    \[\phi_d\colon A_d\to(1+\fp^d)^m,\quad a\mapsto(\alpha_1(a),\cdots,\alpha_m(a))\]
    is surjective.
\end{lem}

\begin{remark}
    For the above cokernel $C$, if $G = \GL_n$, $|C| = 1$ and
    if $G = \Sp_{2n}$, then $|C| = 2.$
\end{remark}

\begin{proof}[Proof of Lemma \ref{simple root action is surjective}]
        For $d>\frac{v(p)}{p-1}$, the exponential map gives an isomorphism
        \[\exp\colon(\fp^d,+)\to(1+\fp^d,\times),\]
        and the logarithm map  $\log\colon1+\fp^d\to\fp^d$ is its inverse.
        
        We take a basis $\{\lambda_j\}_{j=1}^n$ of $X_*(A)$. Consider the following isomorphism
        \[X_*(A)\otimes\fp^d\stackrel{\sim}{\lra} A_d,\quad\sum_{j}\lambda_j\otimes t_j\mapsto\prod_j\lambda_j(\exp t_j).\]
        The composition of this map with the logarithm map is 
        \[f \otimes \id: X_*(A)\otimes\fp^d\lra (\fp^d)^m,\quad \sum_{j}\lambda_j\otimes t_j\mapsto \sum_j\left<\alpha_i,\lambda_j\right>t_j.\]
        Therefore, it suffices to show that $f\otimes \mathrm{id}$ is surjective. 
        
        Note that we have the following exact sequence
        \[X_*(A)\otimes\fp^d\xrightarrow{f\otimes\id}(\fp^d)^m\to C\otimes\fp^d\to 0.\]
        As we assume that $p \nmid |C|$, we have $C\otimes\fp^d=0$. 
        Therefore, $f \otimes \id$ is surjective and $\phi_d$ is
        surjective.

    \end{proof}

\begin{prop}\label{stevens}
    For $d > 0$ sufficiently large and $p \nmid |C|$ (See Lemma \ref{simple root action is surjective}), we have
    \[\Kl_d(a) = \frac{1}{q^{m(\ell-d)}} \sum_{x \in X_d(\dot{w_0}a)}  \prod_{i=1}^{m} \Kl(\varpi^\ell \kappa_i(x),\varpi^\ell
    \kappa_{\sigma(i)}'(x);d,\ell),\quad a\in A.\]
  \end{prop}
\begin{proof} 
   We have
   \[\Kl_d(a) = \sum_{x \in A_d \bs X_d(\dot{w_0}a)} \sum_{y \in A_d \cdot x} \prod_i \psi_F(\kappa_i(y))\prod_i \psi_F(\kappa_i'(y)).\]
   For each $x$, 
   \[\begin{aligned}
   \sum_{y \in A_d \cdot x} \prod_i \psi_F(\kappa_i(y))\prod_i \psi_F(\kappa_i'(y)) 
   &= \frac{1}{\Vol(\Stab(x))} \int_{A_d} \prod_i \psi_F(\kappa_i(t \cdot x))\prod_i \psi_F(\kappa_i'(t \cdot x)) dt \\
   &= \frac{1}{\Vol(\Stab(x))} \int_{A_d} \prod_i \psi_F\left(\alpha_i(t)\kappa_i(x) + \alpha_i(t)^{-1}\kappa_{\sigma(i)}'(x)\right) dt. \\
   \end{aligned}\]
   
   By our assumption on $d$ and $p$, the morphism
   \[A_d \lra \left( 1+\fp^d/1+\fp^\ell \right)^{m}, \quad t \mapsto (\alpha_i(t))_{1 \leq i \leq m}\]
   is surjective by Lemma \ref{simple root action is surjective}. Denote by $\Ker$ its kernel. The above integral over $A_d$ equals to
   \[\frac{\Vol(\Ker)}{\Vol(\Stab(x))} 
   \sum_{\lambda \in \left( 1+\fp^d/1+\fp^\ell \right)^{m}} \prod_i \psi_F\left(\lambda_i\kappa_i(x) + \lambda_i^{-1}\kappa_{\sigma(i)}'(x)\right).\]
   Denote by $N(x)$ the cardinality for the orbit $A_d \cdot x$. Note that
   \[\frac{\Vol(\Ker)}{\Vol(\Stab(x))} = \frac{1}{q^{m(\ell-d)}} \cdot \frac{\Vol(A_d)}{\Vol(\Stab(x))} = 
   \frac{|N(x)|}{q^{m(\ell-d)}}.\]
   We obtain that
  \[\Kl_d(a) = \frac{1}{q^{m(\ell-d)}} \sum_{x \in A_d \bs X_d(\dot{w_0}a)} |N(x)| \prod_{i=1}^{m} \Kl(\varpi^\ell \kappa_i(x),\varpi^\ell
    \kappa_{\sigma(i)}'(x);d,\ell).\]
   In the above, note that the Kloosterman sum $\Kl((\varpi^\ell \kappa_i(x),\varpi^\ell\kappa_{\sigma(i)}'(x);d,\ell)$ depends only on the $A_d$-orbit of
   $x$. We obtain the Proposition.
\end{proof}

Denote by
\[\nu: \fp^{-\ell}/\fp^d \lra [-\ell,d], \quad [x] \mapsto 
\begin{cases}
	v(x), \quad &[x] \not= 0; \\
	d, \quad &[x] = 0.
\end{cases}\]
Then for each $x \in X_d(\dot{w_0}a)$, by Proposition \ref{GL2 kloosterman sum}, for each $1 \leq i \leq m$, we have
\[|\Kl(\varpi^\ell \kappa_i(x),\varpi^\ell
\kappa_{\sigma(i)}'(x);d,\ell)| \leq q^{\ell+\frac{\nu\left(\kappa_{\sigma(i)}'(x)\right) + k}{2}}.\]
By Proposition \ref{stevens}, we obtain
\[|\Kl_d(a)| \leq q^{m(d+k/2)} \sum_{x \in X_d(\dot{w_0}a)}  \prod_{i=1}^{m} q^{\frac{\nu(\kappa_i'(x))}{2}}.\]

Consider the map
\[\kappa' = (\kappa_i')_{1 \leq i \leq m}: X_d(\dot{w_0}a) \lra (\fp^{-\ell}/\fp^d)^{m}\]
and its composition with $\nu$
\[X_d(\dot{w_0}a) \stackrel{\kappa'}{\lra} (\fp^{-\ell}/\fp^d)^{m} \stackrel{\nu}{\lra} [-\ell,d]^{m}.\]
For each $y \in [-\ell,d]^{m}$, denote by $X_d(\dot{w_0}a)_y$ the fibre
of $y$ under $\nu\circ \kappa'$. 

We have
\[
\begin{aligned}
|\Kl_d(a)| &\leq q^{m(d+k/2)} \sum_{x \in X_d(\dot{w_0}a)}  \prod_{i=1}^{m} q^{\frac{\nu(\kappa_i'(x))}{2}} \\
&= q^{m(d+k/2)} \sum_{y \in [-\ell,d]^{m}}  |X_d(\dot{w_0}a)_y| q^{\frac{\sum_{i=1}^{m} y_i}{2}}.
\end{aligned}
\]

We obtain the following result.

\begin{cor}\label{corollary of Stevens' method}
 For $d > 0$ sufficiently large and $p \nmid |C|$ (See Lemma \ref{simple root action is surjective}), we have
\[|\Kl_d(a)| \leq q^{m(d+k/2)} \sum_{y \in [-\ell,d]^{m}}  |X_d(\dot{w_0}a)_y| q^{\frac{\sum_{i=1}^{m} y_i}{2}} .\]
\end{cor}

\begin{remark}
    The assumption that $p \nmid |C|$ can be dropped. See 
    Proposition \ref{refined stevens in Sp_2n} for the case of $\Sp_{2n}$. 
\end{remark}

By this corollary, to obtain a nontrivial bound for the system of Kloosterman sums $\{\mathrm{Kl}_d\}_d$, it is enough to prove the existence of constants $\varepsilon>0$ and $C>0$ such that
\[|X_d(\dot{w_0}a)_y| q^{\frac{\sum_{i=1}^{m} y_i}{2}}\leq C|\delta^{-1/2+\varepsilon}(a)|\]
for any $a\in A$ with $\lambda_a \in \BZ_{\geq 0} \Phi^{\vee,+}$ and any $y\in[-\ell,d]^{m}$.

In the following, we shall count the fibre $X_d(\dot{w_0}a)_y$
for $G = \Sp_4$ via the Pl\"ucker coordinates for $N \bs G$ (See Lemma \ref{subdeterminant for symplectic group}). 

\subsection{Groups of type $C_2$}

In this section, we prove the regularity of Bessel distributions for type $C_2$ split reductive groups.

Let $G$ be a split reductive group over $F$ of type $C_2$, that is, the root system of $G$ is the irreducible root system of type $C_2$. As before, we fix a Borel subgroup $B =AN$ of $G$ and $\psi = \psi_0$ a generic character on $N$.

\begin{thm}\label{thm-type-C_2}
    For any irreducible admissible generic representation $\pi$ of $G$, the Bessel distribution $B_\pi$ is regular.
\end{thm}

    For the proof of Theorem \ref{thm-type-C_2}, the argument is similar to the proof of Theorem \ref{thm-type-A} for type $A$ groups. 
    
    Note that the root system of a Levi subgroup for a group of type $C_{n}$  can be decomposed into the following form
    \[A_{m_1}\sqcup A_{m_2}\sqcup\cdots\sqcup A_{m_s}\sqcup C_{r}, \quad m_1+\cdots+m_s+r = n.\]
    
    This reduces the problem to estimate the nontrivial
    bounds for Kloosterman sums on $\Sp_{2r}$, $r \leq n$
    and on $\SL_m$. As we have obtained the nontrivial bounds 
    for Kloosterman sums on $\SL_m$. It suffices to study
    Kloosterman sums on $\Sp_{2r}$ with $r \leq n$. 
    In particular, for $n=2$, we only need to consider
    Kloosterman sums for $\Sp_4$. 

    Now, we shall obtain the nontrivial bound for Kloosterman
    sums on $\Sp_4$. The main result is
    Proposition \ref{Sp_4 nontrivial bound} and this will
    finish the proof of Theorem \ref{thm-type-C_2}.
    
 Denote by 
\[G=\Sp_{4}=\{M\in\GL_{4}(F)\Big| M^TJM=J\},\quad J=\begin{pmatrix}
     & I_2\\
     -I_2 &
\end{pmatrix}.\]
The standard maximal torus and the standard unipotent subgroup of $G$ are given by
\begin{align*}
    &A=\left\{\begin{pmatrix}
    A_0 & \\
     & A_0^{-1}
    \end{pmatrix}\in G\Big| A_0\in\GL_2(F)\text{ diagonal}\right\}\\
    &N=\left\{\begin{pmatrix}
    N_0 & N_0S \\
     & (N_0^{-1})^t
    \end{pmatrix}\in G \Big| N_0\text{ is upper triangular, unipotent and }S\text{ is symmetric}\right\}
\end{align*}

A set of simple roots of $G$ with respect to the maximal torus $A$ is given by $\Delta=\{\alpha,\beta\}$, where
\[\alpha(\diag(y_1,y_2,y_1^{-1},y_2^{-1}))=y_1y_2^{-1},\quad \beta(\diag(y_1,y_2,y_1^{-1},y_2^{-1}))=y_2^2.\]
Then $\Phi^+=\{\alpha,\beta,\alpha+\beta,2\alpha+\beta\}$ is a set of positive roots. We denote by $s_\alpha$ and $s_\beta$ the simple reflections corresponding to $\alpha$ and $\beta$ respectively. We fix a family of isomorphisms $x_{\gamma},\gamma\in\Phi$: for $t\in F$,
\begin{align*}
    &x_{\alpha}(t)=I_4+t(E_{1,2}-E_{4,3}),\quad x_{-\alpha}(t)=I_4+t(E_{2,1}-E_{3,4})\\
    &x_\beta(t)=I_4+tE_{2,4},\quad x_{-\beta}(t)=I_4+tE_{4,2}\\
    &x_{\alpha+\beta}(t)=I_4+t(E_{1,4}+E_{2,3}),\quad x_{-(\alpha+\beta)}(t)=I_4+t(E_{4,1}+E_{3,2})\\
    &x_{2\alpha+\beta}(t)=I_4+tE_{1,3},\quad x_{-(2\alpha+\beta)}(t)=I_4+tE_{3,1}
\end{align*}
with $E_{i,j}$ the standard elementary matrix with a $1$ in the $(i,j)$–position and $0$ everywhere else. We can define the generic character $\psi_0$ with respect to $\{x_\gamma\}_{\gamma\in\Phi}$:
\[\psi_0\colon N\to\BC^\times,n=\prod_{\gamma\in\Phi^+}x_\gamma(n_\alpha)\mapsto\psi_F\left(\sum_{\gamma\in\Delta}n_\alpha\right).\]
The matrix form of $\psi_0$ is the following
\[\psi_0\begin{pmatrix}
     1 & x_1 & * & * \\
       & 1 & * & x_2\\
       & & 1 & \\
       & & -x_1 & 1
\end{pmatrix}=\psi_F(x_1+x_2).\]
The Tits representative $w^0$ of longest Weyl element $w_0=s_\alpha s_\beta s_\alpha s_\beta\in W$ is 
\[w^0=\begin{pmatrix}
      &  & -1 & \\
      & & & -1\\
      1 & & &\\
      & 1 &  &
\end{pmatrix}.\]
Note that $(w^0)^2=-I_4$.

Now we consider the filtration $\{K_d\}_{d\geq 1}$ in Example \ref{special filtration}. For $d>0$, the subgroup $A_d$ is
\[A_d=\{\diag(a_1,a_2,a_1^{-1},a_2^{-1}):a_1,a_2\in1+\fp^d\}\]
and $N_{\gamma}=x_\gamma(\fp^d)$ for $\gamma\in\Phi$. Then $K_d$ is the principal congruence subgroup of maximal compact open subgroup of $K_0=K=\Sp_4(\CO)$.

The triple $(\psi_0,w^0,\{K_d\}_d)$ gives a family of Kloosterman sums on $G$
\[\mathrm{Kl}_d(a)=\sum_{x\in X_d(w^0a)}\psi_0(u(x))\cdot\psi_0(u'(x)),\quad a\in A.\]

Recall that for each $a\in A$, there exists a unique $\lambda_{a}\in X_*(A)$ such that 
\[a=a_0\lambda_{a}(\varpi), \quad a_0 \in A\cap K.\]

Now we state the main result in this subsection, which concerns the non-trivial bound for local $\Sp_4$ Kloosterman sum. By the result of D\k{a}browski-Reeder \cite{DR98}, we only need to consider the element $a\in A$ such that $\lambda_{a}\in\BZ_{\geq 0}\Phi^{\vee,+}$, which is in the support of Kloosterman sum.

For a $a\in A$ such that $\lambda_{a} \in \BZ_{\geq 0}\Phi^{\vee,+}$, we can assume
\[a=\diag(\varpi^{r}v_1,\varpi^{s-r}v_2,\varpi^{-r}v_1^{-1},\varpi^{r-s}v_2^{-1}),\]
where $r,s\geq 0$ and $v_1,v_2\in\CO^\times$.
\begin{prop}\label{Sp_4 nontrivial bound}
    For all
    \[a=\diag(\varpi^{r}v_1,\varpi^{s-r}v_2,\varpi^{-r}v_1^{-1},\varpi^{r-s}v_2^{-1})\in A,\]
    where $r,s\geq d$, $v_1,v_2\in\CO^\times$, we have
    \[|\operatorname{Kl}_d(a)|\leq C(d,k,q)\cdot(\ell+d+1)^3\cdot q^{\frac{11}{12}(r+s)},\]
    where $k$ is a constant such that $q^{-k}\leq|2|^2$, $\ell=\max\{r,s\}$ and $C(d,k,q)$ is a constant depend only on $d,k,q$.
\end{prop}

\begin{prop}\label{refined stevens in Sp_2n}
    Suppose $G=\Sp_{2n}$ and $\alpha_i,1\leq i\leq n$ are simple roots of $\Sp_{2n}$:
    \begin{align*}
        &\alpha_i(a_1,\cdots,a_n,a_1^{-1},\cdots,a_n^{-1})=\frac{a_i}{a_{i+1}},1\leq i\leq n-1,\\
        &\alpha_n(a_1,\cdots,a_n,a_1^{-1},\cdots,a_n^{-1})=a_n^2.
    \end{align*}
    For $d > 0$ sufficiently large, we have
    \[|\Kl_d(a)| \leq q^{n(d+k/2)+v(2)} \sum_{y \in [-\ell,d]^{n}}  |X_d(\dot{w_0}a)_y| q^{\frac{\sum_{i=1}^{n} y_i}{2}}.\]
    \begin{proof}
        For $G = \Sp_{2n}$ and general $p$ (especially for $p=2$), Proposition \ref{stevens} becomes
        \[\Kl_d(a) = \frac{1}{q^{n(\ell-d)-v(2)}} \sum_{x \in X_d(w^0a)} \prod_{i=1}^{n-1} \Kl(\varpi^\ell \kappa_i(x),\varpi^\ell \kappa_{i}'(x);d,\ell)\cdot \Kl(\varpi^\ell \kappa_n(x),\varpi^\ell \kappa_n'(x);d+v(2),\ell).\]
        As in the proof of Corollary \ref{corollary of Stevens' method}, applying the Weil bound for  Kloosterman sums on $\GL_2$ (Proposition \ref{GL2 kloosterman sum}) yields
        \[|\Kl_d(a)| \leq q^{n(d+k/2)+v(2)} \sum_{y \in [-\ell,d]^{n}}  |X_d(\dot{w_0}a)_y| q^{\frac{\sum_{i=1}^{n} y_i}{2}}.\]
    \end{proof}
\end{prop}
Hence by the corollary \ref{corollary of Stevens' method} and Proposition \ref{refined stevens in Sp_2n}, to obtain a non-trivial bound for the Kloosterman sum it suffices to show that there exist constants $0<\delta<1$ and $C>0$ such that such that for every $a=\diag(\varpi^{r}v_1,\varpi^{s-r}v_2,\varpi^{-r}v_1^{-1},\varpi^{r-s}v_2^{-1})\in A$ and $y\in[-\ell,d]^{2}$,
\[|X_d(w^0a)_y| q^{\frac{y_1+y_2}{2}}\leq C\cdot q^{\delta(r+s)}.\]

Bounds for local $\Sp_4$ Kloosterman sums in the case $F=\BQ_p$, $d=0$ were given by Man \cite{Man22}. In \cite{Man22}, Man expresses the Kloosterman sums via the coset
representatives for the Kloosterman set $X(w^0a)$, corresponding to the case $d=0$. This description is given via Pl\"ucker coordinates, which provide an explicit parametrization of the Bruhat decomposition of $G=\Sp_4$. In the following, we present a different approach, based on Stevens' approach, to treat the case $d>0$ and general $p$-adic fields.

For $g=(g_{ij})_ {1\leq i,j\leq 4}\in G=\Sp_4(F)$, we define the Pl\"ucker coordinates of $g$:
\begin{align*}
    v_{i}&=g_{3i},\quad 1\leq i\leq 4;\\
    v_{ij}&=g_{3i}g_{4j}-g_{3j}g_{4i},\quad 1\leq i<j\leq 4.
\end{align*}
We also have the following relations:
\begin{align}
    v_iv_{jk}-v_jv_{ik}+v_kv_{ij}&=0,\quad 1\leq i<j<k\leq 4,\\
    v_{13}+v_{24}&=0
\end{align}
Define
\[V=\{v=(v_1,\cdots,v_4,v_{12},\cdots,v_{34})\in F^{10}:v\text{ satisfies (5.1) and (5.2)}\}.\]

Now we determine which element in $N/N(\CO)$ lies in $u'(X_d(w^0a))$. Given the large number of Pl\"ucker coordinates, we will not write down a complete set of Pl\"ucker relations characterizing the set $N\backslash\ K_d$. For our purposes, we only use the following result:
\begin{lem}[Man, Proposition 3.1 (1) in \cite{Man24}]\label{subdeterminant for symplectic group}
    The Pl\"ucker coordinates defined above gives a bijection
    \[N\backslash\Sp_4\xrightarrow{\sim} V\backslash\{0\}.\]
\end{lem}
\begin{remark}
    By \cite[Proposition 3.1]{Man24}, this bijection restricts to a bijection between $N\backslash \Sp_4(\CO)$ and a subset $V'\subset V\setminus{0}$ consisting of integral elements satisfying certain coprimality conditions. Man’s work focuses on the case $d=0$ and $F=\BQ_p$, and he expresses the coset representatives for Kloosterman set in terms of these Pl\"ucker coordinates.
\end{remark}
For a fixed element $[g_0]:=[n_1\cdot w^0a\cdot n_2]\in X_d(w^0a)$, where $n_1\in N_d\backslash N$ and $n_2\in N/N_d$. From the definition of the set $X_d(w^0a)$, we can assume that the element
\[n_2=\begin{pmatrix}
    1 & x & u &w+xy\\
     & 1 & w &y\\
     & &1 &\\
     & & -x &  1\\
\end{pmatrix}\in N/N_d.\]
Here $v(x),v(y),v(u),v(w)$ all integers and they satisfy $\leq d$. For given $x,y,u,w$ , we also use $n_{x,y}^{u,w}$ to denote the corresponding element $n_2$ in $N/N_d$.

Note that $w^0a\cdot u_2=u_1^{-1}\cdot g_0$. By direct computation, we have
\[w^0a\cdot u_2=\begin{pmatrix}
    0 & 0 & -\varpi^{-r}v_1^{-1} & 0\\
    0 & 0 & \varpi^{r-s}v_2^{-1}x & -\varpi^{r-s}v_2^{-1}\\
    \varpi^{r}v_1 & \varpi^{r}v_1x & \varpi^{r}v_1u & \varpi^{r}v_1(w+xy)\\
    0 & \varpi^{s-r}v_2 & \varpi^{s-r}v_2w & \varpi^{s-r}v_2y
\end{pmatrix}\]

By Lemma \ref{subdeterminant for symplectic group}, the Pl\"ucker coordinates $v_i,1\leq i\leq 4$ and $ v_{ij},1\leq i<j\leq 4$ of $w^0a\cdot u_2$ coincide with those of $g_0$. Since $g_0\in K_d$, the following conditions (1)-(7) hold:
\begin{enumerate}\label{conditions from plucker coordinates}
    \item $\varpi^{r}v_1x\in\fp^d$, i.e., $v(x)+r\geq d$.
    \item $\varpi^{r}v_1u\in1+\fp^d$, Therefore, we have $v(u)=-r$.
    \item $\varpi^{r}v_1(w+xy)\in\fp^d$. Therefore, we have $-(v(x)+v(y))\leq\max(-v(w),r)$.
    \item $v_{34}=\varpi^s(uy-w^2-wxy)\in1+\fp^d$.
    \item $v_{13}=-v_{24}=\varpi^{s}w\in\fp^d$, i.e., $v(w)+s\geq d$.
    \item $v_{14}=\varpi^{s}y\in\fp^d$, i.e., $v(y)+s\geq d$.
    \item $v_{23}=(\varpi^{s})(xw-u)\in\fp^d$. Therefore, we have $-(v(x)+v(w))\leq\max(-v(u),s)$.
\end{enumerate}

Now for fixed $y\in[-\ell,d]^2$, we estimate $q^{\frac{1}{2}(y_1+y_2)}\cdot|X_d(w^0a)_y|$. Let $R\subset\CO$ be a system of representatives for $\CO/\fp$ such that $0\in R$. Note that every $n\in N/N_d$ can be written as the following form
\[n=\begin{pmatrix}
    1 & x & u &w+xy\\
     & 1 & w & y\\
     & &1 &\\
     & & -x &  1\\
\end{pmatrix},\]
where
\begin{align*}
    &x\in\{j_0\varpi^{v(x)}+j_1\varpi^{v(x)+1}+j_2\varpi^{v(x)+2}+\cdots+j_{d-1-v(x)}\varpi^{d-1}:j_i\in R\}\\
    &y\in\{j_0\varpi^{v(y)}+j_1\varpi^{v(y)+1}+j_2\varpi^{v(y)+2}+\cdots+j_{d-1-v(y)}\varpi^{d-1}:j_i\in R\}\\
    &u\in\{j_0\varpi^{v(u)}+j_1\varpi^{v(u)+1}+j_2\varpi^{v(u)+2}+\cdots+j_{d-1-v(u)}\varpi^{d-1}:j_i\in R\}\\
    &w\in\{j_0\varpi^{v(w)}+j_1\varpi^{v(w)+1}+j_2\varpi^{v(w)+2}+\cdots+j_{d-1-v(w)}\varpi^{d-1}:j_i\in R\}
\end{align*}
and $v(x),v(y),v(w),v(u)\leq d$. If $x=[n_1w^0an_2]\in X_d(w^0a)_y$, then $[n_2]\in u'(X_d(w^0a))$ and we can assume
\[n_2=\begin{pmatrix}
    1 & x & u &w+xy\\
     & 1 & w &y\\
     & &1 &\\
     & & -x &  1\\
\end{pmatrix}\]
has the above form, where $v(x)=y_1,v(y)=y_2$ and $v(x),v(y),v(u),v(w)$ all integers and they satisfy $\leq d$. Since $n_2\in u'(X_d(w^0a))$, the matrix $n_2$ satisfies the conditions (1)-(7).

\begin{lem}\label{estimate on kloosterman set}
    Let $\ell=\max(r,s)$. Then
    \[q^{\frac{1}{2}(y_1+y_2)}\cdot|X_d(w^0a)_y|\leq q^{5d}\cdot(\ell+d+1)\cdot q^{\frac{11}{12}(r+s)}.\]
    \begin{proof}
        Case 1: $2v(w)<-s$ and $v(wxy)\geq -s$. By property (4) we have $2v(w)=v(uy)=-r+v(y)$, i.e., we have $-r=2v(w)-v(y)$. This implies
        \[-\frac{v(x)+v(y)}{2}-v(w)+r=-\frac{1}{2}v(x)-\frac{2}{3}v(y)-\frac{2}{3}v(w)+\frac{7}{6}r.\]
        If $s-r\geq d$, then $s>r$ and
        \begin{align*}
            -\frac{1}{2}v(x)-\frac{2}{3}v(y)-\frac{2}{3}v(w)+\frac{7}{6}r
            =\frac{1}{6}v(x)-\frac{2}{3}v(xyw)+\frac{7}{6}r
            \leq\frac{1}{6}d+\frac{2}{3}s+\frac{7}{6}r
            \leq d+\frac{11}{12}(r+s)
        \end{align*}
        Hence
        \begin{align*}
            q^{\frac{v(x)+v(y)}{2}}\cdot|X_d(w^0a)_y|\leq q^{-\frac{v(x)+v(y)}{2}-v(w)+r+4d}=q^{5d+\frac{11}{12}(r+s)}.
        \end{align*}
        If $s-r<d$, then by property (7), we have $v(x)+v(w)=-r$, then
        \begin{align*}
            &-\frac{1}{2}v(x)-\frac{2}{3}v(y)-\frac{2}{3}v(w)+\frac{7}{6}r
            =\frac{1}{4}v(y)-\frac{11}{12}v(xyw)+\frac{11}{12}r+\frac{1}{6}v(x)
            \leq\frac{11}{12}(r+s)+d.
        \end{align*}
        Hence
        \begin{align*}
            q^{\frac{v(x)+v(y)}{2}}\cdot|X_d(w^0a)_y|\leq q^{5d+\frac{11}{12}(r+s)}.
        \end{align*}
        
        Case 2: $2v(w)<-s$ and $v(wxy)<-s$. Introduce the change of variable $z:=u-xw$. Under this substitution, property (2) becomes $v(z+xw)=-r$, property (4) becomes $\varpi^s(yz-w^2)\in1+\fp^d$ and property (7) becomes $v(z)\geq d-s$. Henceforth, we fix $(x,w)$. We have $\#(u)=\#(z)$ and $yz\in(w^2+\varpi^{-s})+\fp^{d-s}$, then $\#(z)\leq q^{d-v(z)}$. For fixed $z$, we have $y\in z(w^2+\varpi^{-s})+\fp^{d-s-v(z)}$ and hence $\#(y)\leq q^{s+v(y)}$. This implies $\#(y,z)\leq q^{s+d}$. Now we suppose $v(xw)\geq -r$, then $v(z)\geq -r$. We have
        \[\frac{v(x)+v(y)}{2}-v(x)-v(w)+s=\frac{1}{2}v(y)-\frac{1}{2}v(x)-v(w)+s.\]
        Use the condition $2v(w)=v(y)+v(z)$, we have
        \begin{align*}
            s+\frac{1}{2}v(y)-\frac{1}{2}v(x)-v(w)
            =s-\frac{1}{2}v(xw)-\frac{3}{8}v(z)+\frac{1}{4}v(w)+\frac{1}{8}v(y)
            \leq \frac{7}{8}(r+s)+d.
        \end{align*}
        Hence
        \begin{align*}
            q^{\frac{v(x)+v(y)}{2}}\cdot|X_d(w^0a)_y|\leq q^{\frac{v(x)+v(y)}{2}-v(x)-v(w)+s+3d}=q^{\frac{7}{8}(r+s)+4d}.
        \end{align*}
        If $v(xw)<-r$, then we have $v(z)=v(xw)$. Since $2v(w)=v(y)+v(z)$, we get $v(w)=v(x)+v(y)$. We fix $(x,z)$, then we have $\#(x,z)\leq q^{-v(x)-v(z)+2d}$. We have
        \begin{align*}
            w\in(-ux^{-1}+\varpi^{-r}x^{-1})+\fp^{d-r-v(x)},\quad y\in(w^2z^{-1}+\varpi^{-s}z^{-1})+\fp^{d-s-v(z)}
        \end{align*}
        and $\#(w,y)\leq q^{r+s+v(x)+v(z)}$. So we have $\#(x,z,w,y)\leq q^{r+s+2d}$. We estimate
        \begin{align*}
            \frac{v(x)+v(y)}{2}+r+s\leq\frac{1}{8}v(xw)+\frac{1}{4}v(w)+\frac{1}{8}v(y)+r+s\leq\frac{7}{8}(r+s)+d.
        \end{align*}
        Hence
        \begin{align*}
            q^{\frac{v(x)+v(y)}{2}}\cdot|X_d(w^0a)_y|\leq q^{\frac{v(x)+v(y)}{2}+r+s+2d}=q^{\frac{7}{8}(r+s)+3d}.
        \end{align*}
        Case 3: $2v(w)\geq -s$. For $r>s$. By property (7), we must have $v(x)+v(w)=-r$. If $-v(y)+r\leq s$. We get $r\leq s+v(y),-v(w)=r+v(x)\leq s/2$. Then we have
        \[\#(x,y,u,w)\leq q^{-(v(x)+v(y)+v(w))+r+4d}\leq q^{-(v(x)+v(y))+s+v(y)+r+v(x)+4d}\]
        and hence
        \begin{align*}
            q^{\frac{v(x)+v(y)}{2}}\cdot|X_d(w^0a)_y|\leq q^{r+s+\frac{v(x)+v(y)}{2}+4d}\leq q^{r+s+\frac{v(x)}{2}+d+4d},
        \end{align*}
        here we use the fact $v(y)\leq d$. Since $v(x)\leq s/2-r$, we have $r+s+v(x)/2\leq r/2+5s/4$ and
        \[q^{\frac{v(x)+v(y)}{2}}\cdot|X_d(w^0a)_y|\leq(s+1)q^{\frac{1}{2}r+\frac{5}{4}s+5d}.\]
        If $-v(y)+r>s$. By property (4), we have $\#(u)\leq q^{s+v(y)}$. Hence we still have
        \[\#(x,y,u,w)\leq q^{-(v(x)+v(y))+s+v(y)+r+v(x)+4d}.\]
        and
        \[q^{\frac{v(x)+v(y)}{2}}\cdot|X_d(w^0a)_y|\leq q^{-\frac{v(x)+v(y)}{2}+s+v(y)+r+v(x)+4d}\leq q^{\frac{1}{2}r+\frac{5}{4}s+5d}\leq q^{5d}\cdot(s+d+1)\cdot q^{\frac{7}{8}(r+s)}.\]

        For $r\leq s$. If $-v(y)+r>s$, then we have $v(y)-r=v(w)+v(x)+v(y)$, this implies $v(x)+v(w)=-r$. By property (4), we have $\#(u)\leq q^{s+v(y)}$. So we have
        \[\#(x,y,u,w)\leq q^{-v(x)+s-v(w)+4d}\]
        and hence
        \[q^{\frac{v(x)+v(y)}{2}}\cdot|X_d(w^0a)_y|\leq q^{-\frac{v(x)}{2}+s+\frac{v(y)}{2}-v(w)+4d}\leq q^{-\frac{v(x)}{2}+s+\frac{1}{2}r-\frac{1}{2}s-v(w)+4d}\leq q^{r+\frac{3}{4}s+4d}.\]
        If $-v(y)+r\leq s$, then $-(v(x)+v(y)+v(w))\leq s$, so we have $-(v(x)+v(y))/2\leq s/2+v(w)/2$. This gives
        \begin{align*}
            q^{\frac{y_1+y_2}{2}}\cdot|X_d(w^0a)_y|&\leq(s+d+1)q^{-\frac{v(x)+v(y)}{2}+r-v(w)+4d}
            \leq (s+d+1)q^{\frac{1}{2}s+\frac{1}{2}v(w)+r-v(w)+4d}\\
            &\leq(s+d+1)q^{r+\frac{3}{4}s+4d}\leq q^{5d}\cdot(s+d+1)\cdot\cdot q^{\frac{7}{8}(r+s)}
        \end{align*}
    \end{proof}
\end{lem}
Now we can prove the Proposition \ref{Sp_4 nontrivial bound}.
\begin{proof}[Proof of Proposition \ref{Sp_4 nontrivial bound}]
    Using Corollary \ref{corollary of Stevens' method} and Proposition \ref{refined stevens in Sp_2n}, we have
    \begin{align*}
        |\Kl_d(a)| &\leq q^{2(d+k/2)+v(2)}\sum_{y \in [-\ell,d]^{2}}  |X_d(\dot{w_0}a)_y| q^{\frac{y_1+y_2}{2}}.
    \end{align*}
    By Lemma \ref{estimate on kloosterman set}, for $y\in[-\ell,d]^2$, we have 
    \[q^{\frac{1}{2}(y_1+y_2)}\cdot|X_d(w^0a)_y|\leq q^{5d}\cdot(\ell+d+1)\cdot q^{\frac{11}{12}(r+s)}.\]
    Therefore, we have
    \begin{align*}
        |\Kl_d(a)|\leq q^{7d+k+v(2)}\cdot(\ell+d+1)^3\cdot q^{\frac{11}{12}(r+s)}
    \end{align*}
\end{proof}


\begin{thebibliography}{XX}
\bibitem[AV16]{AV16} J. Adams, D. Vogan, {\it Contragredient representations and characterizing the local Langlands correspondence.} Amer. J. Math., 138(3):657–682, 2016.
\bibitem[AGK15]{AGK15} A. Aizenbud, D. Gourevitch and A. Komarsky, {\it Vanishing of certain equivariant distributions on $p$-adic spherical spaces, and nonvanishing of spherical Bessel functions}. International Mathematics Research Notices, (2015), no.18, 8471-8483.
\bibitem[AGS15]{AGS15} A. Aizenbud, D. Gourevitch and E. Sayag, {\it $z$-finite distributions on $p$-adic groups}. Advances in Mathematics 285 (2015), 1376-1414.
\bibitem[Ba97]{Ba97} E. M. Baruch, {\it On Bessel distributions of $GL(2)$ over a $p$-adic field.} J. Number Theory 67 (1997), 190-202.
\bibitem[Ba01]{Ba01} E. M. Baruch, {\it On Bessel distributions for quasi-split groups.} Transactions American Math. Society, 353, no.7, (2001), 2601-2614.
\bibitem[Ba04]{Ba04} E. M. Baruch, {\it Bessel distribution for $GL(3)$ over the $p$-adic field.} Pac. J. Math. 217 (1) (2004), 11-27.
\bibitem[BB05]{BB05} A. Bj\"orner and F. Brenti. {\it Combinatorics of Coxeter Groups}, Graduate Texts in Mathematics, 231, Springer-Verlag, New York, 2005.
\bibitem[BM24]{BM24} V. Blomer and S. Man, {\it Bounds for Kloosterman sums on GL(n)}.  Math. Ann. 390 (2024), 1171-1200.
\bibitem[Bum13]{Bum13} D. Bump, {\it Lie Groups}, volume 225 of Graduate Texts in Mathematics. Springer, New York, second edition, 2013.
\bibitem[Chai19]{Chai19} J. Chai, {\it On Bessel functions over $p$-adic fields.} International Mathematics Research Notices, (2019), no.3, 673-699.
\bibitem[CST17]{CST17} J. W. Cogdell, F. Shahidi, and T-L. Tsai. {\it Local Langlands correspondence for $\GL_n$ and the exterior and symmetric square $\varepsilon$-factors}.
Duke Math. J., 166(11):2053–2132, 2017.
\bibitem[DR98]{DR98} R. D\k{a}browski, M. Reeder, {\it Kloosterman sets in reductive groups}. Journal of Number Theory 73 (1998), no.2, 228-255.
\bibitem[Fri87]{Fri87} S. Friedberg, {\em Poincar\'e Series for GL(n): Fourier Expansion, Kloosterman Sums, and Algebreo-
Geometric Estimates.}, Math. Z. 196 (1987), 165-188.
\bibitem[Del84]{Del84} P. Deligne. {\it Le “centre” de Bernstein.} In Representations of reductive groups over a local field, Travaux en Cours, pages 1–32. Hermann, Paris, 1984. Edited by P. Deligne.
\bibitem[FLO12]{FLO12} B. Feigon, E. Lapid and O. Offen, On representations distinguished by
unitary groups. Publ. Math. Inst. Hautes Etudes Sci. 115 (2012), 185-323.
\bibitem[GSW21]{GSW21} D. Goldfeld, E. Stade, and M. Woodbury, {\it An orthogonality relation for $\GL(4,\BR)$ (with an appendix by Bingrong Huang)}. arXiv: 1910.13586. Forum of Mathematics, Sigma (2021), 1-83.
\bibitem[Guo]{Guo} J. Guo. {\it Spherical characters on certain $p$-adic symmetric spaces.} preprint, available at \url{https://archive.mpim-bonn.mpg.de/id/eprint/525/}.
\bibitem[Ha]{Ha} J. Hakim. {\it Admissible distributions on $p$-adic symmetric spaces.} J. Reine Angew. Math. 455 (1994), 1-19.
\bibitem[HC70]{HC70} Harish-Chandra, {\it Harmonic analysis on reductive $p$-adic groups.} Lecture Notes in Mathematics, 162, Springer-Verlag, Berlin, 1970. Notes by G. van Dijk.
\bibitem[HC99]{HC99} Harish-Chandra, {\it Admissible invariant distributions on reductive $p$-adic groups.} University Lecture Series, vol. 16, American Mathematical Society, Providence, RI. 1999, Preface and notes by Stephen DaBacker and Paul J. Sally, Jr.
\bibitem[H80]{H80} J. E. Humphreys, {\it Linear Algebraic Groups}, Graduate Texts in Mathematics, 021, Springer-Verlag, 1980. 
\bibitem[Jac16]{Jac16} H. Jacquet, {\it Germs for Kloosterman integrals, a review}. Contemp. Math. Vol. 664 (2016), 182-195.
\bibitem[JY96]{JY96} H. Jacquet, Y. Ye, {\it Distinguished representation and quadratic base change for $\GL(3)$}. Transactions American Math. Society, 348, no.3, (1996), 913-939.
\bibitem[Jan03]{Jan03} J. C. Jantzen {\it Representations of algebraic groups}. Mathematical Surveys and Monographs, Vol. 107. American Mathematical Society, Providence, RI, second edition.
\bibitem[JY99]{JY99} H. Jacquet, Y. Ye, {\it Germs of Kloosterman integrals for $\GL(3)$}. Transactions American Math. Society, 351, no.3, (1999), 1227-1255.
\bibitem[Kal22]{Kal22} T. Kaletha, {\it On the local Langlands conjectures for disconnected groups}, 2022, arXiv:2210.02519.
\bibitem[LM13]{LM13} E. Lapid, Z. Mao, {\it Stability of certain oscillatory integrals.} International Mathematics Research Notices, (2013), no.3, 525-547.
\bibitem[Lin26]{Lin26} J. Linn, {\it Bounds for Kloosterman Sums for $\GL_n$}, Adv. in Math., vol.499, 2026.
\bibitem[Man22]{Man22} S. Man, {\it Symplectic Kloosterman sums and Poincare series.} Ramanujan J. 57 (2022), 707-753.
\bibitem[Man24]{Man24} S. Man. Fourier coefficients of Sp(4) Eisenstein series. Acta Arith., 213(3):227–271, 2024.
\bibitem[Mia24]{Mia24} Xinchen Miao. {\it Bessel functions and Kloosterman integrals on $GL(n)$.} Journal of Functional Analysis. 286 (2024), 110267
\bibitem[Mil17]{Mil17} James S. Milne, {\it Algebraic groups: The theory of group schemes of finite type over a field}, Cambridge Studies in Advanced Mathematics, vol. 170, Cambridge
University Press, Cambridge, 2017.
\bibitem[RR]{RR} C. Rader and S. Rallis. {\it Spherical characters on $p$-adic symmetric spaces.} Amer. Jour. Math., 118 (1996), 91-178.
\bibitem[Spr98]{Spr98} T. A. Springer, {\it Linear algebraic groups}, 2nd ed., Progress in Mathematics, vol. 9, Birkh\"auser Boston, Inc., Boston, MA, 1998, DOI 10.1007/978-0-8176-4840-4. MR1642713
\bibitem[Ste87]{Ste87} G. Stevens, {\it Poincar\'e series on $\GL(r)$ and Kloostermann sums}. Math. Annalen 277 (1987), 25-51.
\bibitem[Ste68]{Ste68} R. Steinberg, {\it Lectures on Chevalley Groups}. New Haven, CT: Yale University, 1968. Notes prepared by John Faulkner and Robert Wilson.
\bibitem[SV17]{SV17} Y. Sakellaridis and A. Venkatesh, {\it Periods and Harmonic Analysis on Spherical Varieties}. Astérisque 396, Société Mathématique de France, Paris, 2017.
\end{thebibliography}
\end{document}